\documentclass[a4paper,11pt,reqno]{amsart}

\usepackage{amssymb}
\usepackage[all]{xy}

\usepackage{amsmath}
\usepackage{amsthm}    
\usepackage{mathrsfs}
\usepackage{xcolor}
\usepackage{enumitem}
\usepackage{subcaption}
\usepackage[final]{graphicx}
\usepackage{float}
\usepackage{epic}
\usepackage{setspace}

\usepackage[all]{xy}
\usepackage{color}
\usepackage{calligra}
\usepackage[T1]{fontenc}
\usepackage{verbatim}

\newcount\mycount
\numberwithin{equation}{section}

\numberwithin{equation}{section}
\usepackage[utf8]{inputenc}
\usepackage[english]{babel}
\usepackage{stmaryrd,mathtools,commath,mathrsfs}
\usepackage{bbm} 
\usepackage{amssymb,amsthm}
\numberwithin{equation}{section}
\usepackage{float}
\usepackage{graphicx}
\usepackage{yfonts,verbatim,calrsfs,enumitem}
\usepackage{comment}
\usepackage{todonotes}

\usepackage{amsmath}

\usepackage{xcolor}
\definecolor{cornellred}{rgb}{0.7,0.11,0.11}

\usepackage{cite}
\usepackage[colorlinks=false, allcolors=cornellred, pdfpagelabels]{hyperref}

\hypersetup{
    bookmarksdepth=2,
    bookmarksopen=true
}

\allowdisplaybreaks

\theoremstyle{plain}
\newtheorem{theorem}{Theorem}[section]
\newtheorem{lemma}[theorem]{Lemma}
\newtheorem{proposition}[theorem]{Proposition}
\newtheorem{corollary}[theorem]{Corollary}

\theoremstyle{definition}
\newtheorem{definition}[theorem]{Definition}
\newtheorem*{definition*}{Definition}

\newtheorem{remark}[theorem]{Remark}

\DeclareMathAlphabet{\mathcal}{OMS}{cmsy}{m}{n}

\newcommand{\Cs}{\ensuremath{\mathrm{C}^\ast}}

\newcommand{\id}{\ensuremath{\mathrm{id}}}
\renewcommand{\ker}{\ensuremath{\mathrm{ker}}}
\newcommand{\im}{\ensuremath{\mathrm{im}}}

\newcommand{\rank}{\ensuremath{\mathrm{rank}}}

\newcommand{\R}{\ensuremath{\mathbb{R}}}

\newcommand{\calU}{\ensuremath{\mathcal{U}}}

\providecommand{\cal}{}
\renewcommand{\cal}[1]{\ensuremath{\mathcal{#1}}}

\newcommand{\leb}{\ensuremath{\tau_{\mathrm{Leb}}}}
\newcommand{\tr}{\ensuremath{\mathrm{tr}}}

\usepackage{tikz}
\usetikzlibrary{arrows,matrix,positioning,automata}
\usepackage{quiver}
\usepackage{pgfplots}

\setlist[enumerate,itemize]{noitemsep} 

\hypersetup{
    bookmarksdepth=2,
    bookmarksopen=true
}

\begin{document}

\title[A new characterisation of selflessness]{A new characterisation of selflessness}

\author{Mario Klisse} \address{Department of Mathematics, Christian-Albrechts-Universität Kiel, Heinrich-Hecht-Platz  6, \newline 24118 Kiel, Germany} \curraddr{} \email{klisse@math.uni-kiel.de} \thanks{}

\author{Matteo Pagliero} \address{Department of Mathematics and Computer Science, University of Southern Denmark, Campusvej 55, 5230 Odense, Denmark} \curraddr{} \email{pagliero@imada.sdu.dk} \thanks{}



\date{\today}

\dedicatory{}

\begin{abstract}
We characterise the existence of expectation-preserving embeddings of reduced amalgamated free product $\mathrm{C}^\ast$-algebras into norm ultrapowers of the first factor, fixing its diagonal copy pointwise, in terms of complementary compression relations on Hilbert space ultrapowers. This yields new characterisations of Robert's notion of selflessness, selfless inclusions, and relative selflessness with respect to conditional expectations. We also characterise Ozawa's notion of complete selflessness, including its extension to inclusions. We then illustrate the usefulness of these characterisations by establishing complete selflessness for suitable reduced (amalgamated) free products, free Araki--Woods $\mathrm{C}^\ast$-algebras, certain Hecke $\mathrm{C}^\ast$-algebras, and $q$-Gaussian $\mathrm{C}^\ast$-algebras. Finally, we obtain dynamical characterisations of (complete) $\mathrm{C}^\ast$-selflessness of discrete groups and establish relative selflessness for suitable reduced crossed products.
\end{abstract}

\maketitle

\tableofcontents

\section*{Introduction}
\renewcommand\thetheorem{\Alph{theorem}}


\subsection*{Context and motivation}
In free probability, pioneered by Voiculescu in \cite{Voi85} (see also \cite{VDN92,Voi95,Spe98}), a
\emph{\Cs-probability space} $(A,\rho)$ consists of a unital \Cs-algebra $A$ and a state $\rho$.
The concept generalises the description of classical probability in terms of an algebra of bounded random variables and its expectation, allowing the algebra to be non-commutative. An \emph{operator-valued \Cs-probability space} replaces
scalar-valued expectations by ones taking values in a prescribed \Cs-subalgebra. It consists of a unital inclusion $D\subseteq A$ of
\Cs-algebras together with a conditional expectation
$\mathbb E\colon A\to D$, that is, a unital completely
positive $D$-bimodule map restricting to the identity on $D$. Taking $D=\mathbb C1_A$ recovers the scalar-valued setting. Throughout this work, conditional expectations are assumed to be \emph{non-degenerate}, meaning that their associated GNS-representations are faithful.

Given two operator-valued \Cs-probability spaces $(A_1,\mathbb E_1)$ and $(A_2,\mathbb E_2)$ over the same unital \Cs-subalgebra $D$, their \emph{reduced amalgamated free product} $(A,\mathbb E)=(A_1,\mathbb E_1)\star_D(A_2,\mathbb E_2)$ combines the two algebras by identifying their copies of $D$ and making them freely independent relative to $\mathbb{E}$. More precisely, the \Cs-algebra $A$ is generated by embedded copies of $A_1$ and $A_2$, the expectation $\mathbb E\colon A\to D$ restricts to $\mathbb{E}|_{A_{1}}=\mathbb{E}_{1}$, $\mathbb{E}|_{A_{2}}=\mathbb{E}_{2}$, and $\mathbb{E}(a_{1}\cdots a_{n})=0$ for $a_{1}\in A_{i_{1}}^{\circ},\ldots,a_{n}\in A_{i_{n}}^{\circ}$ with $i_{1}\neq\cdots\neq i_{n}$. The latter vanishing condition for alternating products of centred elements expresses \emph{free independence with amalgamation over $D$}.

Free independence has long been an effective tool for studying
the structure of \Cs-algebras beyond the nuclear setting. An early example of this is Powers' proof that the reduced \Cs-algebra of the free group of rank two is simple and has a unique tracial state \cite{Powers75}, followed
by the work of Avitzour \cite{Avitzour82} and
Voiculescu \cite{Voi85} on reduced free products. More recently, free product methods have been employed to obtain strong regularity properties for non-nuclear \Cs-algebras, as we discuss below. This motivates the problem of realising reduced free product extensions inside norm ultrapowers of given \Cs-algebras, including those not presented in free product form.

\subsection*{The main theorem}

The present article is concerned with the approximate realisation of reduced amalgamated free products within one of their constituent algebras. Adjoining a free factor introduces elements with prescribed interactions with the original algebra. We ask when these interactions can already be approximated within the original algebra, respecting algebraic relations, conditional expectations, and norms of mixed expressions, while keeping the original elements fixed. In this context, the relevant notion is that of an \emph{existential embedding} of operator-valued \Cs-probability spaces. An embedding $(A,\mathbb E_A)\hookrightarrow (C,\mathbb E_C)$ of such spaces is called existential, if there exist an ultrafilter $\mathcal U$ and a unital $\ast$-embedding $\theta\colon C\hookrightarrow A^{\mathcal U}$ for which the diagonal copy of $A$ inside $A^\mathcal{U}$ is fixed, and for which $\mathbb E_A^{\mathcal U} \circ \theta=\Delta \circ \mathbb E_C$, where $A^{\mathcal U}$ denotes the norm ultrapower and $\Delta\colon A\hookrightarrow A^{\mathcal U}$ is the canonical diagonal embedding. Thus, existentiality of the \emph{first free factor embedding}
\[
    (A_1,\mathbb E_1)
    \hookrightarrow
    (A_1,\mathbb E_1)\star_D(A_2,\mathbb E_2)
\]
allows one to exploit free independence even when
$A_1$ is not presented as a free product.

A principal obstacle to constructing existential embeddings as above is controlling the reduced free product norm. Indeed, unital $\ast$-homomorphisms on the free factors that agree on $D$ determine a $\ast$-homomorphism on the algebraic amalgamated free product, but this map
need not continuously extend to its reduced
\Cs-completion. Our main theorem gives a necessary and sufficient criterion for existentiality of the first free factor embedding in terms of representations of the individual
factors and compression relations on complementary
Hilbert subspaces. It thereby provides a way to control mixed expressions through the geometry of these representations. Moreover, the theorem allows the image of the second factor to be constrained to $B^{\mathcal U}$ for a
prescribed unital \Cs-subalgebra $B\subseteq A_1$.

\begin{theorem} \label{MainTheorem} Let $D\subseteq A_{1},A_{2}$ be unital inclusions of \Cs-algebras admitting non-degenerate conditional expectations $\mathbb{E}_{1}\colon A_{1}\to D$ and $\mathbb{E}_{2}\colon A_{2}\to D$. Furthermore, let $B\subseteq A_{1}$ be another unital inclusion of \Cs-algebras. Then the following statements are equivalent:

\begin{enumerate}[label=\textup{(\roman*)},leftmargin=*]
\item There exist an ultrafilter $\mathcal{U}$ and a unital embedding  $\theta\colon(A_{1},\mathbb{E}_{1})\star_{D}(A_{2},\mathbb{E}_{2})\hookrightarrow A_{1}^{\mathcal{U}}$ such that
\[
\theta|_{A_{1}}=\Delta,\quad\mathbb{E}_{1}^{\mathcal{U}}\circ\theta=\Delta \circ(\mathbb{E}_{1}\star\mathbb{E}_{2}),\quad\theta(A_{2})\subseteq B^{\mathcal{U}}.
\]
\item There exist Hilbert spaces $\mathcal{H}$ and $\mathcal{K}$, an ultrafilter $\mathcal{U}$, unital $\ast$-embeddings $\pi_{1}\colon A_{1} \hookrightarrow \mathcal{B}(\mathcal{H})$, $\pi_{2}\colon A_{2}\hookrightarrow B^{\mathcal{U}}\subseteq A_{1}^{\mathcal{U}}$ and $\rho\colon D\hookrightarrow\mathcal{B}(\mathcal{K})$, an isometry $V\colon\mathcal{K}\to\mathcal{H}$, and a projection $p\in\mathcal{B}(\mathcal{H}^{\mathcal{U}})$ satisfying the following conditions:

\begin{enumerate}[leftmargin=*,label=\textup{(\alph*)}]
\item The embedding $\pi_{2}$ restricts to $\Delta|_{D}$ on $D$.
\item For every $a\in A_{1}$, the identity $V^{\ast}\pi_{1}(a)V=(\rho\circ\mathbb{E}_{1})(a)$ holds.
\item For every $a_{1}\in A_{1}$ and $a_{2}\in A_{2}$,
\begin{align*}
p\,\bigl((\pi_{1}^{\mathcal{U}}\circ\Delta)(a_{1})\bigr) \,p & =\bigl((\pi_{1}^{\mathcal{U}}\circ\Delta\circ\mathbb{E}_{1})(a_{1})\bigr)\,p,\\
p^\perp\,\bigl((\pi_{1}^{\mathcal{U}}\circ\pi_{2})(a_{2})\bigr)\, p^\perp & =\bigl((\pi_{1}^{\mathcal{U}}\circ\Delta\circ\mathbb{E}_{2})(a_{2})\bigr)\, p^\perp,\\
p\,\bigl((\pi_{1}^{\mathcal{U}}\circ\Delta)(a_{1})\bigr)\,V^{\mathcal{U}} & =0,
\end{align*}
where $V^{\mathcal{U}}\colon\mathcal{K}^{\mathcal{U}}\to\mathcal{H}^{\mathcal{U}}$ is the isometry induced by $V$.
\end{enumerate}
\end{enumerate}
Here $\Delta$ denotes in both cases the diagonal embedding $A_{1}\hookrightarrow  A_{1}^{\mathcal{U}}$.
\end{theorem}

In the main body of the article, we prove the two implications of Theorem \ref{MainTheorem} separately as Proposition \ref{Direction1} and Proposition \ref{Direction2}. 
Our approach is based on Hasegawa's construction of a
canonical ambient \Cs-algebra $\Delta\mathbf{T}(A,\mathbb{E})$ associated with a reduced
amalgamated free product \cite{Hasegawa18}; a related
construction for reduced graph products appears
in \cite{Klisse25}. Its universal property turns complementary compression relations into a representation of the reduced amalgamated free product. In the converse direction, Hadwin's non-separable version of Voiculescu's theorem \cite{Had81} allows us to transfer the canonical free product representation to a Hilbert space ultrapower while retaining the subspace which implements the conditional expectation.
In particular, our argument does not invoke the universal property of a reduced (amalgamated) free product with the Toeplitz algebra, as used in Ozawa's
proof of his criterion for complete selflessness in
\cite[Theorem~13]{Ozawa25}.
We recover this criterion as a consequence of
Theorem~\ref{AbstractCriterion-completely-selfless};
see Remark~\ref{OzawaRemark}.


\subsection*{A new characterisation of selflessness}
Inspired by the work of Popa \cite{Pop95}, Dykema and R{\o}rdam \cite{DR98,DR00}, Robert introduced \emph{selfless} \Cs-probability spaces in \cite{Rob25} to capture a mechanism by which free product structures can be realised inside ultrapowers. A \Cs-probability space $(A,\rho)$ is selfless if the first free factor embedding of $(A,\rho)$ into its reduced free product  with another non-trivial \Cs-probability space is existential. This allows structural properties of reduced free products to be transferred back to $A$, in a way analogous to that in \cite{BS16}. The resulting class includes the Jiang--Su algebra $\mathcal Z$, infinite-dimensional UHF algebras, and reduced free product examples \cite{Rob25,HKER25,HKEPR25,FFKP26a,GKPT26,GJKPR26}, as well as a broad class of reduced group \Cs-algebras \cite{Rob25,AGKP25,Vig25,RTV25,Ozawa25,MWY25a,MWY25b, FFKP26a,FFKP26b,Vig26,GKPT26,Oml26a,FG26,LBLV26,Ari26, Oml26b,AF26}, and crossed products \cite{MWY25a,MWY25b,FFKP26b,GJKPR26,Bell26a,Ohs26,Bell26b}. A central feature of this theory is that selfless \Cs-probability spaces split into two distinct classes:

\begin{equation*}
    (A,\rho)\text{ is selfless}
    \Longleftrightarrow
    \begin{cases}
        \text{if $\rho$ is tracial:}
        &
        \parbox[t]{0.5\textwidth}{
            $A$ is infinite-dimensional and simple,
            $\rho$ is the unique tracial state on $A$,
            and $A$ has strict comparison of positive
            elements with respect to $\rho$,
        }
        \\[4ex]
        \text{if $\rho$ is not tracial:}
        &
        \text{$A$ is simple and purely infinite.}
    \end{cases}
\end{equation*}

The tracial characterisation follows from \cite[Theorem~3.1]{Rob25} and \cite[Theorem 1.1]{Rob26}. The non-tracial part is a consequence of  \cite[Theorem~3.1]{Rob25} (and \cite[Theorem~2]{Gould26} for the non-faithful case), and \cite[Theorem 3]{Ozawa25}. We stress that, in the tracial setting, selflessness also implies stable rank one (see \cite[Theorem~3.1]{Rob25} or \cite[Theorem 1.1]{Lin25}), and uniqueness of unital embeddings of $\mathcal Z$ up to strong asymptotic unitary equivalence (see \cite[Theorem B]{Pag26}, and \cite[Proposition 6.3.1]{Rob12} for the statement up to approximate unitary equivalence). While the emergence of such strong regularity from free product
constructions in the non-nuclear setting is striking, the aforementioned dichotomy for selfless \Cs-algebras parallels the classification theory of simple nuclear \Cs-algebras; see 
\cite{Win14,Win16,GLN20a,GLN20b,EGLN25,TWW17,CETWW21,CGSTW23,Kir95,KP00,Phi00} for a selection of the many contributions to this collective endeavour.

Since Robert's work first saw light, it inspired a plethora of interesting new concepts. In \cite{HKEPR25}, selflessness was extended to inclusions of
\Cs-probability spaces $B\subseteq(A,\rho)$ by additionally
requiring the auxiliary free factor to be represented in the
norm ultrapower of $B$.
Every intermediate \Cs-probability space of a selfless inclusion is selfless, and consequently, this relative notion gives access to regularity for intermediate algebras. 
To obtain a notion compatible with spatial tensor
products of \Cs-probability spaces, Ozawa defined in \cite{Ozawa25} a strengthening called \emph{complete selflessness}. His definition ensures that spatial tensor products of completely selfless \Cs-probability spaces remain completely selfless.  
Furthermore, in \cite{GJKPR26} a relative notion of selflessness was defined. In particular, the latter work presents the right operator-valued extension of selflessness, called \emph{relative selflessness}, for operator-valued \Cs-probability spaces.


With the full power of Theorem \ref{MainTheorem} at hand, we provide a new characterisation of (completely) selfless \Cs-probability spaces, selfless inclusions, and relative selflessness. To simplify notation, we adopt the following ad hoc terminology
throughout the article.

\begin{definition*}
    Let $(A,\rho)$ be a \Cs-probability space. A \emph{(faithful) representation} $(\pi,\mathcal{H},\xi)$ of $(A,\rho)$ consists of a (faithful) unital $\ast$-representation $\pi\colon A\to\mathcal{B}(\mathcal{H})$ on a Hilbert space $\mathcal{H}$, and a unit vector $\xi\in \mathcal{H}$ such that $\rho(a)=\langle \xi, \pi(a)\,\xi\rangle$ for all $a\in A$.
\end{definition*}

\begin{theorem} \label{AbstractCriterion-general} Let $B\subseteq(A,\rho)$ be an inclusion of \Cs-probability spaces. Then the following statements are equivalent:

\begin{enumerate}[label=\textup{(\roman*)},leftmargin=*]
\item The inclusion $B\subseteq(A,\rho)$ is selfless.
\item There exist an ultrafilter $\mathcal{U}$, a non-trivial \Cs-probability space $(C,\gamma)$ with a state-preserving $\ast$-embedding $\theta\colon(C,\gamma)\hookrightarrow(B^{\mathcal{U}},\rho^{\mathcal{U}})$, a faithful representation $(\pi,\mathcal{H},\xi)$ of $(A,\rho)$, and a projection $p\in\mathcal{B}(\mathcal{H}^{\mathcal{U}})$ such that
\begin{align*}
p\,\bigl((\pi^{\mathcal{U}}\circ\Delta)(a)\bigr)\,p &= \rho(a)\,p, \\
p^{\perp}\,\bigl((\pi^{\mathcal{U}}\circ\theta)(c)\bigr)\,p^{\perp}&= \gamma(c)\,p^{\perp}, \\
p\,\bigl((\pi^{\mathcal{U}}\circ\Delta)(a)\bigr)\,\xi^\mathcal{U} &= 0,
\end{align*}
for all $a\in A$ and $c\in C$.
\end{enumerate}
\end{theorem}

Our characterisation expresses selflessness in terms of complementary compression relations for a projection that need not belong to the represented \Cs-algebra.
This flexibility is particularly useful when the projection is constructed from a natural decomposition of the underlying Hilbert space.

Our approach also allows us to describe  complete selflessness of \Cs-probability spaces. Recall that a \Cs-probability space $(A,\rho)$ is called \emph{completely selfless} if a non-trivial reduced free product extension embeds into its norm ultrapower while fixing the diagonal copy of $A$, and this embedding induces embeddings into $(A\otimes C)^{\mathcal U}$ upon tensoring with any auxiliary \Cs-algebra $C$. For an inclusion $B\subseteq(A,\rho)$, complete selflessness is defined by additionally requiring the auxiliary free factor to be mapped into $B^{\mathcal U}$.

\begin{theorem} \label{AbstractCriterion-completely-selfless} Let $B\subseteq(A,\rho)$ be an inclusion of \Cs-probability spaces. Then the following statements are equivalent:

\begin{enumerate}[label=\textup{(\roman*)},leftmargin=*]
\item  $B\subseteq(A,\rho)$ is completely selfless.
\item There exist an ultrafilter $\mathcal{U}$, a non-trivial \Cs-probability space $(C,\gamma)$ with a state-preserving $\ast$-embedding $\theta\colon(C,\gamma)\hookrightarrow(B^{\mathcal{U}},\rho^{\mathcal{U}})$, a unital faithful representation $(\pi,\mathcal{H},\xi)$ of $(A,\rho)$, a unit vector $\eta\in \ell^2$, and a projection $p\in\mathcal{B}((\mathcal{H}\otimes\ell^2)^{\mathcal{U}})$ such that
\begin{align*}
[p, (\id_{\mathcal{H}}\otimes d)^\mathcal{U}] &= 0, \\
p\,\bigl(((\pi\otimes 1_{\cal{B}(\ell^2)})^{\mathcal{U}}\circ \Delta) (a)\bigr)\,p &= \rho(a)\,p, \\
p^{\perp}\,\bigl(((\pi\otimes 1_{\cal{B}(\ell^2)})^{\mathcal{U}}\circ \theta)(c)\bigr)\,p^{\perp}&= \gamma(c)\,p^{\perp}, \\
p\,\bigl(((\pi\otimes 1_{\cal{B}(\ell^2)})^{\mathcal{U}}\circ \Delta)(a)\bigr)\,(\xi\otimes \eta)^\mathcal{U} &= 0, 
\end{align*}
for all $a\in A$, $c\in C$, $d\in\mathcal B(\ell^2)$.
\item There exist an ultrafilter $\calU$, a unital faithful representation $(\pi,\mathcal H,\xi)$ of $(A,\rho)$, a unit vector $\eta\in\ell^2$, and an isometry $T\in\mathcal B((\mathcal H\otimes\ell^2)^{\mathcal U})$ such that $T+T^* \in (\pi\otimes 1_{\cal{B}(\ell^2)})^{\calU}(B^{\calU})$ and
\begin{align*}
[T, (\id_{\mathcal{H}}\otimes d)^\mathcal{U}] &= 0, \\
T^*\,\bigl(((\pi\otimes 1_{\cal{B}(\ell^2)})^{\mathcal{U}}\circ \Delta)(a)\bigr)\,T &= \rho(a)\,1, \\
T^*\,\bigl(((\pi\otimes 1_{\cal{B}(\ell^2)})^{\mathcal{U}}\circ \Delta)(a)\bigr)\,(\xi\otimes \eta)^\mathcal{U} &= 0, 
\end{align*}
for all $a\in A$ and $d\in\mathcal B(\ell^2)$.
\end{enumerate}
\end{theorem}

In the assertions \textup{(ii)} and \textup{(iii)} of Theorem \ref{AbstractCriterion-completely-selfless}, the conditions may equivalently be required to hold for every unit vector $\eta\in\ell^2$, with all remaining data independent of $\eta$. The Theorem's third assertion should be compared with Ozawa's criterion for complete selflessness in \cite[Theorem 13]{Ozawa25}.


\subsection*{Applications of the main results}

The first application of our new criterion in Theorem \ref{AbstractCriterion-general} concerns reduced amalgamated free products. It shows that a net of unitaries in one factor satisfying suitable expectation estimates produces a completely selfless inclusion in the free product, with the smaller algebra explicitly generated by these unitaries and a subalgebra of the other factor.

\begin{theorem} \label{thm:amalgamated-free-product} Let $(A_1,\rho_1)$ and $(A_2,\rho_2)$ be  \Cs-probability spaces, let $D\subseteq A_1, A_2$ be a common unital \Cs-subalgebra equipped with a state $\delta$, and let $\mathbb{E}_1 \colon (A_1,\rho_1)\rightarrow (D,\delta)$ and $\mathbb{E}_2 \colon (A_2,\rho_2)\rightarrow (D,\delta)$ be state-preserving conditional expectations. Suppose there exists a net \((v_i)_{i\in I}\subseteq A_1\) of unitaries such that, for every \(a\in A_1\),
\begin{equation*}
\|\mathbb{E}_1(v_i^*av_i)-\rho_1(a)1_D\|\longrightarrow0,\quad \|\mathbb{E}_1(v_i^*a)\|\longrightarrow0.
\end{equation*}
Suppose furthermore that \(A_2\) contains a non-trivial unital \Cs-probability subspace $C\subseteq (A_2,\rho_2)$ satisfying $\mathbb{E}_2(c)=\rho_2(c)1_D$ for all $c\in C$. Then $(A_1,\mathbb{E}_1)\star_D(A_2,\mathbb{E}_2)$ equipped with $\delta\circ (\mathbb{E}_1 \star \mathbb{E}_2)$ is completely selfless.

Moreover, if the restriction of $\delta\circ (\mathbb{E}_1 \star \mathbb{E}_2)$ to $\Cs\bigl(C\cup\{v_i\mid i\in I\}\bigr)$
is GNS-faithful, then the inclusion
\[
\Cs\bigl(C\cup\{v_i\mid i\in I\}\bigr)\subseteq\left( (A_1,\mathbb{E}_1)\star_D(A_2,\mathbb{E}_2), \delta\circ (\mathbb{E}_1 \star \mathbb{E}_2)\right)
\]
is completely selfless.
\end{theorem}

To showcase the scope of Theorem \ref{thm:amalgamated-free-product}, we apply it to the special case of reduced free products of \Cs-algebras where one free factor contains a Haar unitary belonging to the centraliser of the state. We note that by \cite[Theorem~5.4]{Thi24} this implies a generalisation of \cite[Theorem C]{HKER25} by removing the rapid decay condition and the von Neumann completion assumption therein.

\begin{corollary}\label{thm:free-product} Let $ C\subseteq(A,\rho)$ with $C\not\cong \mathbb C$ be an inclusion of \Cs-probability spaces, and let $(B,\psi)$ be a \Cs-probability space containing a Haar unitary $u \in B^\psi$, where $B^\psi$ is the {centraliser} of $\psi$. Then the inclusion $\Cs(C \cup \{u\})\subseteq (B,\psi)\star(A,\rho)$ is completely selfless.
\end{corollary}

By applying Theorem~\ref{thm:amalgamated-free-product} and Corollary~\ref{thm:free-product}, we establish complete selflessness of free Araki--Woods \Cs-probability spaces whenever the underlying real Hilbert space has dimension
at least two; see Corollary~\ref{cor:free-araki-woods-selfless}. This strengthens the results of \cite[Theorem~7.6 and Corollary~7.7]{HKER25} and also covers the two-dimensional case associated with non-trivial orthogonal representations.  In particular, the corresponding \Cs-algebras are simple and are purely infinite whenever the orthogonal representation is non-trivial.
As a further application of Corollary \ref{thm:free-product}, we show in Theorem~\ref{thm:Hecke} that Hecke \Cs-algebras associated with free product groups $(\mathbb{Z}/2\mathbb Z)^{\star n}$ with $n\geq3$ are completely selfless if and only if they are simple. Consequently, the Elliott invariant of these \Cs-algebras is completely determined by its unordered version, which was computed in \cite{RS22}; see Remark~\ref{rem:ordered-K-theory}.

As a second application of Theorem \ref{MainTheorem}, we consider reduced group \Cs-algebras and crossed products of discrete groups. Recall that a discrete group is called \emph{(completely) \Cs-selfless} if the \Cs-probability space consisting of the corresponding reduced group \Cs-algebra and its canonical tracial state is (completely) selfless. We give a dynamical characterisation of (complete) \Cs-selflessness in terms of covariant representations of the action on the Bowditch
compactification of a Bass--Serre tree, and obtain completely selfless inclusions for suitable amalgamated free product groups involving
unconfined subgroups; see Corollary~\ref{cor:unconfined}. Secondly, for reduced crossed products, we establish relative selflessness with respect to the canonical conditional expectation whenever the acting group admits an axial sequence whose action on the coefficient
algebra converges pointwise in norm to the identity along a free ultrafilter. Notably, this extends \cite[Theorem~1]{Ozawa25} to the operator-valued setting.

Finally, as a third application of our main results, we establish selflessness of \Cs-probability spaces of $q$-Gaussian \Cs-algebras throughout the range $-1<q<1$. More precisely, for every real Hilbert space $\mathcal{H}_{\mathbb R}$ of dimension at least two, equipped with orthogonal unit vectors $\xi$, $\eta$, the inclusion $\Cs(s_q(\xi),s_q(\eta)) \subseteq(\mathcal A_q(\mathcal{H}_{\mathbb R}),\tau)$ is shown to be completely selfless, where $\tau$ is the vacuum trace; see Theorem~\ref{thm:q-gaussian-selfless}. This resolves Problem 12 asked in the Oberwolfach report \cite{OWR25}, and strengthens the simplicity and trace-uniqueness results of Amrutam, Jekel, and Wasilewski \cite{AmrutamJekelWasilewski26} (see also \cite[Theorem 4.1]{BCKW23}) by establishing complete selflessness, strict comparison, and stable rank one, with the same conclusions for all intermediate algebras. \\

\subsection*{Statement on the use of AI}
Theorems~\ref{MainTheorem},
\ref{AbstractCriterion-general}, and
\ref{AbstractCriterion-completely-selfless}, which provide
the main characterisations of selflessness, relative
selflessness, and complete selflessness (of inclusions) were obtained without
the use of AI.
All results in Section~\ref{sec:groups}, concerning the
characterisation of (completely) \Cs-selfless groups and
relative selflessness for crossed products, were likewise
proved without the use of AI.

Our original proof of the free product result,
Corollary~\ref{thm:free-product}, was obtained without AI
and subsequently generalised to
Theorem~\ref{thm:amalgamated-free-product} with the assistance
of ChatGPT 5.6.
The results concerning Hecke \Cs-algebras in
Section~\ref{sec:Hecke} were obtained without AI.
The result on $q$-Gaussian \Cs-algebras in
Section~\ref{sec:q-gaussian} was proved by ChatGPT 5.6;
the authors subsequently checked the proof carefully,
modified it, and improved its exposition.
Furthermore, ChatGPT 6.1 was used for proofreading and spell-checking.

All conversations with ChatGPT that contributed to the mathematical content of this
work in the ways described above took place before
6 October 2026. The authors take full responsibility for the mathematical
content of this work.

\subsection*{Acknowledgements}
This project originated during M.K.’s visit to the University of Southern Denmark in February 2026, and he is grateful for the warm hospitality received during his stay. M.P.\ was supported by the Carlsberg Foundation grant CF24-1144.

\vspace{3mm}

\renewcommand\thetheorem{\arabic{section}.\arabic{theorem}}


\section{Preliminaries and notation}

\subsection{General notation}

We write $\mathbb N=\{0,1,\ldots\}$, $\mathbb N_{\geq1}=\{1,2,\ldots\}$ and $\ell^2:=\ell ^2 (\mathbb{N})$. In a unital \Cs-algebra $A$, scalars are identified with scalar multiples of the identity $1_A$, and $\Cs(S)$ denotes the unital \Cs-subalgebra generated by a subset $S\subseteq A$. For a projection $p$, we write $p^\perp:=1-p$. Tensor products of \Cs-algebras are spatial unless otherwise specified. Furthermore, inner products on Hilbert spaces are assumed to be linear in the
second variable.

\subsection{\Cs-probability spaces}

A \emph{\Cs-probability space} is a pair $(A,\rho)$ consisting of a unital \Cs-algebra and a {GNS-faithful} state, meaning that the associated GNS-representation is faithful; this assumption is not imposed on ultrapower states. A probability space is \emph{non-trivial} if $A\not\cong\mathbb C$. Embeddings of \Cs-probability spaces are assumed to be unital and state-preserving. An \emph{inclusion of \Cs-probability spaces} $B\subseteq(A,\rho)$ is a unital inclusion of \Cs-algebras with $\rho|_B$ GNS-faithful.

A unitary $u\in A$ is a \emph{Haar unitary} if $\rho(u^n)=0$ for every $n\in\mathbb Z\setminus\{0\}$. The \emph{centraliser} of $\rho$ is defined as $A^\rho:=\{a\in A\mid\rho(ab)=\rho(ba)\text{ for all }b\in A\}$.

\subsection{Ultrapowers}

For an ultrafilter $\mathcal{U}$ on a non-empty set $I$, the \emph{(norm) ultrapower} of a \Cs-algebra $A$ is
\[
    A^{\mathcal{U}}:=\frac{\ell^\infty(I,A)}{\{(a_i)_{i\in I}\in\ell^\infty(I,A)\mid \lim_{i\to\mathcal{U}}\|a_i\|=0\}}.
\]
We denote the class of an element $(a_i)_{i\in I} \in \ell^\infty(I,A)$ by $(a_i)_{i\to\mathcal{U}}$; the corresponding norm is given by $\|(a_i)_{i\to\mathcal U}\|=\lim_{i\to\mathcal U}\|a_i\|$. The diagonal $\ast$-embedding $\Delta\colon A\hookrightarrow
A^{\mathcal U}$ maps $a\in A$ to the class of the constant family $(a)_{i\in I}$, and we identify $A$ with its image
$\Delta(A)$. If $A$ is unital, then $A^{\mathcal U}$ is a unital \Cs-algebra, and $\Delta$ is a unital $\ast$-embedding.

Any state $\rho$ on $A$ induces a state $\rho^{\mathcal U}$ on $A^{\mathcal U}$ given by
$\rho^{\mathcal{U}}((a_i)_{i\rightarrow \mathcal{U}}) :=\lim_{i\to\mathcal{U}}\rho(a_i)$.
More generally, for unital \Cs-algebras $A$ and $B$, a bounded linear map $T\colon A\to B$ induces $T^{\mathcal{U}}\colon A^{\mathcal{U}}\to B^{\mathcal{U}}$ by applying $T$ coordinatewise. 

The \emph{Hilbert space ultrapower} $\mathcal H^{\mathcal{U}}$ is again a Hilbert space with inner product
\[
    \langle (\xi_i)_{i \rightarrow \mathcal{U}},(\eta_i)_{i \rightarrow \mathcal{U}}\rangle :=\lim_{i\to\mathcal{U}} \,\langle\xi_i,\eta_i\rangle.
\]
For a vector $\xi\in\mathcal H$, we denote the class of the constant family by $\xi^{\mathcal{U}}$, or simply by $\xi$ if the ambient space is clear from context. A bounded family of operators $(T_i)_{i\in I}\subseteq\mathcal B(\mathcal H,\mathcal K)$ induces an operator from $\mathcal H^{\mathcal U}$ to $\mathcal K^{\mathcal U}$ by $(\xi_i)_{i\to\mathcal U} \mapsto (T_i\xi_i)_{i\to\mathcal U}$. This gives a canonical isometric linear embedding $\mathcal B(\mathcal H,\mathcal K)^{\mathcal U} \hookrightarrow \mathcal B(\mathcal H^{\mathcal U},\mathcal K^{\mathcal U})$; when $\mathcal H=\mathcal K$, this is a unital
$\ast$-embedding.
For a fixed operator $T \in \mathcal{B}(\mathcal{H}, \mathcal{K})$, the induced operator is denoted by $T^{\mathcal U}$. The construction preserves adjoints and isometries. A faithful $\ast$-representation $\pi\colon A\to\mathcal{B}(\mathcal H)$ induces a faithful representation of $A^{\mathcal{U}}$ on $\mathcal H^{\mathcal{U}}$.

For an ultrafilter $\mathcal{W}$ on $J$, our convention for the \emph{Fubini ultrafilter} $\mathcal{U}\otimes\mathcal{W}$ on $I\times J$ is
\[
    S\in\mathcal{U}\otimes\mathcal{W} \quad\Longleftrightarrow\quad \{j\in J\mid\{i\in I\mid(i,j)\in S\}\in\mathcal{U}\} \in\mathcal{W}.
\]
Thus $(A^{\mathcal{U}})^{\mathcal{W}}$ identifies canonically with $A^{\mathcal{U}\otimes\mathcal{W}}$, and likewise for Hilbert space ultrapowers.

If $I$ is directed, an ultrafilter $\mathcal U$ on $I$ is
called \emph{cofinal} if
$\{j\in I\mid j\geq i\}$ belongs to $\mathcal{U}$ for every $i\in I$. Thus convergence of a net indexed by $I$ implies convergence along any cofinal ultrafilter on $I$.


\subsection{Hilbert \Cs-modules}

For the general theory of Hilbert \Cs-modules we refer to \cite{Lance95}. By convention, Hilbert \Cs-modules are right-modules and the inner product is linear in the right variable. We  denote Hilbert \Cs-modules by gothic letters, while ordinary Hilbert spaces are denoted by calligraphic letters. Furthermore, given a \Cs-algebra $A$ and Hilbert $A$-modules $\mathfrak{X}$ and $\mathfrak{Y}$, we write $\mathcal{B}(\mathfrak{X},\mathfrak{Y})$ for the set of adjointable operators from $\mathfrak{X}$ to $\mathfrak{Y}$ and define $\mathcal{B}(\mathfrak{X}):=\mathcal{B}(\mathfrak{X},\mathfrak{X})$.

A \emph{(faithful) $A$-$B$ \Cs-correspondence} of two \Cs-algebras $A$ and $B$ is a pair $(\mathfrak{X},\phi_{\mathfrak{X}})$ consisting of a Hilbert $B$-module $\mathfrak{X}$ and a (faithful) $\ast$-representation $\phi_{\mathfrak{X}}\colon A\rightarrow\mathcal{B}(\mathfrak{X})$; it is called a \emph{(faithful) \Cs-correspondence} over $A$ if $A=B$. If $(\mathfrak{X},\phi_{\mathfrak{X}})$ and $(\mathfrak{Y},\phi_{\mathfrak{Y}})$ are respectively $A$-$B$ and $B$-$C$ \Cs-correspondences, the \emph{interior tensor product} $\mathfrak{X}\otimes_{B}\mathfrak{Y}$ obtained from $\mathfrak{X}\odot\mathfrak{Y}$ by separation and completion with respect to the semi-inner product $\langle\xi_{1}\otimes\eta_{1},\xi_{2}\otimes\eta_{2}\rangle:=\langle\eta_{1},\phi_{\mathfrak{Y}}(\langle\xi_{1},\xi_{2}\rangle)\eta_{2}\rangle$ for $\xi_{1},\xi_{2}\in\mathfrak{X}$, $\eta_{1},\eta_{2}\in\mathfrak{Y}$ is itself a Hilbert $C$-module. The canonical $\ast$-homomorphism $\mathcal{B}(\mathfrak{X})\rightarrow\mathcal{B}(\mathfrak{X}\otimes_{B}\mathfrak{Y})$ is faithful whenever $\phi_{\mathfrak{Y}}$ is faithful; in this case $\mathfrak{X}\otimes_{B}\mathfrak{Y}$ is an $A$-$C$ \Cs-correspondence.

A \emph{conditional expectation} $\mathbb{E}\colon A\to D$ onto a unital \Cs-subalgebra is a unital completely positive $D$-bimodule map which restricts to the identity on $D$. For a unital inclusion $D\subseteq A$ of \Cs-algebras admitting a conditional expectation $\mathbb{E}\colon A\rightarrow D$, we denote the Hilbert $D$-module obtained from $A$ with respect to the $D$-valued semi-inner product $\langle a,b\rangle:=\mathbb{E}(a^{\ast}b)$ by $L^{2}(A,\mathbb{E})$. Left multiplication by elements in $A$ defines a $\ast$-representation $\pi_{\mathbb{E}}\colon A\rightarrow\mathcal{B}(L^{2}(A,\mathbb{E}))$ and the vector $\xi_{\mathbb{E}}\in L^{2}(A,\mathbb{E})$ corresponding to the identity in $A$ satisfies $\mathbb{E}(a)=\langle\xi_{\mathbb{E}},\pi_{\mathbb{E}}(a)\xi_{\mathbb{E}}\rangle$ for every $a\in A$. The triple $(\pi_{\mathbb{E}},L^{2}(A,\mathbb{E}),\xi_{\mathbb{E}})$ is called the \emph{GNS-representation} of the pair $(A,\mathbb{E})$. The conditional expectation $\mathbb{E}$ is called \emph{non-degenerate} or \emph{GNS-faithful} if $\pi_{\mathbb{E}}$ is faithful.


\subsection{Amalgamated free products and Hasegawa's construction} \label{Subsec:amalgamated}

Let $D\subseteq A_{1},A_{2}$ be unital inclusions of \Cs-algebras and assume that there exist non-degenerate conditional expectations $\mathbb{E}_{1}\colon A_{1}\rightarrow D$ and $\mathbb{E}_{2}\colon A_{2}\rightarrow D$; denote the corresponding GNS-representations by $(\pi_{1},\mathfrak{X}_{1},\xi_{1})$ and $(\pi_{2},\mathfrak{X}_{2},\xi_{2})$ respectively and write $A_{1}^{\circ}:=\ker(\mathbb{E}_{1})$, $A_{2}^{\circ}:=\ker(\mathbb{E}_{2})$. The reduced amalgamated free product
\[
A:=(A_{1},\mathbb{E}_{1})\star_{D}(A_{2},\mathbb{E}_{2})
\]
of $(A_{1},\mathbb{E}_{1})$ and $(A_{2},\mathbb{E}_{2})$ is characterised, up to a canonical expectation-preserving
$\ast$-isomorphism fixing the two factors, as the unital
\Cs-algebra generated by embedded copies of $A_{1},A_{2}\subseteq A$ which are compatible on $D$, admitting a non-degenerate conditional expectation $\mathbb{E}:=\mathbb{E}_{1}\star\mathbb{E}_{2}\colon A\rightarrow D$ with $\mathbb{E}|_{A_{1}}=\mathbb{E}_{1}$, $\mathbb{E}|_{A_{2}}=\mathbb{E}_{2}$, and $\mathbb{E}(a_{1}\cdots a_{n})=0$ for $a_{1}\in A_{i_{1}}^{\circ},\ldots,a_{n}\in A_{i_{n}}^{\circ}$ with $i_{1}\neq\cdots\neq i_{n}$; for details see for instance \cite[Section 4.7]{BrownOzawa08}. Products $a_1\cdots a_n$ satisfying the preceding centering and alternating-index conditions are called \emph{reduced words}. Together with $D$, they linearly span a norm-dense subspace of $A$.

The Hilbert $D$-module of the associated GNS-representation of the pair $(A,\mathbb{E})$ is naturally identified with the \emph{free product Hilbert $D$-module}
\[
\mathfrak{X}:=\Omega D\oplus\bigoplus_{n\geq1}\bigoplus_{i_{1}\neq\cdots\neq i_{n}}\mathfrak{X}_{i_{1}}^{\circ}\otimes_{D}\cdots\otimes_{D}\mathfrak{X}_{i_{n}}^{\circ},
\]
where the \emph{vacuum vector} $\Omega$ corresponds to the canonical cyclic vector, and where $\mathfrak{X}_{1}^{\circ}:=\mathfrak{X}_{1}\ominus\xi_{1}D$, $\mathfrak{X}_{2}^{\circ}:=\mathfrak{X}_{2}\ominus\xi_{2}D$. We usually suppress the representation $\pi_{\mathbb{E}}$ in our notation and view $A$ as a \Cs-subalgebra of $\mathcal{B}(\mathfrak{X})$.

In \cite{Hasegawa18}, Hasegawa defined a \Cs-algebra $\Delta\mathbf{T}(A,\mathbb{E})$ that contains a canonical copy of $A$. By \cite[Corollary~4.2.7]{Hasegawa18} this \Cs-algebra can be characterised as the universal unital \Cs-algebra generated by a unital copy of the algebraic amalgamated free product of $A_1$ and $A_2$ over $D$, and a projection $p$, subject to the relations
\[
p a_1 p=\mathbb{E}_1(a_1)p,\qquad p^\perp a_2 p^\perp=\mathbb{E}_2(a_2)p^\perp
\]
for all $a_1\in A_1$, $a_2 \in A_2$. The \Cs-subalgebra generated by the copies of $A_1$ and $A_2$ inside $\Delta \mathbf{T}(A,\mathbb{E})$ is canonically isomorphic to $A$.

In the case of reduced group \Cs-algebras of amalgamated free products of discrete groups, Hasegawa's construction identifies with a crossed product \Cs-algebras naturally arising in the context of Bass–Serre trees. Indeed, let $\Gamma=\Gamma_1\star_{\Lambda}\Gamma_2$ be an amalgamated free product of discrete groups $\Lambda \leq \Gamma_1,\Gamma _2$, let $A_1=C_\text{r}^{\ast}(\Gamma_1)$, $A_2=C_\text{r}^{\ast}(\Gamma_2)$ and $D=C_{\text{r}}^{\ast}(\Lambda)$ be the corresponding reduced group \Cs-algebras, and let $\mathbb{E}_1\colon A_1\to D$, $\mathbb{E}_2\colon A_2\to D$ be the canonical conditional expectations. Then $A\cong C_\text{r}^{\ast}(\Gamma)$ and by \cite[Proposition~3.1.1]{Hasegawa18} we have that $\Delta\mathbf{T}(A,\mathbb{E}) \cong C(\Delta T)\rtimes_\text{r}\Gamma$, where $T$ is the Bass--Serre tree of the amalgam and $\Delta T=V(T)\sqcup\partial T$ is its \emph{Bowditch compactification}, equipped with the natural $\Gamma$-action. This isomorphism identifies the canonical copy of $A$ with the canonical copy of $C_\text{r}^{\ast}(\Gamma)$ in the crossed product. For details, see \cite[Section 3.1]{Hasegawa18}.


\subsection{Existential embeddings and selflessness}

A state-preserving $\ast$-embedding $\iota\colon(A,\rho)\hookrightarrow (D,\delta)$ of \Cs-probability spaces is called \emph{existential} if there exist an ultrafilter $\mathcal{U}$ and a state-preserving embedding $\sigma\colon(D,\delta)\hookrightarrow(A^{\mathcal{U}},\rho^{\mathcal{U}})$ such that $\sigma\circ\iota=\Delta$; see \cite{Rob25}. For an embedding of inclusions of \Cs-probability spaces
$\iota\colon(B\subseteq(A,\rho))
\hookrightarrow(C\subseteq(D,\delta))$,
we additionally require $\sigma(C)\subseteq B^{\mathcal U}$;
see \cite{HKEPR25}.

An inclusion of \Cs-probability spaces
$B\subseteq(A,\rho)$ is \emph{selfless} if there exist a non-trivial \Cs-probability space $(C,\gamma)$, an ultrafilter $\mathcal{U}$, and a state-preserving embedding $\Phi\colon(A,\rho)\star(C,\gamma) \hookrightarrow(A^{\mathcal{U}},\rho^{\mathcal{U}})$ such that $\Phi|_A=\Delta$ and $\Phi(C)\subseteq B^{\mathcal{U}}$. The \Cs-probability space $(A,\rho)$ is called \emph{selfless} if the inclusion $A\subseteq(A,\rho)$ is selfless. Every unital intermediate \Cs-algebra of a selfless inclusion is selfless \cite{HKEPR25}. Ultrafilters in these definitions may be taken free, allowing arbitrary index sets. 

Following \cite[Section~6]{Ozawa25}, a \Cs-probability space $(A,\rho)$ is called
\emph{completely selfless} if there exist a non-trivial
\Cs-probability space $(C,\gamma)$, an ultrafilter
$\mathcal U$, and a state-preserving $\ast$-embedding
$\Phi\colon(A,\rho)\star(C,\gamma)
\hookrightarrow(A^{\mathcal U},\rho^{\mathcal U})$ with $\Phi|_A=\Delta$ such that, for every \Cs-algebra $R$, the induced algebraic
$\ast$-homomorphism 
$((A,\rho)\star(C,\gamma))\odot R
\rightarrow(A\otimes R)^{\mathcal U}$ given by $x\otimes r\mapsto(a_i\otimes r)_{i\to\mathcal U}$ with $\Phi(x)=(a_i)_{i\to\mathcal U}$ is continuous for the minimal tensor norm. The first free factor embedding is, in this case, said to be \emph{completely existential}. In the definition above, if $\Phi(C)\subseteq B^{\mathcal U}$ for a \Cs-probability subspace $B\subseteq(A,\rho)$, we say that the inclusion is \emph{completely selfless}. Clearly, in this case, every intermediate \Cs-probability space is completely selfless.

\vspace{3mm}


\section{A new characterisation of selflessness}

The present section is devoted to providing the proofs of our characterisation results Theorem \ref{MainTheorem}, Theorem \ref{AbstractCriterion-general}, and Theorem \ref{AbstractCriterion-completely-selfless}. We divide the proof of Theorem \ref{MainTheorem} in two parts: Proposition \ref{Direction1} and Proposition \ref{Direction2}. The first part makes use of Hadwin's generalisation of Voiculescu's theorem from \cite{Had81}.

\begin{theorem}[see~{\cite[Theorem 3.14]{Had81}}] \label{HadwinTheorem} Let $A$ be a unital \Cs-algebra, and let $\pi\colon A\rightarrow\mathcal{B}(\mathcal{H})$ and $\rho\colon A\rightarrow\mathcal{B}(\mathcal{K})$ be unital $\ast$-representations on Hilbert spaces $\mathcal{H}$ and $\mathcal{K}$. Then the following statements are equivalent:

\begin{enumerate}[label=\textup{(\roman*)},leftmargin=*]
\item The representations are \emph{approximately unitarily equivalent}, i.e., there exists a net $(u_{\lambda})_{\lambda\in\Lambda}\subseteq\mathcal{B}(\mathcal{H},\mathcal{K})$ of unitaries such that $\lim_{\lambda\in\Lambda}\Vert u_{\lambda}\pi(a)u_{\lambda}^{\ast}-\rho(a)\Vert=0$ for every $a\in A$.
\item For every $a\in A$ the equality $\rank(\pi(a))=\rank(\rho(a))$ holds.
\end{enumerate}
Here $\rank(\pi(a))$ denotes the dimension of the closure of $\im (\pi(a))$ in the Hilbert space $\mathcal{H}$, and similarly for $\rank(\rho(a))$.
\end{theorem}

Using Theorem \ref{HadwinTheorem}, we prove the following proposition, which yields the forward implication of Theorem \ref{MainTheorem}.

\begin{proposition} \label{Direction1} Let $D\subseteq A_{1},A_{2}$ be unital inclusions of \Cs-algebras admitting non-degenerate conditional expectations $\mathbb{E}_{1}\colon A_{1}\to D$ and $\mathbb{E}_{2}\colon A_{2}\to D$. Furthermore, let $B\subseteq A_{1}$ be another unital inclusion of \Cs-algebras, and let $\rho\colon D\hookrightarrow\mathcal{B}(\mathcal{K})$ be a faithful unital $\ast$-representation. Assume that there exist an ultrafilter $\mathcal{U}$ and a unital $\ast$-embedding $\theta\colon(A_{1},\mathbb{E}_{1})\star_{D}(A_{2},\mathbb{E}_{2})\hookrightarrow A_{1}^{\mathcal{U}}$ such that
\[
\theta|_{A_{1}}=\Delta,\quad\mathbb{E}_{1}^{\mathcal{U}}\circ\theta=\Delta\circ(\mathbb{E}_{1}\star\mathbb{E}_{2}),\quad\theta(A_{2})\subseteq B^{\mathcal{U}},
\]
where $\Delta$ denotes the diagonal embedding. Then there exists a Hilbert space $\mathcal{H}$, an ultrafilter $\mathcal{V}$, unital $\ast$-embeddings $\pi_{1}\colon A_{1}\hookrightarrow\mathcal{B}(\mathcal{H})$ and $\pi_{2}\colon A_{2}\hookrightarrow B^{\mathcal{V}}\subseteq A_{1}^{\mathcal{V}}$, an isometry $V\colon\mathcal{K}\to\mathcal{H}$, and a projection $p\in\mathcal{B}(\mathcal{H}^{\mathcal{V}})$ satisfying the following conditions:

\begin{enumerate}[label=\textup{(\alph*)},leftmargin=*]
\item The embedding $\pi_{2}$ restricts to $\Delta|_{D}$ on $D$.
\item For every $a\in A_{1}$ the identity $V^{\ast}\pi_{1}(a)V=(\rho\circ\mathbb{E}_{1})(a)$ holds.
\item For every $a_{1}\in A_{1}$ and $a_{2}\in A_{2}$,
\begin{align*}
p\,\bigl((\pi_{1}^{\mathcal{V}}\circ\Delta)(a_{1})\bigr)\,p & = \bigl((\pi_{1}^{\mathcal{V}}\circ\Delta\circ\mathbb{E}_{1})(a_{1})\bigr)\,p, \\
p^{\perp}\,\bigl((\pi_{1}^{\mathcal{V}}\circ\pi_{2})(a_{2})\bigr) \,p^{\perp} & = \bigl((\pi_{1}^{\mathcal{V}}\circ\Delta\circ\mathbb{E}_{2})(a_{2})\bigr) \,p^{\perp}, \\
p\,\bigl((\pi_{1}^{\mathcal{V}}\circ\Delta)(a_{1})\bigr) \,V^{\mathcal{V}} &= 0,
\end{align*}
where $V^{\mathcal{V}}\colon\mathcal{K}^{\mathcal{V}}\to\mathcal{H}^{\mathcal{V}}$ is the isometry induced by $V$.
\end{enumerate}
\end{proposition}

\begin{proof}
Let $A:=(A_1,\mathbb{E}_1)\star_D(A_2,\mathbb{E}_2)$ and $\mathbb{E}:=\mathbb{E}_1 \star \mathbb{E}_2$. Writing $\mathfrak{X}_{1}:=L^{2}(A_{1},\mathbb{E}_{1})$ and $\mathfrak{X}_{2}:=L^{2}(A_{2},\mathbb{E}_{2})$ for the GNS-Hilbert $D$-modules of $(A_{1},\mathbb{E}_{1})$ and $(A_{2},\mathbb{E}_{2})$ respectively, let
\[
\mathfrak X:=\Omega D\oplus \bigoplus_{n\geq1} \bigoplus_{i_{1}\neq\cdots\neq i_{n}} \mathfrak X_{i_1}^\circ\otimes_D\cdots\otimes_D \mathfrak X_{i_n}^\circ
\]
be the free product Hilbert $D$-module, where $\Omega$ is the vacuum vector. The interior tensor product $\mathcal F:=\mathfrak X\otimes_D \mathcal K$ is a Hilbert space, and the canonical representation $\varphi\colon A\hookrightarrow\mathcal{B}(\mathcal F)$ is faithful. The map $V_0\colon\mathcal K\to\mathcal F$ defined by $V_0\xi:=\Omega\otimes\xi$ is an isometry satisfying
\begin{equation}
V_0^*\varphi(x)V_0=\rho(\mathbb{E}(x)) \quad\text{for every }x\in A. \label{eq:V0Conjugation}
\end{equation}

The Hilbert space $\mathcal F$ decomposes as $\mathcal F=V_0\mathcal K\oplus\mathcal F_1^\circ\oplus\mathcal F_2^\circ$, where
\[
\mathcal F_i^\circ:= \left( \bigoplus_{n\geq1} \bigoplus_{i=i_{1} \neq\cdots\neq i_{n}} \mathfrak X_{i_1}^\circ\otimes_D\cdots\otimes_D \mathfrak X_{i_n}^\circ \right)\otimes_D \mathcal K
\]
for $i\in \{1,2\}$; denote the projection onto $\mathcal F_2^\circ \subseteq \mathcal{F}$ by $p_0$. Then,
\begin{equation}
\label{eq:Identities}
\begin{aligned}
p_0\varphi(a_1)p_0 &=\varphi(\mathbb{E}_1(a_1))p_0,\\
p_0^\perp\varphi(a_2)p_0^\perp &=\varphi(\mathbb{E}_2(a_2))p_0^\perp,\\
p_0\varphi(a_1)V_0&=0.
\end{aligned}
\end{equation}
for all $a_1 \in A_1$ and $a_2 \in A_2$.

Let $I$ be the index set of $\mathcal{U}$. Choose an infinite cardinal $\lambda\geq\max\{|I|,\dim\mathcal F,\aleph_0\}$ and set $\kappa:=2^\lambda$, so that $\kappa^{|I|}=\kappa^{\aleph_0}=\kappa$. For a set $S$ with $|S|=\kappa$, put $\mathcal H:=\mathcal F\otimes\ell^2(S)$. Let $(\delta_s)_{s\in S}$ be the canonical orthonormal basis of $\ell^2(S)$, fix $s_0\in S$, and write $\delta_0:=\delta_{s_0}$. For $x\in A$ and $\xi \in \mathcal{K}$ define
\[
\Pi(x):=\varphi(x)\otimes1,\qquad V\xi:=V_0\xi\otimes \delta_0,\qquad P:=p_0\otimes1.
\]
Then $\Pi\colon A\hookrightarrow\mathcal{B}(\mathcal H)$ is a faithful unital $\ast$-representation, $V\colon\mathcal K\to\mathcal H$ is an isometry, and $P$ is a projection in $\mathcal{B}(\mathcal H)$. Moreover, $\dim(\mathcal{H})=\dim(\mathcal H^{\mathcal{U}})=\kappa$. For $\pi_1:=\Pi|_{A_1}$ we have by \eqref{eq:V0Conjugation} that
\begin{equation}
V^*\pi_1(a)V=\rho(\mathbb{E}_1(a)) \quad\text{for every }a\in A_1. \label{eq:VConjugation}
\end{equation}

Define faithful unital $\ast$-representations $\alpha,\beta\colon A\hookrightarrow\mathcal{B}(\mathcal H^{\mathcal{U}})$ by $\alpha(x):=(\Pi(x))_{i\to\mathcal{U}}$ and $\beta(x):=\pi_{1}^{\mathcal{U}}(\theta(x))$ for $x\in A$. These agree on $A_1$, with
\begin{equation}
\alpha(a)=\beta(a) =\pi_{1}^{\mathcal{U}}(\Delta(a)) \quad\text{for }a\in A_1. \label{eq:AlphaBetaIdentity}
\end{equation}
The subspaces
\[
\begin{aligned}
\mathcal H_\alpha &:=\overline{\operatorname{span}} \{\alpha(x)V^{\mathcal{U}}\xi\mid x\in A,\ \xi\in\mathcal K^{\mathcal{U}}\},\\
\mathcal H_\beta &:=\overline{\operatorname{span}} \{\beta(x)V^{\mathcal{U}}\xi\mid x\in A,\ \xi\in\mathcal K^{\mathcal{U}}\},
\end{aligned}
\]
reduce $\alpha$ and $\beta$, respectively. Using \eqref{eq:V0Conjugation}, \eqref{eq:VConjugation} and the expectation-preserving property $\mathbb{E}_{1}^{\mathcal{U}}\circ\theta=\Delta\circ\mathbb{E}$ of $\theta$, we obtain
\[
(V^{\mathcal{U}})^*\alpha(x)V^{\mathcal{U}} =\rho^{\mathcal{U}}\bigl(\Delta(\mathbb{E}(x))\bigr) =(V^{\mathcal{U}})^*\beta(x)V^{\mathcal{U}}
\]
for $x\in A$. Applying this identity shows that the linear map $u_{0}\colon\mathcal{H}_{\beta}\rightarrow\mathcal{H}_{\alpha}$ given by $\beta(x)V^{\mathcal{U}}\xi \mapsto \alpha(x)V^{\mathcal{U}}\xi$ for $x\in A$, $\xi\in\mathcal{K}^{\mathcal{U}}$ is a unitary operator with $u_{0}V^{\mathcal{U}}=V^{\mathcal{U}}$ and $u_{0}(\beta(x)|_{\mathcal{H}_{\beta}})u_{0}^{\ast}=\alpha(x)|_{\mathcal{H}_{\alpha}}$ for $x\in A$, i.e., $u_{0}$ intertwines the representations $\alpha$ and $\beta$.\\

In order to apply Theorem \ref{HadwinTheorem}, we prove the following two claims.\\

\emph{Claim 1.} For every $x\in A\setminus\{0\}$ the equality $\operatorname{rank} \bigl(\alpha(x)|_{\mathcal H_\alpha^\perp}\bigr)=\kappa$ holds.

\vspace{1mm}

\emph{Proof of the claim.} Choose $\xi\in\mathcal F$ with $\varphi(x)\xi\neq0$. Since $\mathcal H_\alpha\subseteq (\mathcal F\otimes\mathbb C\delta_0)^{\mathcal{U}}$, the vector $(\xi\otimes \delta_s)_{i\to\mathcal{U}}$ belongs to $\mathcal H_\alpha^\perp$ for every $s\in S\setminus\{s_0\}$ with $\alpha(x)(\xi\otimes \delta_s)_{i\to\mathcal{U}} =(\varphi(x)\xi\otimes \delta_s)_{i\to\mathcal{U}}$. These vectors are non-zero and mutually orthogonal. Thus the rank is at least $\kappa$, and equality follows from $\dim(\mathcal H^{\mathcal{U}})=\kappa$.\\

\emph{Claim 2.} For every $x\in A\setminus\{0\}$ the equality $\operatorname{rank} \bigl(\beta(x)|_{\mathcal H_\beta^\perp}\bigr)=\kappa$ holds.

\vspace{1mm}

\emph{Proof of the claim.} Write $\theta(x)=(a_i)_{i\to\mathcal{U}}$ for a bounded family $(a_i)_{i\in I} \subseteq A_1$. Since $\theta(x)\neq0$, there exist $\varepsilon>0$ and $I_0\in\mathcal{U}$ such that $\|a_i\|>2\varepsilon$ for every $i\in I_0$. Faithfulness of $\varphi$ allows us to choose unit vectors $\xi_i\in\mathcal F$ such that $\|\varphi(a_i)\xi_i\|>\varepsilon$ for every $i\in I_0$. Since $\mathcal H_\beta\subseteq (\mathcal F\otimes\mathbb C\delta_0)^{\mathcal{U}}$, the vector $(\xi_i\otimes \delta_s)_{i\to\mathcal{U}}$ belongs to $\mathcal H_\beta^\perp$ for every $s\in S\setminus\{s_0\}$, and $\beta(x)(\xi_i\otimes \delta_s)_{i\to\mathcal{U}} =(\varphi(a_i)\xi_i\otimes \delta_s)_{i\to\mathcal{U}}$. These vectors are mutually orthogonal and have norm at least $\varepsilon$. The desired equality hence follows as in Claim 1.\\

Applying the claims to $x=1$ gives $\dim(\mathcal H_\alpha^\perp) =\dim(\mathcal H_\beta^\perp)=\kappa$. By Theorem~\ref{HadwinTheorem}, there hence is a net of unitaries $(u_j)_{j\in J}\subseteq \mathcal{B}(\mathcal H_\beta^\perp,\mathcal H_\alpha^\perp)$ such that $\| u_j(\beta(x)|_{\mathcal H_\beta^\perp})u_j^* -\alpha(x)|_{\mathcal H_\alpha^\perp} \|\rightarrow 0 $ for every $x\in A$. The unitaries $u_{0}\oplus u_{j}\in\mathcal{B}(\mathcal{H}^{\mathcal{U}})$ satisfy $(u_{0}\oplus u_{j})V^{\mathcal{U}}=u_{0}V^{\mathcal{U}}=V^{\mathcal{U}}$ and
\[
\Vert(u_{0}\oplus u_{j})\beta(x)(u_{0}\oplus u_{j})^{\ast}-\alpha(x)\Vert\longrightarrow0\quad\text{ for all }x\in A.
\]

Choosing a cofinal ultrafilter $\mathcal{W}$ on the directed set $J$, and letting $\mathcal{V}$ be the Fubini ultrafilter $\mathcal{U} \otimes \mathcal{W}$, the operator $w:=(u_0\oplus u_j)_{j\to\mathcal{W}}$ can be viewed as a unitary in $\mathcal{B}(\mathcal H^{\mathcal{V}})$. The preceding norm convergence and cofinality of $\mathcal{W}$ give
\begin{equation}
wV^{\mathcal{V}}=V^{\mathcal{V}}, \qquad w\beta^{\mathcal{W}}(x)w^*=\alpha^{\mathcal{W}}(x) \quad\text{for every }x\in A. \label{eq:ExactIntertwining}
\end{equation}
In particular, since $\alpha$ and $\beta$ agree on $A_1$ by \eqref{eq:AlphaBetaIdentity}, the unitary $w$ commutes with $\pi_{1}^{\mathcal{V}}(\Delta(a))$ for every $a\in A_1$.

Denote by $\widetilde\Delta\colon A_{1}^{\mathcal{U}}\hookrightarrow A_{1}^{\mathcal{V}}$ the diagonal embedding into the outer ultrapower, and set $\Theta:=\widetilde\Delta\circ\theta$. Then, $\Theta|_{A_1}=\Delta$ and $\beta^{\mathcal{W}}=\pi_{1}^{\mathcal{V}}\circ \Theta$. Furthermore, set $p:=w^*P^{\mathcal{V}}w$ and let $\pi_2$ be the restriction of $\Theta$ to $A_2$. Since $\theta(A_2)\subseteq B^{\mathcal{U}}$ and $\widetilde\Delta(B^{\mathcal{U}})\subseteq B^{\mathcal{V}}$, this defines a unital $\ast$-embedding $\pi_2\colon A_2\hookrightarrow B^{\mathcal{V}}$.

The identities in \eqref{eq:Identities} pass to $P^{\mathcal{V}}$, so that for every $a_1 \in A_1$ and $a_2 \in A_2$,
\begin{align*}
P^{\mathcal{V}}\,\alpha^{\mathcal{W}}(a_1)\,P^{\mathcal{V}} &=\bigl((\alpha^{\mathcal{W}}\circ\mathbb{E}_1)(a_1)\bigr) P^{\mathcal{V}},\\
(P^{\mathcal{V}})^\perp\,\alpha^{\mathcal{W}}(a_2)\, (P^{\mathcal{V}})^\perp &= \bigl((\alpha^{\mathcal{W}}\circ\mathbb{E}_2)(a_2)\bigr) (P^{\mathcal{V}})^\perp,\\
P^{\mathcal{V}}\alpha^{\mathcal{W}}(a_1)V^{\mathcal{V}} &= 0.
\end{align*}
Conjugating these identities by $w^*$, and using \eqref{eq:ExactIntertwining} together with $p^\perp=w^*(P^{\mathcal{V}})^\perp w$, yields
\[
\begin{aligned}
p\,\bigl((\pi_{1}^{\mathcal{V}}\circ\Delta)(a_1)\bigr)\,p &=\bigl((\pi_{1}^{\mathcal{V}}\circ\Delta \circ\mathbb{E}_1)(a_1)\bigr)p,\\
p^\perp\,\bigl((\pi_{1}^{\mathcal{V}}\circ\pi_2)(a_2)\bigr)\,p^\perp &=\bigl((\pi_{1}^{\mathcal{V}}\circ\Delta \circ\mathbb{E}_2)(a_2)\bigr)p^\perp,\\
p\,\bigl((\pi_{1}^{\mathcal{V}}\circ\Delta)(a_1)\bigr) \,V^{\mathcal{V}} &=0
\end{aligned}
\]
for every $a_1\in A_1$ and $a_2\in A_2$. Moreover, $\pi_2|_D=\Delta|_D$ and $V^*\pi_1(a)V=(\rho\circ\mathbb{E}_1)(a)$ for every $a\in A_1$. Thus the unital $\ast$-embeddings $\pi_1\colon A_1\hookrightarrow\mathcal{B}(\mathcal H)$ and $\pi_2\colon A_2\hookrightarrow B^{\mathcal{V}}\subseteq A_{1}^{\mathcal{V}}$, the projection $p\in\mathcal{B}(\mathcal H^{\mathcal{V}})$, and the isometry $V\colon\mathcal K\to\mathcal H$ satisfy conditions (a), (b), and (c).
\end{proof}

The proof of the second implication of Theorem \ref{MainTheorem} follows from the following proposition employing Hasegawa's construction in \cite{Hasegawa18} (see Subsection \ref{Subsec:amalgamated}).

\begin{proposition} \label{Direction2} Let $B\subseteq A_{1}$ and $D\subseteq A_{1},A_{2}$ be unital inclusions of \Cs-algebras admitting non-degenerate conditional expectations $\mathbb{E}_{1}\colon A_{1}\to D$ and $\mathbb{E}_{2}\colon A_{2}\to D$. Assume that there exist Hilbert spaces $\mathcal{H}$ and $\mathcal{K}$, an ultrafilter $\mathcal{U}$, unital $\ast$-embeddings $\pi_{1}\colon A_{1} \hookrightarrow \mathcal{B}(\mathcal{H})$, $\pi_{2}\colon A_{2}\hookrightarrow B^{\mathcal{U}}\subseteq A_{1}^{\mathcal{U}}$ and $\rho\colon D\hookrightarrow\mathcal{B}(\mathcal{K})$, an isometry $V\colon\mathcal{K}\to\mathcal{H}$, and a projection $p\in\mathcal{B}(\mathcal{H}^{\mathcal{U}})$ satisfying the following conditions:

\begin{enumerate}[label=\textup{(\alph*)},leftmargin=*]
\item The embedding $\pi_{2}$ restricts to $\Delta|_{D}$ on $D$.
\item For every $a\in A_{1}$ the identity $V^{\ast}\pi_{1}(a)V=(\rho\circ\mathbb{E}_{1})(a)$ holds.
\item For every $a_{1}\in A_{1}$ and $a_{2}\in A_{2}$,
\begin{align*}
p\,\bigl((\pi_{1}^{\mathcal{U}}\circ\Delta)(a_{1})\bigr) \,p & =\bigl((\pi_{1}^{\mathcal{U}}\circ\Delta\circ\mathbb{E}_{1})(a_{1})\bigr)\,p,\\
p^\perp\,\bigl((\pi_{1}^{\mathcal{U}}\circ\pi_{2})(a_{2})\bigr)\, p^\perp & =\bigl((\pi_{1}^{\mathcal{U}}\circ\Delta\circ\mathbb{E}_{2})(a_{2})\bigr)\, p^\perp,\\
p\,\bigl((\pi_{1}^{\mathcal{U}}\circ\Delta)(a_{1})\bigr)\,V^{\mathcal{U}} & =0,
\end{align*}
where $V^{\mathcal{U}}\colon\mathcal{K}^{\mathcal{U}}\to\mathcal{H}^{\mathcal{U}}$ is the isometry induced by $V$.
\end{enumerate}
Then there exists a unital \Cs-embedding  $\theta\colon(A_{1},\mathbb{E}_{1})\star_{D}(A_{2},\mathbb{E}_{2})\hookrightarrow A_{1}^{\mathcal{U}}$ such that
\[
\theta|_{A_{1}}=\Delta,\quad\mathbb{E}_{1}^{\mathcal{U}}\circ\theta=\Delta \circ(\mathbb{E}_{1}\star\mathbb{E}_{2}),\quad\theta(A_{2})\subseteq B^{\mathcal{U}}.
\]
\end{proposition}

\begin{proof}
Set $A:=(A_{1},\mathbb{E}_{1})\star_{D}(A_{2},\mathbb{E}_{2})$, $\mathbb{E}:=\mathbb{E}_{1}\star\mathbb{E}_{2}$, and write $\sigma_{1}:=\pi_{1}^{\mathcal{U}}\circ\Delta$, $\sigma_{2}:=\pi_{1}^{\mathcal{U}}\circ\pi_{2}$. By the first assumption, the representations $\sigma_{1}$ and $\sigma_{2}$ agree on $D$. The first two identities in (c), together with the universal property of Hasegawa's ambient \Cs-algebra $\Delta\mathbf{T}(A,\mathbb{E})$ (see Subsection \ref{Subsec:amalgamated}) therefore yield a unital $\ast$-representation $\sigma\colon A\to\mathcal{B}(\mathcal{H}^{\mathcal{U}})$ extending $\sigma_{1}$ and $\sigma_{2}$, whose image is contained in $\pi_{1}^{\mathcal{U}}(A_{1}^{\mathcal{U}})$. Since $\pi_{1}^{\mathcal{U}}$ is faithful, we obtain a unital $\ast$-homomorphism $\theta:=(\pi_{1}^{\mathcal{U}})^{-1}\circ\sigma\colon A\rightarrow A_{1}^{\mathcal{U}}$ with $\theta|_{A_{1}}=\Delta$ and $\theta|_{A_{2}}=\pi_{2}$, so that in particular $\theta(A_{2})\subseteq B^{\mathcal{U}}$.

We claim that
\[
(V^{\mathcal{U}})^{\ast}\sigma(x)V^{\mathcal{U}}=(\rho^{\mathcal{U}}\circ\Delta\circ\mathbb{E})(x)\qquad\text{ for every }x\in A.
\]
Indeed, on $A_{1}$ this follows from (b). To verify the claim on words involving $A_{2}$, consider an element of the form $x:=a_{0}b_{1}a_{1}\cdots b_{m}a_{m}\in A$, where $m\geq1$, $a_{1},\ldots,a_{m-1}\in A_{1}^{\circ}$ , $b_{1},\ldots,b_{m}\in A_{2}^{\circ}$, and $a_{0},a_{m}\in A_{1}$. The   assumptions in (c) give $\sigma(a_{m})V^{\mathcal{U}}\mathcal{K}^{\mathcal{U}}\subseteq p^{\perp}\mathcal{H}^{\mathcal{U}}$ and
\begin{eqnarray*}
\sigma(a_{i})p\mathcal{H}^{\mathcal{U}}\subseteq p^{\perp}\mathcal{H}^{\mathcal{U}} \quad&  \text{ for }1\leq i<m,\\
\;\sigma(b_{i})p^{\perp}\mathcal{H}^{\mathcal{U}}\subseteq p\mathcal{H}^{\mathcal{U}}\quad & \text{ for }1\leq i \leq m.
\end{eqnarray*}
Applying these inclusions successively shows that $\sigma(b_{1}a_{1}\cdots b_{m}a_{m})V^{\mathcal{U}}\mathcal{K}^{\mathcal{U}}\subseteq p\mathcal{H}^{\mathcal{U}}$. Moreover, $(V^{\mathcal{U}})^{\ast}\sigma(a_{0})p=0$, by the last identity in (c). Thus, $(V^{\mathcal{U}})^{\ast}\sigma(x)V^{\mathcal{U}}=0=(\rho^{\mathcal{U}}\circ\Delta\circ\mathbb{E})(x)$. Since $A_{1}$ and words of the form above linearly span a dense subspace of $A$, the claim follows.

From (b) we obtain $(V^{\mathcal{U}})^{\ast}\pi_{1}^{\mathcal{U}}(a)V^{\mathcal{U}}=(\rho^{\mathcal{U}}\circ\mathbb{E}_{1}^{\mathcal{U}})(a)$ for all $a\in A_{1}^{\mathcal{U}}$. Combining this with the previous claim and using the faithfulness of $\rho^{\mathcal{U}}$, we get $\mathbb{E}_{1}^{\mathcal{U}}\circ\theta=\Delta\circ\mathbb{E}.$ It remains to show that $\theta$ is faithful. For this, let $x \in A$ with $\theta(x)=0$. For every $y\in A$,
\[
\Delta(\mathbb{E}(y^{*}x^{*}xy))=\mathbb{E}_{1}^{\mathcal{U}}(\theta(xy)^{*}\theta(xy))=0
\]
and hence $\pi_{\mathbb{E}}(x)=0$. The non-degeneracy of $\mathbb{E}$ implies $x=0$; thus $\theta$ is injective. \end{proof}

The assertion of Theorem \ref{AbstractCriterion-general} can also be deduced from Proposition \ref{Direction1} and Proposition \ref{Direction2}.

\begin{proof}[Proof of Theorem \ref{AbstractCriterion-general}] For the implication ``$\text{(i)}\Rightarrow\text{(ii)}$'' suppose that the inclusion $B\subseteq(A,\rho)$ is selfless. Then there exist a non-trivial \Cs-probability space $(C,\gamma)$, an ultrafilter $\mathcal{W}$, and a state-preserving embedding $\Phi \colon (A,\rho)\star(C,\gamma) \hookrightarrow(A^{\mathcal{W}},\rho^{\mathcal{W}})$ for which $\Phi|_A=\Delta$ and $\Phi(C)\subseteq B^{\mathcal{W}}$ holds. We apply Proposition~\ref{Direction1} with $A_1 :=A$, $A_2:=C$, $D:=\mathbb C$, and $\mathbb{E}_1(a):=\rho(a)1$ for $a\in A$, $\mathbb{E}_2(c):=\gamma(c)1$ for $c\in C$, to obtain a Hilbert space $\mathcal H$, an ultrafilter $\mathcal{U}$, unital $\ast$-embeddings $\pi\colon A\hookrightarrow\mathcal{B}(\mathcal H)$, $\theta\colon C\hookrightarrow B^{\mathcal{U}}$, an isometry $V\colon\mathbb C\to\mathcal H$, and a projection $p\in\mathcal{B}(\mathcal H^{\mathcal{U}})$ with
\begin{eqnarray*}
V^*\pi(a)V=\rho(a)1, \quad p\,((\pi^{\mathcal{U}} \circ \Delta)(a))\,p =\rho(a)p, \quad p\,((\pi^{\mathcal{U}}\circ \Delta)(a))\,V^{\mathcal{U}}=0
\end{eqnarray*}
for all $a\in A$ and
\begin{eqnarray*}
p^\perp\,((\pi^{\mathcal{U}}\circ \theta)(c))\,p^\perp =\gamma(c)p^\perp \quad
\end{eqnarray*}
for all $c \in C$. Then $\xi :=V(1)$ is a unit vector implementing $\rho$ through $\pi$ with $p\,((\pi^{\mathcal{U}}\circ \Delta)(a))\,\xi^{\mathcal{U}}=0$ for $a \in A$. In particular, $p\xi^{\mathcal{U}}=0$ and hence
\[
\begin{aligned}
\rho^{\mathcal U}(\theta(c))
&=\langle \xi^{\mathcal U},
   (\pi^{\mathcal U}\circ\theta)(c)\xi^{\mathcal U}\rangle\\
&=\langle \xi^{\mathcal U},
   p^\perp(\pi^{\mathcal U}\circ\theta)(c)
   p^\perp\xi^{\mathcal U}\rangle
 =\gamma(c).
\end{aligned}
\]
for $c\in C$. Thus $\theta$ is state preserving, and all assertions in (ii) follow.

For the converse implication ``$\text{(ii)}\Rightarrow\text{(i)}$'' suppose that (ii) holds and define an isometry $V\colon\mathbb C\to\mathcal H$ by $V(z)=z\xi$ for $z \in \mathbb{C}$. With the same scalar-valued conditional expectations as above, the conditions of Proposition~\ref{Direction2} are precisely the assumed identities, together with $V^*\pi(a)V=\rho(a)1_{\mathbb C}$ for $a\in A$. This yields a unital $\ast$-embedding $\Phi\colon(A,\rho)\star(C,\gamma)\hookrightarrow A^{\mathcal{U}}$ satisfying
\[
\Phi|_A=\Delta,\qquad \rho^{\mathcal{U}}\circ\Phi=\rho\star\gamma,\qquad \Phi(C)\subseteq B^{\mathcal{U}}.
\]
Since $C\neq\mathbb C1$, this witnesses the selflessness of $B\subseteq(A,\rho)$.
\end{proof}

Using similar techniques as in the previous proofs, we can also give a proof of Theorem \ref{AbstractCriterion-completely-selfless}. Note that our proof does not use Ozawa's approach via the universal property of free products with the Toeplitz algebra.
We need an intermediate lemma first.

\begin{lemma} \label{lem:nuclear-probability-subspace}
    Let $(C,\gamma)$ be a non-trivial \Cs-probability space. Then there exists a non-trivial unital commutative \Cs-subalgebra $C_0\subseteq C$ such that the restriction $\gamma|_{C_0}$ is faithful.
\end{lemma}

\begin{proof}
    Choose a non-scalar, selfadjoint element $h\in C$, and denote its spectrum by $X$.
    We claim that there exists a sequence $(c_n)_{n\in\mathbb N}$ of non-zero positive elements of $\Cs(h)$ such that, for every non-zero positive element $ d\in \Cs(h)$, there exists a positive integer $n\in\mathbb N$ with $t c_n\leq d$ for some positive real number $t>0$. Indeed, since $X$ is compact metrizable, it admits a countable basis $(U_n)_{n\in\mathbb N}$ of non-empty open subsets. For each $n$, choose a non-zero function $f_n\in C(X)$ with $0\leq f_n\leq1$ and $\operatorname{supp}(f_n)\subseteq U_n$,
    and set $c_n:=f_n(h)$.
    To check the required property, write $d=g(h)$ for some non-zero positive function $g\in C(X)$, and choose $t>0$, $n\in\mathbb N$ such that
    \[
        U_n\subseteq\{x\in X\mid g(x)>t\}.
    \]
    It follows that $tf_n(x)\le g(x)$ for all $x\in X$, and thus $tc_n\le d$.

    Let $(\pi_\gamma,L^2(C,\gamma),\xi_\gamma)$ be the
GNS-representation of $(C,\gamma)$. For each $n\in\mathbb N$,
the set
\[
    N_n:=\{e\in C\mid\gamma(e^*c_ne)=0\}
\]
is the kernel of the bounded linear map $C\rightarrow\mathcal H_\gamma$, $e\mapsto \pi_\gamma(c_n)^{1/2}\pi_\gamma(e)\xi_\gamma$. Thus $N_n$ is a closed linear subspace of $C$. Furthermore, $N_n\neq C$. Indeed, if $N_n=C$ the
cyclicity of $\xi_\gamma$ implies
$\pi_\gamma(c_n)=0$ and hence $c_n=0$ by the faithfulness of $\pi_\gamma$, leading to a contradiction. 

    Next, observe that every non-empty norm-open subset of $\mathcal U(C)$ linearly spans $C$. Indeed, by translation, it suffices to check this for a neighbourhood of $1$. For every selfadjoint contraction $a\in C$ and $0<t<1$, the elements  $w_t^\pm=(1-t^2a^2)^{1/2}\pm ita$ are unitaries converging in norm to $1$ as $t\to0$, and $a=\frac{w_t^+-w_t^-}{2it}$.
    Hence every neighbourhood of $1$ linearly spans $C$, proving the observation.

    Since each $N_n$ is a proper closed linear subspace, it follows that $N_n\cap\mathcal U(C)$ is closed and has empty interior in $\mathcal U(C)$. The latter is complete in the norm metric, so Baire's theorem provides a unitary $u\in\mathcal U(C)\setminus\bigcup_{n\in\mathbb N}N_n$.
    This unitary in particular satisfies $\gamma(u^*c_nu)>0$ for every $n\in\mathbb N$. Given a non-zero positive element $d\in \Cs(h)$, find $n\in\mathbb N$ and $t>0$ with $tc_n\leq d$. Then
    \[
        \gamma(u^*du)\geq t\gamma(u^*c_nu)>0.
    \]
    Therefore, $\gamma$ is faithful on the non-trivial unital commutative \Cs-subalgebra $C_0:=u^*\Cs(h)u\subseteq C$, which concludes the proof.
\end{proof}

\begin{proof}[Proof of Theorem~\ref{AbstractCriterion-completely-selfless}]
We prove the implications in the following order: ``$\textup{(i)}\Rightarrow\textup{(iii)}\Rightarrow\textup{(ii)}\Rightarrow\textup{(i)}$''.

Let us denote by $\cal{C}$ the \Cs-algebra generated by $(S+S^*) \in\cal{B}(\ell^2)$, where $S$ is the unilateral shift on $\ell^2$, and let $\tau_\cal{C}$ be the state associated with the vector $\delta_0$. Then, $\cal{C}\cong C([-2,2])$ and under this identification $\tau_\cal{C}$ is the state with semicircular law.\\

For the implication ``$\textup{(i)}\Rightarrow\textup{(iii)}$'', suppose that the inclusion $B\subseteq(A,\rho)$
is completely selfless. By the characterisations of selfless (inclusions) from \cite{Rob25,HKEPR25}, and the fact that completely selfless embeddings are preserved by free products (see \cite[Lemma 10]{Ozawa25}), and satisfy the same permanence properties as existential embeddings, we deduce that there exists a completely existential embedding $(A,\rho)\to (A,\rho)\star(C(\mathbb T),\leb)$ that sends $C(\mathbb T)$ into $B^{\mathcal W}$ for some ultrafilter $\mathcal W$. Hence, we may use a state-preserving embedding $(\mathcal C,\tau_{\mathcal C}) \hookrightarrow(C(\mathbb T),\leb)$, and obtain an embedding $\Phi\colon(A,\rho)\star(\mathcal C,\tau_{\mathcal C})
\hookrightarrow(A^{\mathcal W},\rho^{\mathcal W})$ such that $\Phi|_A=\Delta_A$,
$\Phi(\mathcal C)\subseteq B^{\mathcal W}$ and such that the induced algebraic $\ast$-homomorphism
\[
\Phi'\colon
((A,\rho)\star(\mathcal C,\tau_{\mathcal C}))
\odot\mathcal B(\ell^2)
\longrightarrow
(A\otimes\mathcal B(\ell^2))^{\mathcal W}
\]
is continuous with respect to the minimal tensor norm; we denote its continuous extension by $\Phi'$ as well. Note that by our assumption, we have that $\Phi(\cal{C})\subseteq B^{\mathcal W}$. Put $\mathbb E_1:=\rho\otimes\mathrm{id}_{\mathcal{B}(\ell^2)}$ and $\mathbb E_2:=\tau_\cal{C}\otimes\mathrm{id}_{\mathcal{B}(\ell^2)}$. We use the canonical expectation-preserving identification
\begin{equation}\label{eq:complete-free-product-tensor}
    (A\otimes\cal{B}(\ell^2),\mathbb E_1)\star_{\cal{B}(\ell^2)}(\cal{C}\otimes\cal{B}(\ell^2),\mathbb E_2)\cong(((A,\rho)\star(\cal{C},\tau_{\cal{C}}))\otimes \cal{B}(\ell^2),(\rho\star\tau_\cal{C})\otimes\mathrm{id}_{\mathcal{B}(\ell^2)}),
\end{equation}
which restricts to the canonical embeddings of $A\otimes\cal{B}(\ell^2)$ and $\cal{C}\otimes\cal{B}(\ell^2)$ inside
$((A,\rho)\star(\mathcal C,\tau_{\mathcal C}))
\otimes\mathcal B(\ell^2)$. The continuous extension of $\Phi'$, viewed on this
amalgamated free product, restricts to the diagonal embedding on $A\otimes\mathcal B(\ell^2)$ and satisfies
\[
\mathbb E_1^{\mathcal W}\circ\Phi'
=\Delta_{\mathcal B(\ell^2)}
 \circ\bigl((\rho\star\tau_{\mathcal C})
       \otimes\mathrm{id}_{\mathcal B(\ell^2)}\bigr),
\]
where $\Delta_{\mathcal{B}(\ell^2)}$ is the diagonal embedding of $\mathcal{B}(\ell^2)$ inside $\mathcal{B}(\ell^2)^\mathcal{W}$. Indeed, if $\Phi(x)=(a_i)_{i\to\mathcal W}$ for $x\in (A,\rho)\star(\cal{C},\tau_\cal{C})$, then
\[
    \mathbb E_1^{\mathcal W}(\Phi'(x\otimes d))=(\rho(a_i)d)_{i\to\mathcal W}=\Delta_{\cal{B}(\ell^2)}((\rho\star\tau_\cal{C})(x)d)
\]
for all $d\in\cal{B}(\ell^2)$. Now, the expectation $\mathbb E:=(\rho\star\tau_\cal{C})\otimes\mathrm{id}_{\mathcal{B}(\ell^2)}$ is clearly non-degenerate. It follows that, if $\Phi'(x)=0$ for some $x\in ((A,\rho)\star(\cal{C},\tau_\cal{C}))\otimes\cal{B}(\ell^2)$, then
\[
    \Delta_{\cal{B}(\ell^2)}(\mathbb E(y^*x^*xy))=\mathbb E_1^{\mathcal W}(\Phi'(y)^*\Phi'(x)^*\Phi'(x)\Phi'(y))=0
\]
for all $y\in ((A,\rho)\star(\cal{C},\tau_\cal{C}))\otimes\cal{B}(\ell^2)$, and hence $x=0$. This proves that $\Phi'$ is injective.

Let $\mathcal F$ be the GNS-Hilbert space of $(A,\rho)\star(\cal{C},\tau_\cal{C})$ with respect to $\rho \star \tau_\mathcal{C}$, and denote the associated  canonical cyclic vector by $\Omega$. We identify $(A,\rho)\star(\cal{C},\tau_\cal{C})$ with its image in $\cal{B}(\cal{F})$, write $\cal{K}_1:=L^2(A,\rho)$, $\cal{K}_2:=L^2(\cal{C},\tau_\cal{C})$, and set $\cal{K}_1^\circ:=\cal{K}_1 \ominus \mathbb{C}\xi_\rho $, $\cal{K}_2^\circ:=\cal{K}_2 \ominus\mathbb{C} \xi _{\tau_\mathcal{C}}$.
Then, $\cal{F}$ decomposes via $\cal{F}=\mathbb C\Omega \oplus \cal{F}_1^\circ\oplus \cal{F}_2^\circ$, where 
\[
\cal{F}_i^\circ := \bigoplus_{n\ge1} \, \bigoplus_{i=i_1\neq\dots\neq i_n} \cal{K}_{i_1}^\circ\otimes \dots\otimes\cal{K}_{i_n}^\circ.
\]
Furthermore, there exists a unitary $U\colon \cal{F} \to L^2(\cal{C},\tau_\cal{C})\otimes (\mathbb C \Omega \oplus \cal{F}_1^\circ)$, which acts trivially on $\mathbb C \Omega\oplus \cal{F}_1^\circ$, and which canonically maps $\cal{F}_2^\circ$ to $\mathcal{K}_2 ^\circ \otimes (\Omega \mathbb{C} \oplus \cal{F}_1^\circ)$. 
Note also that $L^2(\cal{C},\tau_\cal{C})\cong\ell^2$ canonically, and under this identification $\tau_\cal{C}$ coincides with the vector state associated with $\delta_0 \in \ell^2$. 
The range of $S\otimes \id_{\cal{B}(\ell^2)}$ is $(\ell^{2}\ominus\mathbb{C}\delta_{0})\otimes(\mathbb{C}\Omega\oplus\mathcal{F}_{1}^{\circ})=U(\mathcal{F}_{2}^{\circ})$. Thus, denoting by $L$ the operator in $\cal{B}(\cal{F})$ given by $U^*(S\otimes \id_{\mathbb{C}\Omega \oplus \mathcal{F}_1 ^\circ})U
$, we have that $L+L^*$ is the generator of $\cal{C}$, and 
\begin{equation}\label{eq:complete-canonical-isometry}
   L^*aL=\rho(a)1 \; \text{ and } \; L^*a\Omega=0 \quad \text{ for all } a \in A.
\end{equation}
For the first identity, observe that $a-\rho(a)1_A$ maps the range of $L$ into its orthogonal complement, while the second identity follows because $A\Omega$ is orthogonal to that range.

Let $I$ be the index set of $\mathcal W$. Choose an infinite cardinal $\mu\geq\max\{|I|,\dim(\mathcal F), \aleph_0\}$, put $\kappa:=2^\mu$, and choose a set $J$ of cardinality $\kappa$ together with $j_0\in J$. The triple $(\pi,\mathcal H,\xi)$ with $\mathcal H:=\mathcal F\otimes\ell^2(J)$, $\pi:= \id_A \otimes 1_{\mathcal{B}(\ell^2(J))}$ and $\xi:=\Omega\otimes\delta_{j_0}$ is a unital faithful representation of $(A,\rho)$. We define $V\colon \ell^2 \to \cal{H}\otimes\ell^2$ via $V\eta:=\xi\otimes \eta$ for $\eta \in \ell^2$, and observe that $V^*xV=\mathbb E(x)$ for every $x\in ((A,\rho)\star(\cal{C},\tau_\cal{C}))\otimes\cal{B}(\ell^2)$.

Define two faithful unital $\ast$-representations of $((A,\rho)\star(\cal{C},\tau_\cal{C}))\otimes\cal{B}(\ell^2)$ on $(\cal{H}\otimes\ell^2)^{\mathcal W}$ by $\alpha(x):=(x)_{i\to\mathcal W}$ and $\beta(x):=(\pi\otimes\mathrm{id}_{\mathcal B(\ell^2)})^{\mathcal W}(\Phi'(x))$ for $x \in ((A,\rho)\star(\cal{C},\tau_\cal{C}))\otimes\cal{B}(\ell^2)$. These representations agree on $A\otimes\cal{B}(\ell^2)$. Furthermore, expectation preservation gives
\begin{equation}\label{eq:complete-common-compression}
    (V^{\mathcal W})^*\alpha(x)V^{\mathcal W}=\mathbb E(x)^{\mathcal W}=(V^{\mathcal W})^*\beta(x)V^{\mathcal W}
\end{equation}
for all $x\in((A,\rho)\star(\cal{C},\tau_\cal{C}))\otimes\cal{B}(\ell^2)$. The subspaces
\begin{align*}
    \mathcal{H}_\alpha:&= \overline{\operatorname{span}}\{\alpha(x)V^{\mathcal W}\eta \mid x\in ((A,\rho)\star(\cal{C},\tau_\cal{C}))\otimes\cal{B}(\ell^2),\ \eta \in (\ell^2)^{\mathcal W}\},\\
    \mathcal H_\beta:&= \overline{\operatorname{span}}\{\beta(x)V^{\mathcal W}\eta \mid x\in ((A,\rho)\star(\cal{C},\tau_\cal{C}))\otimes\cal{B}(\ell^2),\ \eta \in (\ell^2)^{\mathcal W}\}
\end{align*}
reduce the respective representations $\alpha$ and $\beta$. An application of \eqref{eq:complete-common-compression} to $y^*x$ shows that the assignment $u_0\colon\mathcal H_\beta\to\mathcal H_\alpha$ given by $\beta(x)V^{\mathcal W}\eta\mapsto\alpha(x)V^{\mathcal W} \eta$ for $x\in ((A,\rho)\star(\cal{C},\tau_\cal{C}))\otimes\cal{B}(\ell^2)$, $\eta \in (\ell^2)^\mathcal{W}$
extends linearly to a well-defined unitary. It satisfies $u_0V^{\mathcal W}=V^{\mathcal W}$ and intertwines the restrictions of $\beta$ and $\alpha$ to these subspaces. In the same way as in the proof of Proposition \ref{Direction1}, one verifies that for all $0 \neq x\in((A,\rho)\star(\cal{C},\tau_\cal{C}))\otimes\cal{B}(\ell^2)$,
\begin{eqnarray*}
\operatorname{rank}(\alpha(x)|_{\mathcal H_\alpha^\perp})=\operatorname{rank}(\beta(x)|_{\mathcal H_\beta^\perp})=\kappa.
\end{eqnarray*}
Theorem~\ref{HadwinTheorem} hence gives a net of unitaries $(u_\lambda)_{\lambda \in \Lambda} \subseteq \mathcal{B}(\mathcal H_\beta^\perp,\mathcal H_\alpha^\perp)$ with $(u_0\oplus u_\lambda) V^{\mathcal W}=V^{\mathcal W}$ and
$\lim_{\lambda\in \Lambda}\|(u_0\oplus u_\lambda)\beta(x)(u_0\oplus u_\lambda)^\ast -\alpha(x)\|=0$ for all $x\in((A,\rho)\star(\cal{C},\tau_\cal{C}))\otimes\cal{B}(\ell^2)$.

Choose a cofinal ultrafilter $\mathcal V$ on the directed set $\Lambda$ and put $\mathcal U:=\mathcal W\otimes\mathcal V$. Under the identification $((\cal{H}\otimes\ell^2)^{\mathcal W})^{\mathcal V}=(\cal{H}\otimes\ell^2)^{\mathcal U}$, the family $(u_0 \oplus u_\lambda)_{\lambda \in \Lambda}$ induces a unitary $w\in\mathcal B((\cal{H}\otimes\ell^2)^{\mathcal U})$. Let $\widetilde\Delta\colon (A\otimes\cal{B}(\ell^2))^{\mathcal W}\to (A\otimes\cal{B}(\ell^2))^{\mathcal U}$ be the diagonal embedding into the outer ultrapower, and put $\Theta:=\widetilde\Delta\circ\Phi'$. The preceding identities give
\begin{equation}\label{eq:complete-intertwining}
    w\,(\pi\otimes\mathrm{id}_{\mathcal B(\ell^2)})^{\mathcal U}
   (\Theta(x))\,w^*=x^{\mathcal U}\quad\text {and} \quad  wV^{\mathcal U}=V^{\mathcal U}
\end{equation}
for all $x\in((A,\rho)\star(\cal{C},\tau_\cal{C}))\otimes\cal{B}(\ell^2)$.
Since $\Theta|_{A\otimes\cal{B}(\ell^2)}=\Delta$, the unitary $w$ commutes with $(\pi(a)\otimes d)_{i\to\mathcal{U}}$ for all $a\in A$, $d\in \mathcal{B}(\ell^2)$.
Set now $T:=w^*(L\otimes\id_{\ell^2(J)} \otimes\id_{\ell^2})^{\mathcal U}w$, and fix a unit vector $\eta\in \ell^2$. Then $T$ is an isometry with $[T,(\id_\mathcal{H} \otimes d)_{i\to\mathcal{U}}]=0$ for every $d\in \mathcal{B}(\ell^2)$. Moreover, \eqref{eq:complete-canonical-isometry} and \eqref{eq:complete-intertwining} give
\begin{align*}
T^*\,\bigl(((\pi\otimes 1_{\cal{B}(\ell^2)})^{\mathcal{U}}\circ \Delta)(a)\bigr)\,T &= \rho(a)\,1, \\
T^*\,\bigl(((\pi\otimes 1_{\cal{B}(\ell^2)})^{\mathcal{U}}\circ \Delta)(a)\bigr)\,(\xi\otimes \eta)^\mathcal{U} &= 0, 
\end{align*}
for all $a\in A$. Writing $t=S+S^*\in\mathcal C$, we also have
\[
T+T^*
=(\pi\otimes\mathrm{id}_{\mathcal B(\ell^2)})^{\mathcal U}
   \bigl(\Theta(t\otimes1_{\mathcal B(\ell^2)})\bigr)
\in
(\pi\otimes1_{\mathcal B(\ell^2)})^{\mathcal U}
   (B^{\mathcal U}).
\]
Indeed, $\Phi'(t\otimes1_{\mathcal B(\ell^2)})$ is the
coordinatewise image of $\Phi(t)\in B^{\mathcal W}$, and
the outer diagonal embedding preserves this image. Thus all assertions in \textup{(iii)} hold. \\

For the implication ``$\textup{(iii)}\Rightarrow\textup{(ii)}$'', suppose that \textup{(iii)} holds. Fix any unit vector $\eta\in\ell^2$, and set $\zeta:=\xi\otimes\eta \in \mathcal{H} \otimes \ell^2$, and $p:=TT^*$.
For every $b\in A^{\mathcal U}$ we then have
\begin{equation}\label{eq:complete-vector-state}
    \langle \zeta^\mathcal{U},(\pi\otimes 1_{\mathcal{B}(\ell^2)})^{\mathcal U}(b)\zeta^\mathcal{U}\rangle
=\rho^{\mathcal U}(b).
\end{equation}
The last condition in \textup{(iii)}, applied to $a=1$, gives $T^*\zeta^{\mathcal{U}}=0$, and hence $p^{\perp}\zeta^{\mathcal{U}}=\zeta^{\mathcal{U}}$. Let $b\in B^{\mathcal U}$ be the unique selfadjoint element satisfying $(\pi\otimes1_{\mathcal B(\ell^2)})^{\mathcal U}(b)
  =T+T^*$. Faithfulness of the amplified representation gives $\|b\|\leq2$, so that continuous functional calculus defines
a unital $\ast$-homomorphism
$\theta_0\colon\mathcal C\to B^{\mathcal U}$
by $\theta_0(f):=f(b)$. The identities $T^*T=1$ and $p^{\perp}T=0=T^*p^{\perp}$ imply
\begin{equation}\label{eq:complete-shift-compression}
    p^{\perp}f(T+T^*)p^{\perp}=\tau_\cal{C}(f)p^{\perp}
\end{equation}
for all $f\in \cal{C}$.
Indeed, for polynomials this follows by reducing to words of the form $T^m(T^*)^n$ and comparing with the vacuum compression of $f(S+S^*)$; the general case follows by uniform approximation. Combining \eqref{eq:complete-vector-state} and \eqref{eq:complete-shift-compression}, we obtain
\[
    \rho^{\mathcal U}(\theta_0(f))
=\langle \zeta^\mathcal{U},f(T+T^*)\zeta^{\mathcal{U}}\rangle
=\tau_{\mathcal C}(f) \quad \text{ for } f\in \mathcal{C}. \]
Since $\tau_\cal{C}$ is faithful, $\theta_0$ is injective. Thus $(\cal{C},\tau_\cal{C})$ is a non-trivial \Cs-probability space with a state-preserving embedding into $(B^{\mathcal U},\rho^{\mathcal U})$. Furthermore,
\begin{align*}
   p\,\bigl(((\pi\otimes 1_{\cal{B}(\ell^2)})^{\mathcal{U}}\circ \Delta) (a)\bigr)\,p & =T\bigl(T^*(((\pi\otimes 1_{\cal{B}(\ell^2)})^{\mathcal{U}} \circ  \Delta) (a))T\bigr)T^*=\rho(a)p,\\
 p\,\bigl(((\pi\otimes 1_{\cal{B}(\ell^2)})^{\mathcal{U}}\circ \Delta)(a)\bigr)\,\zeta^\mathcal{U} & =T\bigl(T^*(((\pi\otimes 1_{\cal{B}(\ell^2)})^{\mathcal{U}} \circ \Delta)(a))\zeta^\mathcal{U}\bigr)=0.
\end{align*}
The commutation condition for $p$ follows from that for $T$, and the remaining compression identity is \eqref{eq:complete-shift-compression}. This proves \textup{(ii)}.\\

For the implication ``$\textup{(ii)}\Rightarrow\textup{(i)}$'', suppose that \textup{(ii)} holds, and write again $\mathbb E_1:=\rho\otimes\mathrm{id}_{\mathcal{B}(\ell^2)}$, $\mathbb E_2:=\gamma\otimes\mathrm{id}_{\mathcal{B}(\ell^2)}$, and $\zeta:=\xi\otimes\eta$, where $\eta\in\ell^2$ is a fixed unit vector. Both expectations are non-degenerate because $\rho$ and $\gamma$ are GNS-faithful. Note that the ranges of the $\ast$-homomorphisms $(\id_A\otimes \id_{\ell^2})^{\cal{U}} \circ \theta$ and $\Delta\circ(1_A\otimes\mathrm{id}_{\mathcal B(\ell^2)})$ commute inside  $(A\otimes\cal{B}(\ell^2))^{\mathcal U}$. By Lemma \ref{lem:nuclear-probability-subspace}, after passing to a non-trivial \Cs-probability subspace, we can assume that $C$ is Abelian, and thus nuclear. Therefore, there is a well-defined unital $\ast$-homomorphism on the spatial tensor product:
\[
    \theta'\colon C\otimes\cal{B}(\ell^2) \longrightarrow (A\otimes\cal{B}(\ell^2))^{\mathcal U},\quad \theta'(c\otimes d)=(\id_A\otimes \id_{\ell^2})^\mathcal{U}(\theta(c))\cdot \Delta(1_A\otimes d)
\]
for all $c\in C$ and $d\in\cal{B}(\ell^2)$.
Define
\[
    \sigma_1:=(\pi\otimes \id_{\mathcal{B}(\ell^2)})^{\mathcal U}\circ\Delta\colon A\otimes\cal{B}(\ell^2) \to\mathcal B((\mathcal H\otimes \ell^2)^{\mathcal U})
\]
and 
\[
    \sigma_2:=(\pi\otimes \id_{\mathcal{B}(\ell^2)})^{\mathcal U}\circ\theta'\colon C\otimes\cal{B}(\ell^2)\to\mathcal B((\mathcal H\otimes \ell^2)^{\mathcal U}).
\]
These representations agree on $\cal{B}(\ell^2)$. The commutation condition for $p$ and the compression identities in \textup{(ii)} give
\begin{align*}
    p\sigma_1(a\otimes d)p &=\rho(a)(\pi\otimes\id_{\mathcal{B}(\ell^2)})^{\cal{U}}(\Delta(1_A\otimes d))p, \\
    p^{\perp}\sigma_2(c\otimes d)p^{\perp} &=\gamma(c)(\pi\otimes\id_{\mathcal{B}(\ell^2)})^{\cal{U}}(\Delta(1_A\otimes d))p^{\perp}.
\end{align*}
By linearity and continuity, we obtain
\[
    p\sigma_1(x)p=(\pi\otimes\id_{\mathcal{B}(\ell^2)})^{\cal{U}}(\Delta(\mathbb E_1(x)))p,\qquad p^{\perp}\sigma_2(y)p^{\perp}=(\pi\otimes\id_{\mathcal{B}(\ell^2)})^{\cal{U}}(\Delta(\mathbb E_2(y)))p^{\perp}
\]
for all $x\in A\otimes\cal{B}(\ell^2)$ and $y\in C\otimes\cal{B}(\ell^2)$.
Using the notation $D:=(A\otimes\cal{B}(\ell^2),\mathbb E_1)\star_{\cal{B}(\ell^2)}(C\otimes\cal{B}(\ell^2),\mathbb E_2)$ and $\mathbb{E}:=\mathbb{E}_1 \star \mathbb{E}_2$, we can apply Hasegawa's universal property in \cite[Corollary~4.2.7]{Hasegawa18} to obtain a unital $\ast$-homomorphism $\widehat\sigma\colon\Delta\mathbf T(D,\mathbb E)\longrightarrow\mathcal B((\mathcal H\otimes \ell^2)^{\mathcal U})$ which restricts to $\sigma_1$ and $\sigma_2$ on the two factors and maps the canonical projection to $p$. The restriction $\sigma:=\widehat\sigma|_D$ takes values in $(\pi\otimes\id_{\cal{B}(\ell^2)})^\cal{U}((A\otimes\cal{B}(\ell^2))^{\mathcal U})$ because $D$ is generated by the embedded copies of $A\otimes\cal{B}(\ell^2)$ and $C\otimes\cal{B}(\ell^2)$, and both of their images lie in this \Cs-subalgebra. Composing with the inverse of $(\pi\otimes\id_{\mathcal{B}(\ell^2)})^{\mathcal U}$ on its range, and using the canonical identification in \eqref{eq:complete-free-product-tensor} with $(C,\gamma)$ in place of $(\cal{C},\tau_\cal{C})$, gives a unital $\ast$-homomorphism
\[
    \Psi
    :=\bigl((\pi\otimes
       \id_{\mathcal{B}(\ell^2)})^{\mathcal U}\bigr)^{-1}
       \circ\sigma\colon ((A,\rho)\star(C,\gamma))\otimes \cal{B}(\ell^2)\longrightarrow (A\otimes\cal{B}(\ell^2))^{\mathcal U},
\]
which restricts to $\Delta$ on $A\otimes \cal{B}(\ell^2)$ and to $\theta'$ on $C\otimes \cal{B}(\ell^2)$.
The image $\Psi((A,\rho)\star(C,\gamma)\otimes \id_{\ell^2})$ lies in $(\id_A\otimes \id_{\ell^2})^{\mathcal U}(A^{\mathcal U})$, since this holds on the generating copies of $A$ and $C$. We may therefore define
\[
    \Phi\colon (A,\rho)\star(C,\gamma)\longrightarrow A^{\mathcal U},\qquad \Phi(x):=((\id_A\otimes \id_{\ell^2})^{\mathcal U})^{-1}(\Psi(x\otimes \id_{\ell^2})).
\]
This is a unital $\ast$-homomorphism with $\Phi|_A=\Delta_A$ and $\Phi|_C=\theta$. In particular, the latter condition also implies that $\Phi(C)\subseteq B^{\mathcal U}$.
We now verify that $\Phi$ is state-preserving. For this, write $\alpha=(\pi\otimes \id_{\ell^2})^{\mathcal U}\circ \Delta_A$ and $\beta=(\pi\otimes \id_{\ell^2})^{\mathcal U}\circ\theta$. As in \eqref{eq:complete-vector-state}, $\rho^{\mathcal U}(b)
=\langle \zeta^{\mathcal{U}},
   (\pi\otimes 1_{\cal{B}(\ell^2)})^{\mathcal U}(b)\zeta^{\mathcal{U}}\rangle$ for every $b\in A^{\mathcal U}$. The last condition in \textup{(ii)} gives $\alpha(A)\zeta\subseteq p^{\perp}(\mathcal H\otimes \ell^2)^{\mathcal U}$, while the two compression identities imply
\[
    \alpha(\ker(\rho))\,p(\mathcal H\otimes \ell^2)^{\mathcal U}\subseteq p^{\perp}(\mathcal H\otimes \ell^2)^{\mathcal U},\qquad \beta(\ker(\gamma))\,p^{\perp}(\mathcal H\otimes \ell^2)^{\mathcal U}\subseteq p(\mathcal H\otimes \ell^2)^{\mathcal U}.
\]
Let $n\geq1$, $a_0,a_n\in A$, $a_1,\ldots,a_{n-1}\in\ker(\rho)$, and $c_1,\ldots,c_n\in\ker(\gamma)$. The inclusions above imply that
\[
    v:=\beta(c_1)\alpha(a_1)\cdots\beta(c_n)\alpha(a_n)\zeta\in p(\mathcal H\otimes \ell^2)^{\mathcal U}.
\]
Since $\alpha(a_0^*)\zeta\in p^{\perp}(\mathcal H\otimes \ell^2)^{\mathcal U}$, it follows that
\[
    \rho^{\mathcal U}(\Phi(a_0c_1a_1\cdots c_na_n))
=\langle \alpha(a_0^*)\zeta,v\rangle
=0.
\]
Each such word has free product state zero, and these words, together with $A$, linearly span a dense subspace of $(A,\rho)\star(C,\gamma)$. Since $\rho^{\mathcal U}\circ\Phi$ restricts to $\rho$ on $A$, this proves state preservation.

If $\Phi(x)=0$, then $(\rho\star\gamma)(y^*x^*xy)=0$ for every $y\in (A,\rho)\star(C,\gamma)$, and thereby $x=0$, since $\rho\star\gamma$ is GNS-faithful. Therefore $\Phi$ is an embedding of \Cs-probability spaces.
Note that, by definition, we have 
\[
    \Psi(x\otimes d)=(\id_A\otimes \id_{\ell^2})^{\mathcal U}(\Phi(x))\cdot \Delta(1_A\otimes d)
\]
for all $x\in(A,\rho)\star(C,\gamma)$ and $d\in\cal{B}(\ell^2)$.
Consequently, $\Psi$ is a continuous extension of the algebraic tensor map $\Phi\odot\id_{\mathcal{B}(\ell^2)}$.

Finally, we explain why it is enough to show complete selflessness with respect to tensor products with $\cal{B}(\ell^2)$.
Let $E$ be any \Cs-algebra and fix $z=\sum_{k=1}^n x_k\otimes e_k\in ((A,\rho)\star(C,\gamma))\odot E$. The separable \Cs-subalgebra $E_0:=\Cs(\{e_1,\ldots,e_n\})\subseteq E$ admits a faithful representation $\iota\colon E_0\hookrightarrow \cal{B}(\ell^2)$. The algebraic tensor maps satisfy
\[
    (\mathrm{id}_A\otimes \iota)^{\mathcal U}\circ (\Phi\odot\id_{E_0})=\Psi \circ (\mathrm{id}_{(A,\rho)\star(C,\gamma)}\odot\iota).
\]
Since minimal tensor products and ultrapowers preserve embeddings, it follows that
\begin{align*}
    \|(\Phi\odot\id_{E})(z)\|&=\|(\Phi\odot\id_{E_0})(z)\| \\
    &=\|\Psi((\mathrm{id}_{(A,\rho)\star(C,\gamma)}\odot\iota)(z))\| \\
    &\leq\|z\|_{((A,\rho)\star(C,\gamma))\otimes E_0} \\
    &=\|z\|_{((A,\rho)\star(C,\gamma))\otimes E}.
\end{align*}
Thus $\Phi\odot\mathrm{id}_E$ is continuous for the minimal tensor norm for every $\Cs$-algebra $E$. Since $\Phi(C)\subseteq B^{\mathcal U}$, the embedding $\Phi$ witnesses complete selflessness of the inclusion $B\subseteq(A,\rho)$.
\end{proof}

\begin{remark} \label{OzawaRemark}
(i) In the second and third assertion of
Theorem~\ref{AbstractCriterion-completely-selfless},
if the conditions hold for one unit vector
$\eta_0\in\ell^2$, they hold for every unit vector
$\eta\in\ell^2$, with all other data unchanged. Indeed, choose a unitary $d\in\mathcal B(\ell^2)$
satisfying $d\eta_0=\eta$.
By the commutation hypotheses, the operator
$(\id_{\mathcal H}\otimes d)^{\mathcal U}$
commutes with $p$ in \textup{(ii)}, respectively with
$T^*$ in \textup{(iii)}, and it also commutes with $(\pi\otimes1_{\mathcal B(\ell^2)})^{\mathcal U} (\Delta_A(a))$ for all $a \in A$. Applying this operator to the last identity in the
respective assertion gives the same identity with
$\eta$ in place of $\eta_0$, while the remaining conditions do not involve $\eta$.

(ii) Let $(A,\rho)$ be a \Cs-probability space. Ozawa's criterion
in \cite[Theorem~13]{Ozawa25} states that $(A,\rho)$ is
completely selfless if there exist a \Cs-probability space
$(B,\psi)$ with a unital inclusion $A\subseteq B$ satisfying
$\psi|_A=\rho$, an ultrafilter $\mathcal U$ over $I$, and an isometry
$T\in B^{\mathcal U}$ such that $T+T^*\in A^{\mathcal U}$ and $T^*aT=\rho(a)1_A$, $\psi^{\mathcal U}(aTT^*a^*)=0$
for all $a\in A$. We explain how these hypotheses imply assertion
(iii) of
Theorem~\ref{AbstractCriterion-completely-selfless}
for the inclusion $A\subseteq(A,\rho)$.

Choose a faithful representation $(\pi,\mathcal H,\xi)$
of $(B,\psi)$, fix a unit vector $\eta\in\ell^2$, and put $\widehat T :=(\pi\otimes1_{\mathcal B(\ell^2)})^{\mathcal U}(T)$ and $\zeta:=\xi\otimes\eta$.
Then $\widehat T$ is an isometry on
$(\mathcal H\otimes\ell^2)^{\mathcal U}$ and commutes with
$(\mathrm{id}_{\mathcal H}\otimes d)_{i\to\mathcal U}$
for every $d\in\mathcal B(\ell^2)$. Moreover,
\begin{eqnarray*}
\widehat T+\widehat T^* \in \bigl((\pi|_A)\otimes1_{\mathcal B(\ell^2)}\bigr)^{\mathcal U} (A^{\mathcal U}).
\end{eqnarray*}
Finally, $$\widehat T^* (\pi(a)\otimes1_{\mathcal B(\ell^2)})^{\mathcal U} \widehat T =\rho(a)1$$
and $$\bigl\| \widehat T^* (\pi(a)\otimes1_{\mathcal B(\ell^2)})^{\mathcal U} \zeta^{\mathcal{U}} \bigr\|^2 =\psi^{\mathcal U}(a^*TT^*a)=0$$ for every $a\in A$. Thus $(\pi|_A,\mathcal H,\xi)$,
$\eta$, and $\widehat T$ satisfy all conditions in
assertion (iii). Consequently, our
characterisation recovers Ozawa's criterion without
invoking the universal property of the reduced free
product with the Toeplitz algebra.
\end{remark}

\vspace{3mm}


\section{Reduced (amalgamated) free products with a Haar unitary}

In this section, we apply Theorem~\ref{AbstractCriterion-completely-selfless} to prove Theorem~\ref{thm:amalgamated-free-product}, establishing selflessness of inclusions of \Cs-probability spaces arising from
suitable reduced amalgamated free products.
Even for reduced free products, our result generalises earlier
work; see Corollary~\ref{thm:free-product}.

\begin{proof}[Proof of Theorem \ref{thm:amalgamated-free-product}]
For notational convenience, write $\gamma:=\rho_2|_C$, $\mathbb{E}:=\mathbb{E}_1 \star \mathbb{E}_2$, and $A_i^\circ :=\ker(\mathbb{E}_i)$ for $i=1,2$. Since $\rho_i =\delta\circ \mathbb{E}_i$ is GNS-faithful, each $\mathbb{E}_i$ is non-degenerate. Moreover, the canonical conditional expectation $\mathbb{E}_{A_1}\colon A\to A_1$ is non-degenerate and satisfies $\mathbb{E}=\mathbb{E}_1\circ \mathbb{E}_{A_1}$. Hence the state $\rho=\rho_1\circ \mathbb{E}_{A_1}$ is GNS-faithful.

We identify $A$ with its image inside $\mathcal{B}(\mathcal{F})$, where $\mathcal{F}:=L^2(A,\rho)$, and describe the decomposition of $\mathcal{F}$ into reduced-word subspaces. For this, put $\Omega:=\xi_\rho$ and $\mathcal H_\delta=\overline{D\Omega} \subseteq \mathcal{F}$. By $\rho|_D=\delta$, this subspace canonically identifies with the GNS-Hilbert space $L^2(D,\delta)$, where $\Omega$ corresponds to its distinguished cyclic vector.

    For every alternating sequence $\boldsymbol i=(i_1,\dots,i_r)$ with $r\geq1$, $i_1,\ldots,i_r \in\{1,2\}$ and $i_1\neq \ldots \neq i_{r}$, define
    \[
        \mathcal{F}_{\boldsymbol i} := \overline{\operatorname{span}} \bigl\{ a_1\cdots a_r\Omega \mid a_1\in A_{i_1}^\circ,\ldots ,a_r \in A_{i_r}^\circ  \bigr\}.
    \]
For distinct alternating sequences, these subspaces are mutually orthogonal, and each of them is orthogonal to $\mathcal H_\delta$. 
Moreover, $D$ and the reduced words in $A$ linearly span a norm-dense subspace of $A$, so that
    \[
        \mathcal{F} = \mathcal H_\delta \oplus \bigoplus_{r\geq 1} \bigoplus_{\ i_1 \neq \dots \neq i_r} \mathcal{F}_{(i_1,\dots,i_r)}.
    \]
    Grouping the summands according to their first index, we obtain $\mathcal{F} = \mathcal H_\delta \oplus\mathcal{F}_1^\circ \oplus\mathcal{F}_2^\circ$, where
$\mathcal{F}_i^\circ = \bigoplus _{r\geq 1} \bigoplus_{\ i= i_1 \neq \ldots \neq i_r} \mathcal{F}_{(i_1,\dots,i_r)}$ is the space spanned by reduced-word vectors starting in $A_i^\circ$. Denote by $p_0$ the projection onto $\mathcal F_2^\circ$ and note that $p_0d=dp_0$ for every $d\in D$.
    As a result, for every $a\in A_1$ and $b\in A_2$, we have that
    \[
        p_0ap_0=\mathbb{E}_1(a)p_0,\quad p_0^\perp b p_0^\perp=\mathbb{E}_2(b)p_0^\perp,\quad p_0a\Omega=0.
    \]

For each $i \in I$ define $p_i=v_i p_0v_i^*  \, \in \cal{B}(\cal{F})$ and a faithful $\ast$-homomorphism $\theta_i \colon C\to B$ via $\theta_i(c)=v_icv_i^*$ for $c\in C$, where $B:=\Cs(C\cup\{v_i\mid i\in I\})$. Using the assumptions, we have
    \begin{align*}
        p_i^\perp\theta_i(c) p_i^\perp =v_i (p_0^\perp c p_0^\perp) v_i^*=\gamma(c) p_i^\perp
    \end{align*}
and $p_i \Omega=v_i p_0v_i^*\Omega=0$ for all $i \in I$, $c\in C$. In particular, by $\rho(\theta_i(c))
=\langle \Omega,
   p_i^\perp \theta_i(c) p_i^\perp \Omega\rangle
=\gamma(c)$, each $\theta_i$ is state-preserving.

    We claim that, for every $x\in A$,
\begin{eqnarray} \label{limits}
        \|p_ixp_i-\rho(x)p_i\|\longrightarrow 0, \quad \|p_ix\Omega\|\longrightarrow 0.
\end{eqnarray}
    In order to check these two norm limit statements, first take $a\in A_1$. Then,
    \[
        p_i ap_i  =v_i p_0(v_i ^*av_i )p_0v_i^* =v_i \mathbb{E}_1(v_i^*av_i)p_0v_i^*
    \]
and hence, by our assumptions,
    \[
        \|p_i ap_i -\rho(a)p_i \| \leq\|\mathbb{E}_1(v_i^*av_i)-\rho_1(a)1_D\| \longrightarrow 0.
    \]
Furthermore, $p_i a\Omega=v_i p_0(v_i^*a)\Omega=0$.

    Now take a word $x=a_0b_1a_1\dots b_{r-1}a_{r-1}b_ra_r$, where $r\geq1$, $a_0,a_r\in A_1$, $a_1,\ldots a_{r-1}\in A_1^\circ$, and $b_1,\ldots b_r\in A_2^\circ$. Such words, together with $A_1$, linearly span a norm-dense subspace of $A$. Expanding $a_0$ and $a_r$ into their $\mathbb{E}_1$-centered parts shows that $\mathbb{E}(x)=0$, and hence $\rho(x)=0$. Put $w:=b_1a_1\dots b_{r-1}a_{r-1}b_r$, and write $v_i^*a_0=\alpha_i+c_i$ and $a_rv_i=\beta_i+e_i$,
    where $\alpha_i:=\mathbb{E}_1(v_i^*a_0)$ and $\beta_i=\mathbb{E}_1(a_rv_i)$ belong to $D$, and $c_i,e_i\in A_1^\circ$. The hypotheses give
    \[
        \|\alpha_i\|\longrightarrow 0,\quad \|\beta_i\| =\|\mathbb{E}_1(v_i^*a_r^*)\|\longrightarrow 0.
    \]
    Moreover, $\|c_i\|\leq2\|a_0\|$ and $\|e_i \|\leq2\|a_r\|$ for every $i \in I$.
    We have
    \[
        v_i^*xv_i =\alpha_iw\beta_i+\alpha_iwe_i +c_iw\beta_i+c_iwe_i,
    \]
where the word $c_i we_i$ starts and ends with letters in $A_1^\circ$, and therefore maps $p_0\mathcal{F}$ into $\im(p_0^\perp)$. In particular, $p_0c_iwe_ip_0=0$. Combining these observations, we obtain
    \begin{align*}
        \|p_ixp_i\| &=\|p_0v_i^*xv_ip_0\|\\
        &\leq\|w\|\bigl( \|\alpha_i\|\|\beta_i\| +2\|a_r\|\|\alpha_i\| +2\|a_0\|\|\beta_i\| \bigr) \longrightarrow 0.
    \end{align*}
By linearity and norm density, this implies the first limit in \eqref{limits}.

    To check the second one, observe that $wa_r\Omega$ belongs to $\mathcal{F}_2^\circ$. Since $c_i\in A_1^\circ$ and $\alpha_i\in D$ for every $i \in I$, it follows that
    \[
        p_0v_i^*x\Omega =p_0(\alpha_i+c_i)wa_r\Omega =\alpha_iwa_r\Omega.
    \]
    Hence,
    \[
        \|p_i x\Omega\| =\|p_0v_i^*x\Omega\| \leq\|\alpha_i\|\|w\|\|a_r\| \longrightarrow 0.
    \]
    By linearity, norm density, and the uniform bound $\|p_i\|=1$, we obtain the second limit in \eqref{limits}, proving the claim.

    Now, define $q_i:= p_i \otimes 1_{\cal{B}(\ell^2)}$ for every $i \in I$ and note that $[q_i,1_{\cal{B}(\cal{F})}\otimes T]=0$ for all $T\in\cal{B}(\ell^2)$. Furthermore, we have that 
    \begin{align*}
        \|q_i (a\otimes 1_{\cal{B}(\ell^2)}) q_i-\rho(a)q_i\| &=\|p_i ap_i -\rho(a)p_i\|\longrightarrow 0 \\
        q_i^{\perp}(\theta_i(c) \otimes1_{\cal{B}(\ell^2)}) q_i^{\perp}&=\gamma(c)q_i^{\perp} \\
        \|q_i(a\otimes 1_{\cal{B}(\ell^2)})(\Omega\otimes\eta)\|&=\|p_i a\Omega\| \longrightarrow 0
    \end{align*}
    for all $a\in A$, $c\in C$, and any unit vector $\eta\in\ell^2$.
    
Finally, choose a cofinal ultrafilter $\mathcal{U}$ on $I$, let $q:=(q_i)_{i \rightarrow \mathcal{U}}$, viewed as a projection in $\mathcal{B}((\mathcal{F}\otimes \ell^2)^{\mathcal{U}})$, and define $\theta\colon C\to B^{\mathcal{U}}$ by $\theta(c)=(\theta_i(c))_{i \rightarrow \mathcal{U}}$ for $c\in C$. Since each $\theta_i$ is a unital, isometric, state-preserving $\ast$-homomorphism, $\theta$ is a unital, state-preserving $\ast$-embedding. In particular, the restriction of $\rho^{\mathcal{U}}$ to $\theta(C)$ is GNS-faithful. Moreover, the identities and limits established above give
\begin{align*}
q\,\bigl(((\pi_\rho \otimes 1_{\cal{B}(\ell^2)})^{\mathcal{U}}\circ \Delta)(x)\bigr) \,q &= \rho(x) \, q,\\
q^{\perp} \, \bigl(((\pi_\rho \otimes 1_{\cal{B}(\ell^2)})^\mathcal{U} \circ \theta)(c)\bigr) \, q^{\perp} &= \gamma(c) \, q^{\perp}, \\
q \, \bigl(((\pi_\rho \otimes1_{\cal{B}(\ell^2)})^{\mathcal{U}}\circ \Delta)(x)\bigr)(\Omega\otimes \eta)^\mathcal{U} &= 0
\end{align*}
for all $x\in A$ and $c\in C$, where $\pi_\rho$ denotes the GNS-representation of $(A,\rho)$. Viewing $\theta$ as an embedding into $A^{\mathcal U}$, the preceding identities imply with assertion \textup{(ii)} of Theorem \ref{AbstractCriterion-completely-selfless} that $(A,\rho)$ is completely selfless.

If, in addition, $\rho|_B$ is GNS-faithful, then
$B\subseteq(A,\rho)$ is an inclusion of
$\Cs$-probability spaces. By employing assertion \textup{(ii)} of Theorem \ref{AbstractCriterion-completely-selfless}   once again, we obtain that $B\subseteq(A,\rho)$ is completely selfless.
\end{proof}

We are now ready to prove the generalisation of  \cite[Theorem C]{HKER25}.

\begin{proof}[Proof of Corollary \ref{thm:free-product}]
Set $D:=\mathbb C$, $\mathbb E_1:=\psi$,
$\mathbb E_2:=\rho$, and $v_n:=u^n$ for $n\in\mathbb N$.
By $u\in B^\psi$, we have $\psi(u^{-n}bu^n)=\psi(b)$ for all $b\in B$, $n\in\mathbb N$. Moreover, the Haar condition implies that
$(u^n\xi_\psi)_{n\in\mathbb N}$ is an orthonormal sequence, so that for every $b\in B$,
\[
    |\psi(u^{-n}b)|
    =|\langle u^n\xi_\psi,b\xi_\psi\rangle|
    \longrightarrow0.
\]
Thus both limit conditions in
Theorem~\ref{thm:amalgamated-free-product} hold,
and its condition on $C$ is automatic.

It remains to verify the GNS-faithfulness required for
the inclusion conclusion. Since $u$ is a Haar unitary, $(\Cs(u),\psi|_{\Cs(u)})
    \cong(C(\mathbb T),\leb)$. The canonical embeddings induce
a state-preserving $\ast$-embedding $(\Cs(u),\psi|_{\Cs(u)})\star(C,\rho|_C)
    \hookrightarrow(B,\psi)\star(A,\rho)$ with image $\Cs(C\cup\{u\})$.
The restricted state therefore identifies with the
GNS-faithful free-product state on the source.
The additional assertion of
Theorem~\ref{thm:amalgamated-free-product}
now yields the claimed complete selflessness.
\end{proof}

We now apply our results to free Araki--Woods
\Cs-algebras by establishing complete selflessness with respect to the vacuum state whenever the underlying real Hilbert space has dimension at least two. This strengthens the selflessness results of \cite{HKER25} and covers the two-dimensional case associated with non-trivial orthogonal representations.
In the two-dimensional case, one can also show selflessness (but not complete selflessness) by combining the same proof as in \cite{HKER25} and the very recent result of Shlyakhtenko \cite{Shl26} that reduced free products of at least two semicircular algebras are all isomorphic to each other.

Let us first recall Shlyakhtenko's construction; see
\cite[Remark~2.1 and Definition~2.3]{Shl97}.
Let $\mathcal H_{\mathbb R}$ be a real Hilbert space, and let
$(U_t)_{t\in\mathbb R}$ be a strongly continuous one-parameter
group of orthogonal transformations on $\mathcal H_{\mathbb R}$.
Set
$\mathcal H_{\mathbb C}:=\mathcal H_{\mathbb R}\otimes_{\mathbb R}\mathbb C$,
and also denote by $U_t$ its complex-linear extension to
$\mathcal H_{\mathbb C}$. Let $a$ be the positive injective,
possibly unbounded, self-adjoint operator on
$\mathcal H_{\mathbb C}$ satisfying $U_t=a^{it}$ and define
\[
    j_U\colon\mathcal H_{\mathbb R}\longrightarrow
    \mathcal H_{\mathbb C},
    \qquad
    j_U(\xi):=
    \left(\frac{2a}{1+a}\right)^{1/2}(\xi\otimes1).
\]
On the full Fock space
\[
    \mathcal F(\mathcal H_{\mathbb C})
    :=
    \mathbb C\Omega
    \oplus
    \bigoplus_{n=1}^{\infty}
    \mathcal H_{\mathbb C}^{\otimes n},
\]
let $l(\eta)$ denote the left creation operator associated
with $\eta\in\mathcal H_{\mathbb C}$, determined by $l(\eta)\Omega :=\eta$ and  $l(\eta)\zeta:=\eta\otimes\zeta$ for $\zeta\in\mathcal H_{\mathbb C}^{\otimes n}$. The associated \emph{free Araki--Woods \Cs-algebra} is
\[
    \Gamma(\mathcal H_{\mathbb R},U_t)
    :=
    \Cs\bigl(
        \{l(j_U(\xi))+l(j_U(\xi))^*
        \mid \xi\in\mathcal H_{\mathbb R}\}
    \bigr)
    \subseteq\mathcal B(\mathcal F(\mathcal H_{\mathbb C})).
\]
Its \emph{free quasi-free state} is the vacuum state $\varphi_U$ given by $x \mapsto \langle\Omega,x\Omega\rangle$. This state is faithful; see \cite[Section~3]{Shl97}.
The von Neumann algebra
$\Gamma(\mathcal H_{\mathbb R},U_t)''$ is called the
\emph{free Araki--Woods von Neumann algebra}.

For trivial $(U_t)_{t\in\mathbb R}$, selflessness in dimension at least two was established in \cite[Theorem~7.6]{HKER25}; for non-trivial $(U_t)_{t\in\mathbb R}$, the corresponding result in dimension at least three appears in \cite[Corollary~7.7]{HKER25}.
The following corollary establishes complete selflessness for
$\dim_{\mathbb R}(\mathcal{H}_{\mathbb R})\geq2$.

\begin{corollary}\label{cor:free-araki-woods-selfless}
Let $\mathcal H_{\mathbb R}$ be a real Hilbert space with
$\dim_{\mathbb R}(\mathcal H_{\mathbb R})\geq2$, and let
$(U_t)_{t\in\mathbb R}$ be a strongly continuous one-parameter
group of orthogonal transformations on $\mathcal H_{\mathbb R}$.
Then the free Araki--Woods \Cs-probability space
$(\Gamma(\mathcal H_{\mathbb R},U_t),\varphi_U)$ is completely
selfless.
\end{corollary}

\begin{proof}
Suppose first that $\mathcal H_{\mathbb R}$ contains a
non-zero finite-dimensional $U_t$-invariant subspace.
By \cite[Remark~2.12]{Shl97}, we may choose such a subspace
to be one-dimensional with trivial action or
two-dimensional with non-trivial action.
Using \cite[Theorem~2.11]{Shl97} and, in the latter case,
\cite[Theorem~4.8 and Corollary~4.9]{Shl97}, we obtain
a state-preserving decomposition
\[
    (\Gamma(\mathcal H_{\mathbb R},U_t),\varphi_U)
    \cong (A,\rho)\star(B,\psi),
\]
where $A\not\cong\mathbb C$, where $B$ is Abelian, and where the state $\psi$ is diffuse. By \cite[Theorem~5.4]{Thi24}, $B$ contains a Haar unitary,
which automatically belongs to $B^\psi$.
Corollary~\ref{thm:free-product} therefore yields complete selflessness of the \Cs-probability space
$(\Gamma(\mathcal H_{\mathbb R},U_t),\varphi_U)$.

Suppose now that $\mathcal H_{\mathbb R}$ has no non-zero
finite-dimensional $U_t$-invariant subspace.
By the same argument as in the proof of
\cite[Corollary~7.7]{HKER25}, one obtains a decomposition of the form $\mathcal H_{\mathbb R}=\bigoplus_{n \in \mathbb{N}}\mathcal K_n$ into
non-zero closed $U_t$-invariant subspaces.
For $n \in \mathbb{N}$ let $A_n:=\Gamma(\mathcal K_n,U_t|_{\mathcal K_n})$, $\rho_n :=\varphi_{U|_{\mathcal K_n}}$, and define $(A,\rho):=\mathop{\star}_{n\geq1}(A_n,\rho_n)$, so that by \cite[Theorem~2.11]{Shl97},
\[
    (\Gamma(\mathcal H_{\mathbb R},U_t),\varphi_U)
    \cong (A,\rho)\star(A_0,\rho_0).
\]

Equation (3) in  \cite[Section 2]{Shl97} implies that $(A_n,\rho_n)$ contains a semicircular \Cs-probability subspace, generated by the operator associated with any unit vector in $\mathcal K_n$, implying the existence of a Haar unitary $v_n\in A_n$. For $N \in \mathbb{N}$, $a\in\Cs(A_1\cup\cdots\cup A_N)$ and $n>N$, freeness gives $\rho(v_n^*av_n)=\rho(a)$ and $\rho(v_n^*a)=0$. It follows that
\[
    |\rho(v_n^*av_n)-\rho(a)|\longrightarrow0,
    \qquad
    |\rho(v_n^*a)|\longrightarrow0
\]
for every $a\in A$. Thereby
Theorem~\ref{thm:amalgamated-free-product} applies to $(A,\rho)\star(A_0,\rho_0)$ yielding complete selflessness.
\end{proof}

\vspace{3mm}


\section{Group \Cs-algebras and crossed products}\label{sec:groups}

Let $\Gamma$ be a discrete group, and let
$\lambda\colon\Gamma\to\mathcal B(\ell^2(\Gamma))$
denote its left regular representation.
The associated \emph{reduced group \Cs-algebra} is
\[
    C_{\mathrm r}^*(\Gamma)
    :=\Cs(\{\lambda_g\mid g\in\Gamma\})
    \subseteq\mathcal B(\ell^2(\Gamma));
\]
we denote its canonical tracial state by $\tau_\Gamma$
and call $\Gamma$ \emph{(completely) \Cs-selfless} if
$(C_{\mathrm r}^*(\Gamma),\tau_\Gamma)$ is a (completely) selfless
\Cs-probability space.

The following corollary of Theorem~\ref{AbstractCriterion-general} characterises (complete) \Cs-selflessness in terms of covariant representations of the canonical action of the free product group $\Gamma\star\mathbb Z$ on the Bowditch compactification of its Bass--Serre tree.

\begin{corollary}
Let $\Gamma$ be a non-trivial discrete group, and let $\alpha\colon\Gamma\star\mathbb{Z}\curvearrowright\Delta T$ be the canonical action of the free product $\Gamma\star\mathbb{Z}$ on the Bowditch compactification of its Bass--Serre tree. Denote by $h$ the canonical generator of $\mathbb{Z}$, and let $o\in V(T)$ be the vertex corresponding to the coset $\Gamma\in(\Gamma\star\mathbb{Z})/\Gamma$. Then the following statements are equivalent:

\begin{enumerate}[label=\textup{(\roman*)},leftmargin=*]
\item The group $\Gamma$ is \Cs-selfless.
\item There exist an ultrafilter $\mathcal{U}$, a faithful representation $(\pi,\mathcal{H},\xi)$ of $(C_{\mathrm{r}}^{\ast}(\Gamma),\tau_{\Gamma})$, and a covariant unital representation $(u,\rho,\mathcal{H}^{\mathcal{U}})$ of the $(\Gamma\star\mathbb{Z})$-\Cs-algebra $C(\Delta T)$ such that
\begin{eqnarray*}
\begin{aligned}u_{g}  =(\pi^{\mathcal{U}}\circ\Delta)(\lambda_{g}), \quad  u_{h} \in\pi^{\mathcal{U}}\bigl(C_{\mathrm{r}}^{\ast}(\Gamma)^{\mathcal{U}}\bigr), \quad \rho(f)\xi^{\mathcal{U}} =f(o)\xi^{\mathcal{U}} ,
\end{aligned}
\end{eqnarray*}
for all $g \in \Gamma$ and $f\in C(\Delta T)$.
\end{enumerate}

\vspace{1mm}

\noindent Similarly, the following statements are equivalent:
\begin{enumerate}[label=\textup{(\roman*)},leftmargin=*]
\item
The group $\Gamma$ is completely \Cs-selfless.
\item There exist an ultrafilter $\mathcal U$, a faithful
representation $(\pi,\mathcal H,\xi)$ of
$(C_{\mathrm r}^*(\Gamma),\tau_\Gamma)$, a unit vector
$\eta\in\ell^2$, and a covariant unital representation
$(u,\rho,(\mathcal H\otimes\ell^2)^{\mathcal U})$
of the $(\Gamma\star\mathbb Z)$-\Cs-algebra $C(\Delta T)$
such that for all $g\in\Gamma$, $f\in C(\Delta T)$, $d\in\mathcal B(\ell^2)$ the conditions 
\begin{align*}
u_g &= \bigl((\pi\otimes1_{\mathcal B(\ell^2)})^{\mathcal U}
  \circ \Delta\bigr) (\lambda_g),\\
    u_h &\in  \bigl(\pi\otimes1_{\mathcal{B}(\ell^2)}\bigr) ^{\mathcal U}
       \bigl(C_{\mathrm r}^*(\Gamma)^{\mathcal U}\bigr), \\
       \rho(f)(\xi\otimes\eta)^{\mathcal U} &= f(o)(\xi\otimes\eta)^{\mathcal U}.
\end{align*}
and $[\rho(f),(\mathrm{id}_{\mathcal H}\otimes d)^{\mathcal U}]=0$ hold.
\end{enumerate}
\end{corollary}

\begin{proof}
By \cite[Proposition~3.1.1]{Hasegawa18} and \cite[Corollary~4.2.7]{Hasegawa18} we have that $\Delta\mathbf{T}(A,\tau)\cong C(\Delta T)\rtimes_{\alpha,\mathrm{r}}(\Gamma\star\mathbb{Z})$, where $A:=C_{\text{r}}^{\ast}(\Gamma\star\mathbb{Z})\cong(C_{\text{r}}^{\ast}(\Gamma),\tau_{\Gamma})\star(C_{\text{r}}^{\ast}(\mathbb{Z}),\tau_{\mathbb{Z}})$ and where $\tau:=\tau_{\Gamma\star\mathbb{Z}}$ is the canonical tracial state on $A$. Hence $C(\Delta T)\rtimes_{\alpha,\mathrm{r}}(\Gamma\star\mathbb{Z})$ is universally generated by unital copies of $C_{\mathrm r}^{\ast}(\Gamma)$ and $C_{\mathrm r}^{\ast}(\mathbb Z)$, and a projection $q\in C(\Delta T)$, subject to the relations $qaq=q^{\perp}bq^{\perp}=0$ for all $a\in\ker(\tau)\cap C_{\text{r}}^{\ast}(\Gamma)$ and $b\in\ker(\tau)\cap C_{\text{r}}^{\ast}(\mathbb{Z})$. Here $q\in C(\Delta T)$ is the characteristic function on the clopen subset of \(\Delta T\) determined by the half-tree opposite $o$ across the fundamental edge. Its $\Gamma$-translates partition $\Delta T\setminus\{o\}$, so that
\begin{equation}
\{f\in C(\Delta T)\mid f(o)=0\}=\overline{\operatorname{span}}\{f\alpha_{g}(q)\mid f\in C(\Delta T),\ g\in\Gamma\}.\label{eq:SetEquality}
\end{equation}

For the implication ``$\text{(i)}\Rightarrow\text{(ii)}$'' of the first equivalence above, suppose that $\Gamma$ is \Cs-selfless. By \cite[Theorem~2.6(iv)]{Rob25}, selflessness admits a witness with auxiliary probability space $(C_{\mathrm r}^{\ast}(\mathbb Z),\tau_{\mathbb Z})$. Applying the forward construction in the proof of Theorem~\ref{AbstractCriterion-general} to this witness, we obtain an ultrafilter $\mathcal{U}$, a faithful representation $(\pi,\mathcal{H},\xi)$ of $(C_{\mathrm{r}}^{\ast}(\Gamma),\tau_{\Gamma})$, a state-preserving $\ast$-embedding $\theta\colon(C_{\text{r}}^{\ast}(\mathbb{Z}),\tau_{\mathbb{Z}})\hookrightarrow(C_{\text{r}}^{\ast}(\Gamma)^{\mathcal{U}},\tau_{\Gamma}^{\mathcal{U}})$, and a projection $p\in\mathcal{B}(\mathcal{H}^{\mathcal{U}})$ satisfying $p((\pi^{\mathcal{U}}\circ\Delta)(a))p=p^{\perp}((\pi^{\mathcal{U}}\circ\theta)(b))p^{\perp}=0$ for all $a\in\ker(\tau)\cap C_{\text{r}}^{\ast}(\Gamma)$, $b\in\ker(\tau)\cap C_{\text{r}}^{\ast}(\mathbb{Z})$ and $p((\pi^{\mathcal{U}}\circ\Delta)(a))\xi^{\mathcal{U}}=0$ for all $a\in C_{\text{r}}^{\ast}(\Gamma)$. The universal property of $\Delta\mathbf{T}(A,\tau)$ thereby yields a unital $\ast$-representation $\Psi\colon C(\Delta T)\rtimes_{\alpha,\mathrm{r}}(\Gamma\star\mathbb{Z})\to\mathcal{B}(\mathcal{H}^{\mathcal{U}})$ extending $\pi^{\mathcal{U}}\circ\Delta$ on $C_{\text{r}}^\ast (\Gamma)$ and $\pi^{\mathcal{U}}\circ\theta$ on $C_\text{r}^\ast (\mathbb Z)$, and satisfying $\Psi(q)=p$. Let $(u,\rho,\mathcal{H}^{\mathcal{U}})$ be the associated covariant representation. The required conditions on $u$ then hold by construction, and moreover
\[
\rho(\alpha_{g}(q))\xi^{\mathcal{U}}=u_{g}pu_{g}^{*}\xi^{\mathcal{U}}=u_{g}p((\pi^{\mathcal{U}}\circ\Delta)(\lambda_{g^{-1}}))\xi^{\mathcal{U}} =0
\]
for $g \in \Gamma$ giving with \eqref{eq:SetEquality} that $\rho(f)\xi^{\mathcal{U}}=f(o)\xi^{\mathcal{U}}$ for every $f\in C(\Delta T)$.

For the opposite implication ``$\text{(ii)}\Rightarrow\text{(i)}$'' of the first statement above, suppose that (ii) holds and set $p:=\rho(q)$. The half-tree relations $q\alpha_{g}(q)=q^{\perp}\alpha_{h^{n}}(q^{\perp})=0$ for $g\in\Gamma\setminus\{e\}$, $n\in\mathbb{Z}\setminus\{0\}$ imply in combination with the covariance of $(u,\rho,\mathcal{H}^{\mathcal{U}})$ that $pu_{g}p=p^{\perp}u_{h}^{n}p^{\perp}=0$. Moreover, since $\Gamma$ fixes $o$ and $q(o)=0$,
\[
pu_{g}\xi^{\mathcal{U}}=u_{g}\rho(\alpha_{g^{-1}}(q))\xi^{\mathcal{U}}=(\alpha_{g^{-1}}(q)(o))u_g\xi^{\mathcal{U}}=0\qquad\text{ for }g\in\Gamma.
\]
Let $v\in C_\text{r}^\ast(\Gamma)^{\mathcal{U}}$ be the unitary satisfying $\pi^{\mathcal{U}}(v)=u_{h}$. Then the $\ast$-homomorphism $\theta\colon(C_{\text{r}}^{\ast}(\mathbb{Z}),\tau_{\mathbb{Z}})\rightarrow(C_\text{r}^\ast(\Gamma)^{\mathcal{U}},\tau_{\Gamma}^{\mathcal{U}})$ given by $\theta(\lambda_{h}):=v$ satisfies by $p\xi^{\mathcal{U}}=0$ the equality
\[
(\tau_{\Gamma}^{\mathcal{U}}\circ\theta)(\lambda_{h}^{n})=\tau_{\Gamma}^{\mathcal{U}}(v^{n})=\langle \xi^{\mathcal{U}},u_{h}^{n} \,\xi^{\mathcal{U}}\rangle=\langle\xi^{\mathcal{U}},(p^{\perp}u_{h}^{n}p^{\perp})\,\xi^{\mathcal{U}}\rangle=0
\]
for $n\in\mathbb{Z}\setminus\{0\}$. Hence $\theta$ is trace-preserving and its injectivity follows from the faithfulness of $\tau_{\mathbb{Z}}$. By linearity and continuity, the preceding identities imply with Theorem~\ref{AbstractCriterion-general} the \Cs-selflessness of $\Gamma$.

The second equivalence above follows from adapting the previous argument using Theorem~\ref{AbstractCriterion-completely-selfless}. We leave the details to the reader.
\end{proof}

We now turn to applications of Theorem \ref{thm:amalgamated-free-product} in the context of group \Cs-algebras.

\begin{definition}
    Let $\Gamma$ be a discrete group. A subgroup $\Lambda\le \Gamma$ is said to be \emph{unconfined} in $\Gamma$ if for every finite subset $F\subseteq \Gamma\setminus\{e\}$, there exists an element $g\in\Gamma$ such that $g\Lambda g^{-1}\cap F=\emptyset$.
\end{definition}

\begin{corollary}\label{cor:unconfined}
Let $\Gamma_1,\Gamma_2$ be discrete groups with $\Gamma_1$ infinite, and let $\Lambda\leq\Gamma_1,\Gamma_2$ be a non-trivial common subgroup that is unconfined in $\Gamma_1$. Furthermore, let $\Lambda_2\leq\Gamma_2$ be a non-trivial subgroup satisfying $\Lambda\cap\Lambda_2=\{e\}$. Then there exists a net $(g_i)_{i \in I}\subseteq\Gamma_1$ such that the inclusion
\begin{equation} \label{C*SelflessInclusion}
    \Cs\bigl(
        \{\lambda_{g_i}\mid i \in I\}
        \cup C_{\mathrm r}^{\ast}(\Lambda_2)
    \bigr)
    \subseteq
    \bigl(
        C_{\mathrm r}^{\ast}(\Gamma_1\star_\Lambda\Gamma_2),
        \tau_{\Gamma_1\star_\Lambda\Gamma_2}
    \bigr)
\end{equation}
is completely selfless. In particular, $\Gamma_1\star_\Lambda\Gamma_2$ is completely
\Cs-selfless.
\end{corollary}

\begin{proof}
Let $I$ be the set of finite subsets of $\Gamma_1\setminus\{e\}$, directed by inclusion. Since $\Lambda$ is unconfined in $\Gamma_1$, for each $i \in I$ we may choose $g_i \in\Gamma_1$ with $i \cap g_i\Lambda g_i^{-1}=\emptyset$. Consequently, $\mathbb E_1(\lambda_{g_i}^{\ast}\lambda_h\lambda_{g_i})=0$ for all $h \in i$, where $\mathbb{E}_1$ is the canonical conditional expectation $\Cs_\text{r}(\Gamma_1) \to \Cs_\text{r}(\Lambda)$. It follows that $\|\mathbb E_1(\lambda_{g_i}^{\ast}a\lambda_{g_i})-\tau_{\Gamma_1}(a) 1_{\Cs_\text{r}(\Lambda)} \| \rightarrow 0$ for every $a\in C_{\mathrm r}^{\ast}(\Gamma_1)$.

Fix $k\in\Lambda\setminus\{e\}$.
For any $h\in\Gamma_1$, if $g_i^{-1}h\in\Lambda$ we have that $hkh^{-1}\in g_i\Lambda g_i^{-1}$, leading to a contradiction whenever $i$ contains $hkh^{-1}$.
Hence, $\mathbb E_1(\lambda_{g_i}^{\ast}\lambda_h)=0$
eventually for every $h\in\Gamma_1$, so that $\|\mathbb E_1(\lambda_{g_i}^{\ast}a)\| \rightarrow0$ for every $a\in C_{\mathrm r}^{\ast}(\Gamma_1)$.

Finally, let
$\mathbb E_2\colon C_{\mathrm r}^{\ast}(\Gamma_2)
\to C_{\mathrm r}^{\ast}(\Lambda)$
be the canonical conditional expectation. Since $\Lambda\cap\Lambda_2=\{e\}$, we have $\mathbb E_2(c)=\tau_{\Gamma_2}(c) 1_{\Cs_\text{r}(\Lambda)}$ for every $c\in C_{\mathrm r}^{\ast}(\Lambda_2)$.
We may therefore apply Theorem \ref{thm:amalgamated-free-product} to obtain that the inclusion in \eqref{C*SelflessInclusion} is completely selfless. 
\end{proof}
\begin{remark}
Let $\Gamma_1$ be a \Cs-simple discrete group,
let $\Gamma_2$ be a discrete group, and let
$\Lambda\leq\Gamma_1,\Gamma_2$ be a non-trivial common
amenable subgroup. Kennedy's criterion in \cite[Theorem~1.1]{Ken20} states
that a discrete group is \Cs-simple if and only if every amenable subgroup is unconfined. This implies in combination with Corollary~\ref{cor:unconfined} that, if $\Gamma_2$ contains a non-trivial subgroup
$\Lambda_2$ satisfying $\Lambda\cap\Lambda_2=\{e\}$, the amalgamated free product $\Gamma_1\star_\Lambda\Gamma_2$ is completely \Cs-selfless.
\end{remark}

\begin{corollary}
Let $(A,\rho)$, $(B,\psi)$ be \Cs-probability spaces
with $B\not\cong\mathbb C$.
Write $\mathbb F_2=\langle a,b\rangle$, let
$H:=\langle aba^{-1}b^{-1}\rangle$, and denote the canonical conditional expectation $\Cs_\text{r}(\mathbb F_2)\to\Cs_\text{r}(H)$ by $\mathbb E_H$. Then the reduced amalgamated free product
\begin{equation} \label{AmalgamatedExpression}
    (A\otimes\Cs_\text{r}(\mathbb F_2),\rho\otimes\mathbb E_H)
    \star_{\Cs_\text{r}(H)}
    (\Cs_\text{r}(H)\otimes B,\id\otimes\psi)
\end{equation}
is completely selfless with respect to the composition of $\tau_H$ with the canonical conditional expectation onto $\Cs_\text{r}(H)$.
\end{corollary}

\begin{proof}
Set $c:=aba^{-1}b^{-1}$ and let $|\cdot | \colon \mathbb{F}_2 \rightarrow \mathbb{N}$ be the word length function on $\mathbb{F}_2$ with respect to the standard generators $a$ and $b$. For all $n\in \mathbb{N}_{\geq 1}$ and $k\neq 0$,
\[
|a^{-n}c^ka^n|=2n+4|k|-2\ge 2n+2.
\]
In particular, for any non-trivial element $g\in \mathbb F_2$ and $n \in \mathbb{N}_{\geq 1}$  with $2n+2>|g|$, we have that $a^{n} g a^{-n}\not\in H$. It follows that $ \|\mathbb E_H(\lambda_{a^n}x\lambda_{a^{-n}})-\tau_{\mathbb F_2}(x)1_{\Cs_\text{r}(H)}\| \rightarrow 0$ for all $x\in\Cs_\text{r}(\mathbb F_2)$. By a density argument, this  implies that for all $x\in A\otimes\Cs_\text{r}(\mathbb F_2)$,
    \[
    \lim_{n\to\infty}\| (\rho\otimes \mathbb E_H)((1_A\otimes\lambda_{a^n})x(1_A\otimes\lambda_{a^{-n}}))-(\rho\otimes\tau_{\mathbb F_2})(x)1_{\Cs_\text{r}(H)}\|=0.
    \]

Now consider the unique group homomorphism $\chi\colon\mathbb F_2\rightarrow\mathbb Z$ with $\chi(a)=1$ and $\chi(b)=0$. Since $H\subseteq\ker(\chi)$, if for some element $g\in\mathbb F_2$ one had $a^ng\in H$, then $n+\chi(g)=0$. For large $n$, this leads to a contradiction, so that by a density argument, $\|\mathbb E_H(\lambda_{a^n}x)\| \rightarrow 0$
for all $x\in\Cs_\text{r}(\mathbb F_2)$. Again, tensoring with $1_A$ and approximating by sums in the algebraic tensor product, we obtain 
    \[
    \lim_{n\to\infty}\| (\rho\otimes \mathbb E_H)((1_A\otimes\lambda_{a^n})x)\|=0
    \]
    for all $x\in A\otimes \Cs_\text{r}(\mathbb F_2)$.

All hypotheses of
Theorem~\ref{thm:amalgamated-free-product}
are therefore satisfied, yielding the assertion.
\end{proof}

An action $\alpha\colon\Gamma\curvearrowright X$ of a countable discrete group $\Gamma$ on a compact Hausdorff space $X$ is called \emph{minimal} if every orbit $\{\alpha_g(x)\mid g\in\Gamma\}$ is dense in $X$, and \emph{extremely proximal} if, for every pair of non-empty open subsets $U,V\subseteq X$, there exists $g\in\Gamma$ such that $\alpha_g(X\setminus U)\subseteq V$. We call $X$ an \emph{extreme boundary} if the action is both minimal and extremely proximal. Finally, the action is \emph{topologically free} if the points with trivial stabilizer form a dense subset of $X$.

Following \cite{Ozawa25}, we call a sequence $(g_n)_{n\in\mathbb N}\subseteq\Gamma$ \emph{axial} if there exist a topologically free extreme boundary action $\alpha\colon\Gamma\curvearrowright X$ and distinct points $z_+,z_-\in X$ such that, for every choice of neighbourhoods $U_+$ of $z_+$ and $U_-$ of $z_-$, $\alpha_{g_n}(X\setminus U_-)\subseteq U_+$ for all sufficiently large $n$, and
\[
\{\alpha_g(z_+),\alpha_g(z_-)\}\cap\{z_+,z_-\}=\emptyset \qquad\text{for every }g\in\Gamma\setminus\{e\}.
\]
As observed in \cite{Ozawa25}, a second-countable extreme boundary $X$ with $|X|>2$ admits an axial sequence if and only if the action is topologically free.

The following theorem establishes a generalization of \cite[Theorem 1]{Ozawa25} and provides a conceptually different proof using Theorem \ref{MainTheorem}. 

\begin{theorem} Let $\alpha\colon\Gamma\curvearrowright A$ be an action of a countable discrete group $\Gamma$ on a unital \Cs-algebra $A$, and let $\mathbb{E}\colon A\rtimes_{\alpha,\text{r}}\Gamma\rightarrow A$ be the canonical conditional expectation from the corresponding reduced crossed product \Cs-algebra $A\rtimes_{\alpha,\text{r}}\Gamma$ onto $A$. Assume that $\Gamma$ admits an axial sequence $(g_{n})_{n\in\mathbb{N}}\subseteq\Gamma$ and a free ultrafilter $\mathcal{U}$ with $\lim_{i\rightarrow\mathcal{U}}\Vert\alpha_{g_{i}}(a)-a\Vert=0$ for every $a\in A$. Then there exists a unital $\ast$-embedding
\[
\theta\colon(A\rtimes_{\alpha,\text{r}}\Gamma,\mathbb{E})\star_{A}(A\otimes C(\mathbb{T}),\id \otimes \tau_{\text{\emph{Leb}}})\hookrightarrow(A\rtimes_{\alpha,\text{r}}\Gamma)^{\mathcal{U}}
\]
with $\theta|_{A\rtimes_{\alpha,\text{r}}\Gamma}=\Delta$ and $\mathbb{E}^{\mathcal{U}}\circ\theta=\Delta\circ(\mathbb{E}\star(\id \otimes \tau_{\text{\emph{Leb}}}))$. In particular, the triple $(A\rtimes_{\alpha,\text{r}}\Gamma,\mathbb{E},A)$ is relatively selfless in the sense of \cite[Definition 7.1]{GJKPR26}. \end{theorem}
\begin{proof}
Let $\beta\colon\Gamma\curvearrowright X$ be an action witnessing the axiality of the sequence $(g_{n})_{n\in\mathbb{N}}\subseteq\Gamma$, and denote the corresponding boundary points by $z_{\pm}\in X$. Assume that $A\subseteq\mathcal{B}(\mathcal{H})$ for some Hilbert space $\mathcal{H}$. Since $\Gamma$ is assumed to be countable, we may choose a point $z\in X\setminus\{\beta_{g}(z_{\pm})\mid g\in\Gamma\}$. The minimality of $\beta$ implies that $C(X)$ can be viewed as a C$^{\ast}$-subalgebra of $\ell^{\infty}(\Gamma)\subseteq\mathcal{B}(\mathcal{H}\otimes\ell^{2}(\Gamma))$ via $f(g):=f(\beta_{g}(z))$ for $g\in\Gamma$.

Note that by our assumptions the element $u:=(\lambda_{g_{i}})_{i\to\mathcal{U}}\in(A\rtimes_{\alpha,\mathrm{r}}\Gamma)^{\mathcal{U}}$ commutes with $\Delta(A)$. Write $\pi_{1}$ for the canonical $\ast$-embedding $A\rtimes_{\alpha,\mathrm{r}}\Gamma\hookrightarrow\mathcal{B}(\mathcal{H}\otimes\ell^{2}(\Gamma))$ and let $\pi_{2}\colon A\otimes C(\mathbb{T})\hookrightarrow(A\rtimes_{\alpha,\mathrm{r}}\Gamma)^{\mathcal{U}}$ be the $\ast$-embedding induced by $\pi_{2}(a\otimes\operatorname{id}_{\mathbb{T}}):=\Delta(a)u$ for $a\in A$. Then $\pi_{2}$ restricts to $\Delta|_{A}$ on $A$ and $V^{*}\pi_{1}(x)V=\mathbb{E}(x)$ for every $x\in A\rtimes_{\alpha,\mathrm{r}}\Gamma$, where the isometry $V\in\mathcal{B}(\mathcal{H},\mathcal{H}\otimes\ell^{2}(\Gamma))$ is defined by $V\xi:=\xi\otimes\delta_{e}$ for $\xi\in\mathcal{H}$. Finally, let $p$ be the orthogonal projection on $(\mathcal{H}\otimes\ell^{2}(\Gamma))^{\mathcal{U}}$ onto the Hilbert subspace $\mathcal{K}:=\mathcal{K}_{+}\oplus\mathcal{K}_{-}$, where
\[
\mathcal{K}_{\pm}:=\left\{ \xi\in(\mathcal{H}\otimes\ell^{2}(\Gamma))^{\mathcal{U}}\mid f\xi=f(z_{\pm})\xi\text{ for all }f\in C(X)\right\} .
\]

We proceed by proving the following claims.\\

\emph{Claim 1}. For every $x\in A\rtimes_{\alpha,\mathrm{r}}\Gamma$ with $\mathbb{E}(x)=0$ the equality $p((\pi_{1}^{\mathcal{U}}\circ\Delta)(x))p=0$ holds.

\vspace{1mm}

\emph{Proof of the claim}. Since $\pi_{1}(A)$ commutes with $C(X)$, the projection $p$ commutes with $(\pi_{1}^{\mathcal{U}}\circ\Delta)(A)$. For $g\in\Gamma\setminus\{e\}$ the axiality of the sequence $(g_{n})_{n\in\mathbb{N}}$ gives $\{\beta_{g}(z_{+}),\beta_{g}(z_{-})\}\cap\{z_{+},z_{-}\}=\emptyset$, so that
\[
p(\pi_{1}^{\mathcal{U}}\circ\Delta)(a\lambda_{g})p=(\pi_{1}^{\mathcal{U}}\circ\Delta)(a)\,p(\pi_{1}^{\mathcal{U}}\circ\Delta)(\lambda_{g})p=0
\]
for $a\in A$. Since elements of the form $a\lambda_{g}$ with $a\in A$ and $g\in \Gamma\setminus\{e\}$ linearly span a norm dense subspace of $\ker(\mathbb{E})$, the claim hence follows by continuity.\\

\emph{Claim 2}. For every $x\in A\otimes C(\mathbb{T})$ with $(\id\otimes \tau_{\text{Leb}})(x)=0$ the equality $p^{\perp}((\pi_{1}^{\mathcal{U}}\circ\pi_{2})(x))p^{\perp}=0$ holds.

\vspace{1mm}

\emph{Proof of the claim}. Let $f\in C(X)$ be a function vanishing on a neighbourhood of $\{z_{+},z_{-}\}\subseteq X$. Then,
\[
\|(F-F(z_{+})1)\pi_{1}(\lambda_{g_{i}^{n}})f\|\leq\|f\|\sup_{y\in\operatorname{supp}(f)}|(F\circ\beta_{g_{i}^{n}})(y)-F(z_{+})|\rightarrow0
\]
for every $F\in C(X)$ and $n\in\mathbb{N}_{\geq 1}$; hence $(\pi_{1}^{\mathcal{U}}\circ\pi_{2})(1\otimes\operatorname{id}_{\mathbb{T}}^{n})f\xi\in\mathcal{K}_{+}$ for every $\xi\in(\mathcal{H}\otimes\ell^{2}(\Gamma))^{\mathcal{U}}$. The common kernel of the multiplication operators $f$ vanishing on a neighbourhood of $\{z_{+},z_{-}\}$ is precisely $\mathcal{K} \subseteq (\mathcal{H}\otimes\ell^{2}(\Gamma))^{\mathcal{U}}$, so that
\[
p^{\perp}(\pi_{1}^{\mathcal{U}}\circ\pi_{2})(1\otimes\operatorname{id}_{\mathbb{T}}^{n})p^{\perp}=0.
\]
 Taking adjoints implies the same identity for $n<0$. Since $p$ commutes with $(\pi_{1}^{\mathcal{U}}\circ\Delta)(A)$ the claim follows.\\

\emph{Claim 3}. For every $x\in A\rtimes_{\alpha,\mathrm{r}}\Gamma$ the equality $p\bigl((\pi_{1}^{\mathcal{U}}\circ\Delta\bigr)(x))V^{\mathcal{U}}=0$ holds.

\vspace{1mm}

\emph{Proof of the claim}. For every $a\in A$, $g\in\Gamma$, $f\in C(X)$, and $\xi\in\mathcal{H}^{\mathcal{U}}$ one has that
\[
f(\pi_{1}^{\mathcal{U}}\circ\Delta)(a\lambda_{g})V^{\mathcal{U}}\xi=f(\beta_{g}(z))\,(\pi_{1}^{\mathcal{U}}\circ\Delta)(a\lambda_{g})V^{\mathcal{U}}\xi.
\]
Since $z\in X\setminus\{\beta_{g}(z_{\pm})\mid g\in\Gamma\}$, it follows that the vector $(\pi_{1}^{\mathcal{U}}\circ\Delta)(a\lambda_{g})V^{\mathcal{U}}\xi$ is orthogonal to $\mathcal{K}$. Consequently, $p(\pi_{1}^{\mathcal{U}}\circ\Delta)(a\lambda_{g})V^{\mathcal{U}}=0$. The claim then follows by continuity and since the elements $a\lambda_{g}$ with $a\in A$, $g\in\Gamma$ linearly span a norm-dense subspace of $A\rtimes_{\alpha,\mathrm{r}}\Gamma$.\\

An application of Theorem~\ref{MainTheorem} now implies the desired assertion.
\end{proof}

\vspace{3mm}


\section{Hecke \Cs-algebras}\label{sec:Hecke}

In this section, we characterise complete selflessness for Hecke \Cs-algebras associated with free products of copies of $\mathbb Z/{2\mathbb{Z}}$, equipped with their canonical tracial states, and discuss the consequences for their classification by $K$-theoretic invariants.

We begin by establishing two permanence properties of completely selfless \Cs-probability spaces that will be used below. For selfless \Cs-probability spaces, the first statement of the following lemma is proved in \cite{Rob25} in full generality, whereas the second is proved there only in the tracial setting. To our knowledge, neither statement has previously been established for completely selfless \Cs-probability spaces. The proof of the second statement employs our characterisation of complete selflessness in Theorem~\ref{AbstractCriterion-completely-selfless}.

\begin{lemma}[{cf.\ \cite[Theorems 4.3 and 4.4]{Rob25}}]\label{lem:permanence}
For a completely selfless \Cs-probability space $(A,\rho)$ the following statements hold:
    \begin{enumerate}[label=\textup{(\roman*)},leftmargin=*]
        \item For every $m\in\mathbb N_{\geq 1}$ the \Cs-probability space $(M_m(A),\rho\otimes\mathrm{tr}_m)$ is completely selfless, where $\mathrm{tr}_m$ denotes the normalised trace on $M_m$.
        \item For every projection $p\in A$ with $\rho(p)>0$ the \Cs-probability space $\big(pAp,\rho(p)^{-1}\rho\big)$ is completely selfless.
    \end{enumerate}
\end{lemma}

\begin{proof}
    \emph{About (i)}: If $\rho$ is not faithful, \cite[Theorem~2]{Gould26} implies that $A$ is purely infinite simple, so that $M_m(A)$ is purely infinite simple as well. By \cite[Theorem~3]{Ozawa25}, the \Cs-probability space $(M_m(A),\rho\otimes\tr_m)$ is hence completely selfless. We may therefore assume that $\rho$ is faithful. Consider a $\ast$-embedding $\theta\colon (A,\rho)\star(\Cs_r(\mathbb F_\infty),\tau_{\mathbb F_\infty}) \hookrightarrow  (A^{\mathcal U},\rho^{\mathcal U})$ witnessing the complete selflessness of the \Cs-probability space $(A,\rho)$. The matrix amplification $\theta^{(m)}$ of $\theta$
witnesses that the inclusion
\[
    M_m(A)\subseteq
    \left(
        M_m((A,\rho)\star
        (\Cs_r(\mathbb F_\infty),\tau_{\mathbb F_\infty})),
        (\rho\star\tau_{\mathbb F_\infty})\otimes\tr_m
    \right)
\]
is completely existential. In (the proof of) \cite[Theorem 4.3]{Rob25}, Robert constructs an embedding of \Cs-probability spaces
    \[
    (M_m(A),\rho\otimes\tr_m)\star(C(\mathbb T),\leb) \hookrightarrow \big(M_m((A,\rho)\star(\Cs_r(\mathbb F_\infty),\tau _{\mathbb{F}_\infty})),(\rho\star\tau_{\mathbb F_\infty})\otimes\tr_m\big)
    \]
    that is trivial on $M_m(A)$.
    Therefore, the completely existential embedding $\theta^{(m)}$ factors through the first free factor embedding 
    \[
    (M_m(A),\rho\otimes\tr_m) \hookrightarrow (M_m(A),\rho\otimes\tr_m)\star(C(\mathbb T),\leb),
    \]
    which is thus completely existential itself by the completely existential analogue of \cite[Lemma 1.2]{Rob25}.\\

\emph{About (ii)}: For notational convenience, we write $\rho_p:=\rho(p)^{-1}\rho $. Since $A$ is selfless, the \Cs-algebra $pAp$ is simple and thereby $\rho_p$ is GNS-faithful. Fix the data provided by Theorem \ref{AbstractCriterion-completely-selfless} (iii): a faithful representation $(\pi,\cal{H},\xi)$ of $(A,\rho)$, a unit vector $\eta\in\ell^2$, an ultrafilter $\mathcal U$ over $I$, and an isometry $T\in \cal{B}((\cal{H}\otimes\ell^2)^{\mathcal U})$. Identify $A$ with its image under the diagonal embedding, and set $\zeta :=(\xi\otimes\eta)_{i\to\mathcal U}$. We then have that $T^*aT=\rho(a)1$ and $T^*a\zeta=0$ for all $a\in A$, that  $T$  commutes with every operator $(\id_{\cal{H}}\otimes d)_{i\to\mathcal U}$ with $d \in \mathcal{B}(\ell^2)$, and that $T+T^*\in(\pi\otimes1_{\cal{B}(\ell^2)})^{\mathcal U}(A^{\mathcal U})$.
Consider the projection $P:=(\pi(p)\otimes1_{\cal{B}(\ell^2)})_{i\to\calU}$, and define 
    \[
    S:=\rho(p)^{-1/2}(PTP) \in\cal{B}(P(\mathcal{H}\otimes\ell^2)^{\mathcal U})\cong \cal{B}(\bigl(\pi(p)\mathcal{H}\otimes\ell^2\bigr)^{\mathcal U})
    \]
and $\zeta_p :=\rho(p)^{-1/2}P\zeta$. Note that the vector $\zeta_p$ is the constant ultrapower vector associated with $\rho(p)^{-\frac{1}{2}}\pi(p)\xi$, which implements the state $\rho_p$ in the faithful representation of $pAp$ on $\pi(p)\mathcal{H}$.

We verify the conditions of Theorem \ref{AbstractCriterion-completely-selfless}. For this,  set $\pi_p(a):=\pi(a)|_{\pi(p)\mathcal H}$ for $a\in pAp$, set $\xi_p:=\rho(p)^{-1/2}\pi(p)\xi$, and note that by
    \[
    S^*S=\rho(p)^{-1}(PT^*PTP) =\rho(p)^{-1}\left( P(\rho(p)1)P\right)=P,
    \]
the operator $S$ is an isometry in $\cal{B}(P(\mathcal{H}\otimes\ell^2)^{\mathcal U})$. Furthermore, for $a\in pAp$, 
    \[
    S^*aS=\rho(p)^{-1}(PT^*aTP)=\rho(a)\rho(p)^{-1}P=\rho_p(a)P
    \]
and similarly, $S^*a\zeta_p=\rho(p)^{-1}PT^*a\zeta=0$. Next, 
    \[
    S+S^*=\rho(p)^{-\frac{1}{2}}P(T+T^*)P\in(\pi_p\otimes1)^{\mathcal U}\bigl((pAp)^{\mathcal U}\bigr),
    \]
    where we used that $pA^{\mathcal U}p\subseteq (pAp)^{\mathcal U}$.
    Finally, both $P$ and $T$ commute with $(\id_{\cal{H}}\otimes d)_{i\to\mathcal U}$ for every $d\in \mathcal{B}(\ell^2)$, so that $S$ commutes with their respective restrictions to the compressed Hilbert space. Hence the conditions of Theorem \ref{AbstractCriterion-completely-selfless} (iii) hold for the \Cs-probability space $(pAp,\rho_p)$, which is thereby completely selfless.
\end{proof}

\begin{lemma}\label{lem:corners}
    Let $(A,\tau)$ be a tracial \Cs-probability space, and assume that $p$ is a full projection in $A$.
    Then $(A,\tau)$ is completely selfless if and only if $(pAp,\tau(p)^{-1}(\tau|_{pAp}))$ is completely selfless.
\end{lemma}

\begin{proof}
Let $\tau_p:=\tau(p)^{-1}(\tau|_{pAp})$. If $(A,\tau)$ is completely selfless, so is $(pAp,\tau_p)$ by Lemma~\ref{lem:permanence}.

Conversely, if $(pAp,\tau_p)$ is completely
selfless, the fullness of $p$ implies that there exist an integer $m\in\mathbb N_{\geq1}$ and a projection $P \in M_m(pAp)$ such that $A \cong PM_m(pAp)P$, where the isomorphism intertwines the states $\tau$ and $\left((\tau_{p}\otimes\tr_{m})(P)\right)^{-1}(\tau_{p}\otimes\tr_{m})|_{PM_{m}(pAp)P}$. Lemma~\ref{lem:permanence}, applied first to the matrix
amplification and then to its corner, therefore yields
complete selflessness of $(A,\tau)$. \end{proof}

Hecke \Cs-algebras may be viewed as deformations of reduced group \Cs-algebras of Coxeter systems, obtained by completing the corresponding Iwahori--Hecke algebras; they have been studied in \cite{CKL21, RS22, Klisse23-1, Klisse23-2, KP25, NPRS25}. In the context of the present article, we restrict to the case of right-angled Coxeter groups of the form $W:=(\mathbb Z/2\mathbb Z)^{\star n}$, whose Hecke \Cs-algebras admit a particularly simple description in terms of reduced free products of 2-dimensional Abelian \Cs-algebras. More precisely, for $n\in\mathbb N_{\geq1}$ and a multi-parameter $q=(q_1,\ldots,q_n)\in \mathbb{R}_{>0}^n$, the corresponding \emph{Hecke \Cs-algebra} $C_{\text{r},q}^\ast (W)$ is of the form
\[
C_{\text{r},q}^\ast((\mathbb Z/2\mathbb Z)^{\star n})
    \cong
    ( \mathbb{C}\oplus \mathbb{C},\tau_{q_1}) \star \cdots \star ( \mathbb{C}\oplus \mathbb{C},\tau_{q_n}),
    \]
where $\tau_{q_i}(z_1,z_2):=(1+q_i)^{-1}(z_1 +q_iz_2)$ for $1\leq i\leq n$ and $z_1,z_2\in\mathbb C$; see \cite[Section~1.10]{CKL21} and \cite[Section~2.3]{RS22}. In this picture, the canonical faithful tracial state on $C_{\text{r},q}^\ast((\mathbb Z/2\mathbb Z)^{\star n})$ identifies with the free product state $\tau_{q}:=(\tau_{q_1} \star \ldots \star \tau_{q_n})$. Note that for $q_1=\cdots=q_n=1$, the construction recovers the
reduced group \Cs-algebra
$\Cs_{\mathrm r}((\mathbb Z/2\mathbb Z)^{\star n})$, equipped with its canonical tracial state.

\begin{theorem}\label{thm:Hecke}
    Let $n\geq 3$ be an integer and $q=(q_1,\dots,q_n)\in \mathbb{R}_{>0}^n$ a multi-parameter. Then the \Cs-probability space $(\Cs_{\text{r}, q}((\mathbb Z/2\mathbb Z)^{\star n}),\tau_{q})$ is completely selfless if and only if $\sum_{i=1}^n\frac{\min\{1,q_i\}}{1+q_i}>1$.
\end{theorem}

\begin{proof}
By \cite[Theorem~4.8]{Dyk99}, the \Cs-algebra $A:=\Cs_{\mathrm r,q}((\mathbb Z/2\mathbb Z)^{\star n})$ is simple precisely when the stated inequality holds. Since complete selflessness implies simplicity, it thereby remains to prove the sufficiency of the assertion. For this, assume that the inequality $\sum_{i=1}^n\frac{\min\{1,q_i\}}{1+q_i}>1$ holds. By
\cite[Proposition~4.7]{CKL21}, we may replace each parameter greater than $1$ by its inverse and hence assume that $0<q_i\leq1$ for every $1 \leq i \leq n$. By \cite[Proposition~2.7]{Dyk99}, the first two free factors
generate a subalgebra $D\oplus C\subseteq A$, where
$D\neq0$, $C$ is finite-dimensional, and
the normalised tracial state on $D$ is diffuse.
Hence, \cite[Theorem~5.4]{Thi24} implies that
$(D,\tau_q(e)^{-1}\tau_q|_D)$ with $e:=1_D$ contains a Haar unitary.

Let $B\subseteq A$ be the \Cs-subalgebra generated by
$\mathbb Ce\oplus C$ and the remaining free factors.
By \cite[Proposition~2.8]{Dyk99} there is a state-preserving isomorphism
\[
    \left(eAe,\frac{\tau_q|_{eAe}}{\tau_q(e)}\right)
    \cong
    \left(D,\frac{\tau_q|_D}{\tau_q(e)}\right)
    \star
    \left(eBe,\frac{\tau_q|_{eBe}}{\tau_q(e)}\right),
\]
where $eBe\neq\mathbb Ce$.

Corollary~\ref{thm:free-product} now implies that
$(eAe,\tau_q(e)^{-1}\tau_q|_{eAe})$ is completely selfless.
Since $A$ is simple and $e\neq0$, the projection $e$ is full.
Lemma~\ref{lem:corners} therefore yields complete
selflessness of $(A,\tau_q)$.
\end{proof}

We now record some consequences for ordered $K$-theory and the Elliott invariant of the Hecke \Cs-algebras considered in Theorem \ref{thm:Hecke}.

\begin{remark}\label{rem:ordered-K-theory}
Let $A$ be one of the Hecke \Cs-algebras covered by Theorem~\ref{thm:Hecke}, and let $\tau$ denote its unique tracial state. Since $A$ is simple and has strict comparison of positive elements, we have
\[
    K_0(A)^+
    =
    \{0\}\cup\{x\in K_0(A)\mid\tau_*(x)>0\}.
\]
Thus, the positive cone is completely determined by the trace pairing. Consequently, the Elliott invariant of $A$ contains exactly the same information as its unordered version, which in this case consists of $K_0(A)$, its distinguished order unit $[1_A]_0$, $K_1(A)$, and the trace pairing $\tau_*$.

In~\cite[Remark~5.5]{RS22}, Raum and Skalski highlight the problem of computing ordered $K$-theory for reduced free products beyond those that satisfy the Avitzour condition, which are treated in the work of Dykema and R{\o}rdam~\cite{DR98}. Theorem~\ref{thm:free-product} and Theorem~\ref{thm:Hecke} address this problem for a broad class of reduced free products without the Avitzour condition. Together with the computations in~\cite{RS22}, Theorem~\ref{thm:Hecke} therefore yields a complete description of the Elliott invariant for the Hecke \Cs-algebras covered by Theorem~\ref{thm:Hecke}.
\end{remark}

\vspace{3mm}


\section{$q$-Gaussian \Cs-algebras}\label{sec:q-gaussian}

In the present section we apply our previous results to $q$-Gaussian \Cs-algebras, establishing selflessness of the inclusions generated by any two orthogonal unit vectors throughout the range $-1<q<1$ with respect to the corresponding canonical tracial state. This strengthens the simplicity and trace-uniqueness results of \cite{AmrutamJekelWasilewski26}  (see also \cite[Theorem 4.1]{BCKW23}) and additionally yields strict comparison and stable rank one, with the same conclusions for all unital intermediate \Cs-algebras.

We begin by recalling the construction of $q$-Gaussian \Cs-algebras. Fix $q\in(-1,1)$ and a real Hilbert space $\mathcal{H}_{\mathbb R}$. Furthermore, put $\mathcal{H}=\mathcal{H}_{\mathbb R}\otimes_{\mathbb R}\mathbb C$. For every $r \in \mathbb{N}_{\geq 1}$ define the \emph{symmetrisation operator} $P_q^r \in \mathcal{B}(\mathcal{H}^{\otimes r})$ by
\[
P_q^r(\zeta_1\otimes\cdots\otimes \zeta_r):= \sum_{\sigma \in S_r} q^{i(\sigma)} (\zeta_{\sigma(1)}\otimes \cdots \otimes \zeta_{\sigma(r)})
\]
for $\zeta_1, \ldots, \zeta_r \in \mathcal{H}$, where $S_r$ is the symmetric group of permutations of $r$ elements and $$i(\sigma):=\#\{(a,b)\mid 1\leq a<b\leq r,\; \sigma(a)>\sigma(b)\}$$ is the corresponding number of inversions. By \cite[Proposition 1]{BozejkoSpeicher91} the operator $P_q^r$ is strictly positive and invertible, so that $\langle \zeta_1, \zeta_2\rangle_q:=\langle P_q^r \zeta_1, \zeta_2 \rangle$ for $\zeta_1, \zeta_2 \in \mathcal{H}^{\otimes r}$ defines an inner product on $\mathcal{H}^{\otimes r}$; we denote the associated Hilbert space by $\mathcal{H}_q^{\otimes r}$.

The \emph{$q$-Fock space} of $\mathcal{H}$ is defined by
\[
    \mathcal{F}_q(\mathcal{H}):=\mathbb C\Omega\oplus \bigoplus_{r=1}^\infty \mathcal{H}_q^{\otimes r},
\]
where $\Omega$ is a unit vector called the \emph{vacuum vector}. For $\zeta\in\mathcal H$, the \emph{left creation operator}
$l_q(\zeta)$ is initially defined on the algebraic span of elementary tensors via $l_q(\zeta)\Omega:=\zeta$ and $l_q(\zeta)(\zeta_1\otimes\cdots\otimes\zeta_r):=\zeta\otimes\zeta_1\otimes\cdots\otimes\zeta_r$ and extends to a bounded operator on $\mathcal F_q(\mathcal H)$. Its adjoint $l_q^\ast(\zeta):=(l_q (\zeta))^\ast $ is called the \emph{annihilation operator} and satisfies
\[
    l_q^\ast(\zeta)\Omega=0,\qquad
l_q^\ast (\zeta)(\zeta_1\otimes\cdots\otimes \zeta_r)
=\sum_{j=1}^r q^{j-1}\langle \zeta,\zeta_j\rangle
 (\zeta_1\otimes\cdots\otimes\widehat{\zeta_j}
 \otimes\cdots\otimes \zeta_r)
\]
for $\zeta_1,\ldots,\zeta_r \in \mathcal{H}$, so that in particular for $\zeta_1,\zeta_2 \in \mathcal{H}$,
\begin{eqnarray} \label{q-Relation}
l_q^\ast(\zeta_1)l_q(\zeta_2)-q\;l_q(\zeta_2)l_q^\ast(\zeta_1) =\langle \zeta_1,\zeta_2 \rangle1.
\end{eqnarray}

The \Cs-subalgebra $\mathcal{A}_q (\mathcal{H}_\mathbb{R})$ of $\mathcal{B}(\mathcal{F}_q(\mathcal{H}))$ generated by all operators of the form
\[
s_q(\xi):=l_q(\xi)+l_q^\ast (\xi) \in \mathcal{B}(\mathcal{F}_q(\mathcal{H}))
\]
with $\xi \in \mathcal{H}_\mathbb{R}$ is called the \emph{$q$-Gaussian \Cs-algebra}. It carries a  canonical faithful tracial state $\tau$, called the \emph{vacuum state}, obtained by restricting the vector state associated with $\Omega$; see \cite[Definition~1.1, Remark~1.2 and Proposition~2.3]{BKS97}. All vectors $\zeta_1, \ldots, \zeta_r \in \mathcal{H}$ admit a unique operator $W_q(\zeta_1 \otimes \cdots \otimes \zeta_r) \in \mathcal{A}_q(\mathcal{H}_\mathbb{R})$ satisfying $W_q(\zeta_1 \otimes \cdots \otimes \zeta_r)\Omega =\zeta_1 \otimes \cdots \otimes \zeta_r$. These operators, together with $W_q(\Omega):=1$, are called \emph{Wick operators}. Their linear span is the unital $\ast$-algebra of polynomials in the operators $s_q(\xi)$ with $\xi\in\mathcal H_{\mathbb R}$, and is therefore norm-dense in $\mathcal A_q(\mathcal H_{\mathbb R})$; see \cite[Definition~2.5, Remarks~2.6, Proposition~2.7]{BKS97}.

For $r \in \mathbb{N}_{\geq 1}$ we write $[r]_q:=(1-q)^{-1}(1-q^r)$, $[0]_q!:=1$, and $[r]_q!:=\prod_{j=1}^r[j]_q$.

\begin{lemma}\label{lem:oscillator} Let $\xi$ be a unit vector in $\cal{H}_{\R}$. The map
\[
V_\xi(\delta_r\otimes\zeta):=\frac{1}{\sqrt{[r]_q!}}(l_q(\xi))^r\zeta \quad \text{ for } r\in \mathbb{N}, \zeta\in \ker(l_q^\ast (\xi))
\]
extends uniquely to a unitary operator $V_\xi \in \mathcal{B}\bigl(\ell^2\otimes\ker(l_q^*(\xi)),\mathcal F_q(\mathcal H)\bigr)$ satisfying
\begin{eqnarray} \label{WeightedShift}
        \bigl( V_\xi^\ast l_q(\xi)V_\xi \bigr) (\delta_r\otimes\zeta) =\sqrt{[r+1]_q}\,(\delta_{r+1}\otimes\zeta).
\end{eqnarray}
\end{lemma}

\begin{proof}
Using the identity in \eqref{q-Relation}, an inductive argument implies that for every $r \in \mathbb{N}_{\geq 1}$,
    \[
        l_q^\ast(\xi)l_q(\xi)^r=[r]_q l_q(\xi)^{r-1}+q^rl_q(\xi)^rl_q^\ast(\xi).
    \]
It follows that the subspaces $l_q(\xi)^r\ker(l_q^\ast (\xi)) \subseteq \mathcal{F}_q(\mathcal{H})$ are pairwise orthogonal. Furthermore, 
    \[
        \langle l_q(\xi)^r\zeta_1,l_q(\xi)^r\zeta_2\rangle_q =[r]_q!\,\langle\zeta_1,\zeta_2\rangle_q \quad \text{for all }\zeta_1,\zeta_2\in\ker(l_q^\ast (\xi)).
    \]
Thus $V_\xi$ is an isometry onto a subspace $\mathcal F_0 \subseteq \mathcal{F}_q(\mathcal{H})$ invariant under both $l_q(\xi)$ and $l_q^\ast(\xi)$.

For $q\geq0$ we have that $l_q^\ast(\xi)l_q(\xi)\geq1$, while $l_q^\ast(\xi)l_q(\xi)\geq(1+q)1$ for $q<0$; in both cases the operator $l_q(\xi)$ is bounded below. Its restriction to $\mathcal{F}_0^\perp$ hence has closed range and its adjoint has zero kernel. It follows that $l_q(\xi)$ is onto $\mathcal{F}_0^\perp$ and thereby
    \[
        \mathcal{F}_0^\perp\subseteq \bigcap_{N\in \mathbb{N}_{\geq 1}}\operatorname{ran}(l_q(\xi)^N)\subseteq\bigcap_{N \in \mathbb{N}_{\geq 1}} \bigoplus_{r\geq N}\mathcal{H}_q^{\otimes r}=\{0\}.
    \]
This proves the surjectivity of $V_\xi$ and the formula in \eqref{WeightedShift}.
\end{proof}

Let us introduce some notation that will be used in the next result.
For $N\in \mathbb{N}_{\geq 1}$, and a unit vector $\xi\in\cal{H}_{\mathbb R}$, let $p_{N,\xi} \in \mathcal{B}(\mathcal{F}_q(\mathcal{H}))$ be the projection onto the range of the operator $l_q(\xi)^N$. Observe that, as a consequence of Lemma~\ref{lem:oscillator},  $p_{N,\xi}$ strongly converges to $0$. Moreover, for any $N\in \mathbb{N}_{\geq1}$, we further have that $p_{N,\xi}^\perp\,\Omega=\Omega$.

By \cite[Theorem~2.13]{BKS97}, for every unit vector $\xi\in\cal{H}_{\mathbb R}$ there exist a
$\tau$-preserving conditional expectation $\mathbb E_\xi\colon \mathcal A_q(\mathcal H_{\mathbb R}) \to\Cs(s_q(\xi))$ satisfying 
\begin{eqnarray}\label{eq:q-expectation-wick}
    \mathbb E_\xi
    \bigl(W_q(\zeta_1\otimes\cdots\otimes\zeta_r)\bigr)=
    \left(\prod_{j=1}^r\langle\xi,\zeta_j\rangle\right)
    W_q(\xi^{\otimes r})
\end{eqnarray}
for all $\zeta_1,\ldots,\zeta_r \in \mathcal{H}$.

\begin{proposition}\label{prop:two-estimates}
The following statements hold true:
\begin{enumerate}[leftmargin=*,label=\textup{(\roman*)}]
\item $\lim_{N\to\infty}\|p_{N,\eta}(a-\mathbb E_\eta(a))p_{N,\eta}\| = 0$ for every $a \in \mathcal A_q(\mathcal H_{\mathbb R})$, and for every unit vector $\eta\in\cal{H}_{\mathbb R}$, 
\item $\lim_{n\to\infty} \|\mathbb E_\eta(u^{-n}au^n)-\tau(a)1\|=0$ for every $a \in \mathcal A_q(\mathcal H_{\mathbb R})$, and for every pair of orthogonal unit vectors $\xi,\eta \in \cal{H}_{\R}$, where $u:=\exp(i s_q(\xi))$.
\end{enumerate}
\end{proposition}

\begin{proof}
By Lemma~\ref{lem:oscillator}, for $N \in \mathbb{N}_{\geq1}$, $r\in\mathbb N$,
and $\zeta_0\in\ker(l_q(\eta)^*)$, we have
\[
    \|l_q(\eta)^NV_\eta(\delta_r\otimes\zeta_0)\|^2
    =(1-q)^{-N}
      \prod_{j=r+1}^{r+N}(1-q^j)\,\|\zeta_0\|^2.
\]
Set $\alpha_q:=\prod_{j\geq1}(1-|q|^j)>0$ and $\beta_q:=\prod_{j\geq1}(1+|q|^j)<\infty$. The equality above in particular implies that for $\zeta\in\mathcal F_q(\mathcal H)$,
\[
    \sqrt{\alpha_q}(1-q)^{-N/2}\|\zeta\|
    \leq\|l_q(\eta)^N\zeta\|
    \leq\sqrt{\beta_q}(1-q)^{-N/2}\|\zeta\|.
\]
Combining this in the case where $\zeta\in\mathcal H_{\mathbb R}$ with $\zeta\perp\eta$ with  $l_q^\ast(\zeta)l_q(\eta)^N=q^Nl_q(\eta)^Nl_q^\ast(\zeta)$
yields
\begin{equation}\label{eq:q-perpendicular-annihilator}
    \|l_q^\ast(\zeta)p_{N,\eta}\|
    \leq C_q\|l_q(\zeta)\|\,|q|^N,
\end{equation}
where $C_q:=\sqrt{\beta_q/\alpha_q}$.\\

\emph{About (i)}: It suffices to prove the first assertion for elements of the form $a=W_q(\zeta_1\otimes\cdots\otimes\zeta_d)$ with $\zeta_1,\ldots,\zeta_d\in\mathcal H_{\mathbb R}$. The multilinearity of $W_q$ allows to expand $a$ as a finite linear combination of Wick operators whose inputs are either $\eta$ or perpendicular to $\eta$.
By \eqref{eq:q-expectation-wick}, the summand whose inputs all equal $\eta$ is precisely $\mathbb E_\eta(a)$.
Consequently, without loss of generality, we may assume that $\zeta_1,\ldots,\zeta_d\in \{\eta\} \cup (\mathcal{H}_\mathbb{R}\ominus \mathbb{R}\eta)$ with $\zeta_i \perp \eta$ for at least one integer $1\leq i\leq d$. An application of \cite[Proposition~2.7]{BKS97} then allows to express $a$ as a finite linear combination of products of the form $X:=\bigl(l_q(\xi_1)\cdots l_q(\xi_m)\bigl) \bigl(l_q^\ast(\eta_1)\cdots l_q^\ast(\eta_n)\bigl)$, where each input index $\zeta_1,\ldots,\zeta_d$ occurs exactly once.

By our assumptions, the product $X$ contains a creation operator or an annihilator operator whose input is perpendicular to $\eta$. We distinguish two cases:\\

\begin{itemize}
\item \emph{Case 1}: Suppose first that the annihilation part of $X$ contains an operator whose input is perpendicular to $\eta$, and choose the rightmost annihilator with this property. Every factor to its right then equals $l_q^\ast(\eta)$, so that $X$ is of the form $X=Yl_q^\ast(\zeta)l_q^\ast(\eta)^r$ for suitable elements $\zeta\in\mathcal H_{\mathbb R}$ with $\zeta\perp\eta$ and $0\leq r\leq d-1$, where $Y$ is the product of the remaining factors. Using the identification of $\mathcal F_q(\mathcal H)$ and $\ell^2\otimes\ker(l_q^*(\eta))$ in Lemma~\ref{lem:oscillator}, one checks that $l_q     ^\ast(\eta)^rp_{N,\eta}= p_{N-r,\eta}l_q^\ast(\eta)^rp_{N,\eta}$ for $N>d$. Using \eqref{eq:q-perpendicular-annihilator}, we obtain
\begin{align*}
    \|p_{N,\eta}Xp_{N,\eta}\|
    &\leq
    \|Y\|\,\|l_q^\ast(\zeta)p_{N-r,\eta}\|\,
    \|l_q(\eta)\|^r\\
    &\leq
    C_q\|Y\|\,\|l_q(\zeta)\|\,
    \|l_q(\eta)\|^r\,|q|^{N-r}
    \longrightarrow0.
\end{align*}

\item  \emph{Case 2}: Now assume that the annihilation part of $X$ contains no perpendicular input, so that its creation part must have one.
By applying the preceding argument to the adjoint $X^\ast$, we obtain $\|p_{N,\eta}Xp_{N,\eta}\|=\|p_{N,\eta}X^*p_{N,\eta}\|\rightarrow0$.
\end{itemize}

\vspace{1mm}

Combined, the cases above establish that $\|p_{N,\eta}(a-\mathbb E_\eta(a))p_{N,\eta}\| \rightarrow 0$, as claimed.\\

\emph{About (ii)}: By Lemma~\ref{lem:oscillator}, we have $V_\xi^*uV_\xi\in \mathcal B(\ell^2) \otimes\mathbb C1 \subseteq \mathcal{B}(\ell^2 \otimes \ker(l_q^\ast (\xi)))$. Similarly, $V_\xi^\ast p_{1,\xi}^{\perp} V_\xi$ and $V_\xi^\ast p_{N,\xi}^\perp V_\xi$ correspond to the projections onto the first coordinate and the first $N$ coordinates of the factor $\ell^2$ respectively, so that
\begin{equation}\label{eq:q-oscillatory-compression}
    \|p_{N,\xi}^\perp u^np_{1,\xi}^{\perp}\|^2
    =
    \sum_{r=0}^{N-1}
    \frac{|\langle \xi^{\otimes r},u^n\Omega\rangle_q|^2}
         {[r]_q!}
\end{equation}
for every $N\in \mathbb{N}_{\geq 1}$, where $\xi^{\otimes0}:=\Omega$. By \cite[Proposition~2.9]{BKS97}, each
Wick operator $W_q(\xi^{\otimes r})$ can be expressed as a polynomial in $s_q(\xi)$. Since the vacuum distribution of $s_q(\xi)$ is absolutely
continuous by \cite[Theorem~1.10]{BKS97}, the
Riemann--Lebesgue lemma gives that $\langle \xi^{\otimes r},u^n\Omega\rangle_q = \tau\bigl(W_q(\xi^{\otimes r})^*u^n\bigr) \rightarrow0$ for every $r\in \mathbb{N}$. With \eqref{eq:q-oscillatory-compression} it follows that $ \|p_{N,\xi}^\perp u^np_{1,\xi}^{\perp}\| \rightarrow 0$ as $n\rightarrow \infty$.

The orthogonality of $\xi$ and $\eta$ gives
\[
    \overline{\Cs(s_q(\eta))\Omega}
    =
    \overline{\operatorname{span}}
    \{\eta^{\otimes r}\mid r\in\mathbb N\}
    \subseteq
    \ker(l_q^\ast(\xi))
    =
    \im(p_{1,\xi}^{\perp}).
\]
By trace preservation and bimodularity, the compression
of an element $x\in\mathcal A_q(\mathcal H_{\mathbb R})$
to this cyclic subspace agrees with the restriction of
$\mathbb E_\eta(x)$.
Since the representation of $\Cs(s_q(\eta))$ on this
subspace is faithful, it follows that $\|\mathbb E_\eta(x)\|
    \leq \|{p_{1,\xi}^\perp}x{p_{1,\xi}^\perp}\|$.

The formula in \eqref{eq:q-expectation-wick} implies that
$\mathbb E_\eta(\mathbb E_\xi(a))=\tau(a)1$.
Moreover, $u$ commutes with $\mathbb E_\xi(a)$, since both elements belong to the Abelian \Cs-algebra $\Cs(s_q(\xi))$.
Using the preceding compression inequality, we thereby obtain
\begin{eqnarray*}
 \|\mathbb E_\eta(u^{-n}au^n)-\tau(a)1\| &=& \left\Vert \mathbb E_\eta(u^{-n}(a-\mathbb E_\xi(a))u^n) \right\Vert \\
&\leq& \left\Vert {p_{1,\xi}^\perp}u^{-n}(a-\mathbb E_\xi(a))u^n{p_{1,\xi}^\perp} \right\Vert \\
&\leq& \left\Vert p_{N,\xi}(a-\mathbb E_\xi(a))p_{N,\xi}\right\Vert  +2\|a-\mathbb E_\xi(a)\|\,\|p_{N,\xi}^\perp u^n{p_{1,\xi}^\perp}\|.
\end{eqnarray*}
Employing the limit above together with the statement in (1) with $\xi$ in place of $\eta$ finishes the proof.
\end{proof}

We are now ready to prove the main result of the present section, which shows selflessness for $q$-Gaussian \Cs-algebras. This gives a positive answer to the question asked in Problem 12 of \cite{OWR25}.

\begin{theorem}\label{thm:q-gaussian-selfless}
Let $-1<q<1$ and let $\mathcal{H}_{\mathbb R}$ be a real Hilbert space with $\dim(\mathcal{H}_{\mathbb R})\geq2$. Then, for any two orthogonal unit vectors $\xi,\eta \in \mathcal{H}_\mathbb{R}$, the inclusion
\[
\Cs(s_q(\xi),s_q(\eta))\subseteq (\mathcal A_q(\mathcal{H}_{\mathbb R}),\tau)
\]
is completely selfless.
\end{theorem}

\begin{proof}
As before, we denote
$u:=\exp(is_q(\xi))$.
By \cite[Theorem~1.10]{BKS97}, the vacuum distribution
of $s_q(\eta)$ is atomless, so that \cite[Proposition 4.1(i)]{DHR97} (or \cite[Proposition 4.9]{Thi24}) provides a Haar unitary $v\in\Cs(s_q(\eta))$. The corresponding system $(v^m\Omega)_{m \in \mathbb{Z}} \subseteq \mathcal{F}_q(\mathcal{H})$ is orthonormal. Consequently, traciality implies that for all $r,s\in\mathbb N$,
\[
    \langle\eta^{\otimes r},v^m\eta^{\otimes s}\rangle_q
    =
    \bigl\langle
        W_q(\eta^{\otimes r})W_q(\eta^{\otimes s})^*
        \Omega,\,
        v^m\Omega
    \bigr\rangle_q
    \longrightarrow0
\]
as $|m| \rightarrow \infty$.

For $N\in \mathbb{N}_{\geq 1}$, every vector in $\im(p_{N,\eta}^\perp)$
has the form $\sum_{l=0}^{N-1} V_\eta(\delta_l\otimes\zeta_l)$ for suitable elements $\zeta_0,\ldots, \zeta_{N-1}\in\ker(l_q^*(\eta))$, where $V_\eta$ is defined as in Lemma \ref{lem:oscillator}. An explicit calculation gives that for $0\leq l\leq N-1$,
\[
    p_{N,\eta}^\perp v^mV_\eta(\delta_l\otimes\zeta_l)
    =
    \sum_{k=0}^{N-1}
   c_{k,l}^{(m)}\, V_\eta(\delta_k\otimes\zeta_l), \quad c_{k,l}^{(m)}:=\frac{\langle\eta^{\otimes k},v^{m}\eta^{\otimes l}\rangle_{q}}{\sqrt{[k]_{q}!\,[l]_{q}!}}
\]
so that unitarity of $V_\eta$ and the Cauchy--Schwarz inequality imply
\begin{align*}
    \left\|
        p_{N,\eta}^\perp v^mV_\eta
        \left(\sum_{s=0}^{N-1}
            \delta_s\otimes\zeta_s\right)
    \right\|^2
    &=
    \sum_{r=0}^{N-1}
    \left\|\sum_{k=0}^{N-1}
        c_{r,k}^{(m)}\zeta_k\right\|^2\\
    &\leq
    \left(\sum_{r,k=0}^{N-1}|c_{r,k}^{(m)}|^2\right)
    \left(\sum_{k=0}^{N-1}\|\zeta_k\|^2\right).
\end{align*}
Hence,
\begin{equation} \label{CompressionLimit}
    \|p_{N,\eta}^\perp v^mp_{N,\eta}^\perp\|
    \leq
    \left(
        \sum_{r,k=0}^{N-1}|c_{r,k}^{(m)}|^2
    \right)^{1/2}
    \longrightarrow0
\end{equation}
as $|m|\to\infty$ because every coefficient tends to zero.

Now consider the set
\[
I:=\left\{(F,k) \mid k\in \mathbb{N}_{\geq 1} \text{ and } 1\in F\subseteq \mathcal A_q(\mathcal{H}_{\mathbb R}) \text{ finite}\right\},
\]
ordered by inclusion in the
first coordinate and the usual order in the second. For a fixed element $i=(F,k) \in I$, by Proposition~\ref{prop:two-estimates}(ii), we may choose an integer $n_i\geq k$ with $\|\mathbb E_\eta(u^{-n_i}au^{n_i})-\tau(a)1\|<k^{-1}$ for all $a\in F$.
Furthermore, Proposition~\ref{prop:two-estimates}(i)
and the strong convergence of the sequence $(p_{N,\eta})_{N \in \mathbb{N}_{\geq 1}}$ to zero provide an integer $N_i\geq k$ with
\begin{eqnarray*}
\bigl\| p_{N_i,\eta} \bigl(u^{-n_i}au^{n_i} -\mathbb E_\eta(u^{-n_i}au^{n_i})\bigr) p_{N_i,\eta} \bigr\|<\frac1k,\quad
    \|p_{N_i,\eta}u^{-n_i}a\Omega\|
    &<\frac1k
\end{eqnarray*}
for every $a\in F$.
Finally, the limit in \eqref{CompressionLimit} provides
$m_i\geq k$ with
$\|{p_{N_i,\eta}^\perp}v^{m_i r}{p_{N_i,\eta}^\perp}\|<k^{-1}$ for $0 < |r| \leq k$.

Define $ p_i:=u^{n_i}p_{N_i,\eta}u^{-n_i} \in \mathcal{B}(\mathcal{F}_q(\mathcal{H}))$, so that $\|p_iap_i-\tau(a)p_i\|<2k^{-1}$ and $\|p_ia\Omega\|<k^{-1}$ for all $a\in F$. Choose a cofinal ultrafilter $\mathcal U$ on $I$, and put $p:=(p_i\otimes1)_{i\to\mathcal U}\in\mathcal B\bigl((\mathcal F_q(\mathcal H)\otimes\ell^2)^{\mathcal U}\bigr)$.
Furthermore, define a unital $\ast$-homomorphism
\[
\theta\colon C(\mathbb T)\longrightarrow \Cs(s_q(\xi),s_q(\eta))^{\mathcal U}, \quad \theta(f):= \bigl(u^{n_i}f(v^{m_i})u^{-n_i}\bigr)_{i\to\mathcal U}.
\]
Since $\tau$ is tracial and every $v^{m_i}$ with $i \in I$ is a Haar unitary, we have that $\tau^{\mathcal U}(\theta(f))=\leb(f)$ for $f\in C(\mathbb T)$, implying that $\theta$ is injective.

Representing $\mathcal A_q(\mathcal H_{\mathbb R})^{\mathcal U}$ canonically on
$(\mathcal F_q(\mathcal H)\otimes\ell^2)^{\mathcal U}$, identifying
$\mathcal A_q(\mathcal H_{\mathbb R})$ with its diagonal copy, and tensoring the preceding estimates with $1$, we obtain that $pap=\tau(a)p$ and $pa(\Omega\otimes\delta_0)^{\mathcal U}=0$ for every $a\in \mathcal A_q(\mathcal H_{\mathbb R})$, that $p^\perp\theta(f)p^\perp=\leb(f)p^\perp$ for every $f\in C(\mathbb T)$, and that $[p,(1\otimes d)_{i\to\mathcal U}]=0$ for all $d\in \mathcal{B}(\ell^2)$. Thus the faithful Fock representation with vector
$\Omega$, the unit vector $\delta_0\in\ell^2$,
the projection $p$, and the state-preserving embedding
$\theta$ satisfy the conditions of
Theorem~\ref{AbstractCriterion-completely-selfless} (ii).
Hence the inclusion $\Cs(s_q(\xi),s_q(\eta))\subseteq(\mathcal A_q(\mathcal H_{\mathbb R}),\tau)$ is completely selfless.
\end{proof}

\vspace{3mm}


\bibliographystyle{mybstnum}
\bibliography{references}

\end{document}